\documentclass[11pt,a4paper,english]{article}

\usepackage[T1]{fontenc}
\usepackage{babel}
\usepackage{authblk}
\usepackage{amsmath,amssymb,amsfonts,amsthm,mathtools}
\usepackage{amsbsy}
\usepackage{graphicx,subcaption}
\usepackage{fancyhdr,makeidx,color,xcolor}

\usepackage{tikz}                
\usetikzlibrary{arrows.meta}     
\usetikzlibrary{arrows}          

\usepackage{pgfplots}            
\pgfplotsset{compat=1.12}

\usepackage{algorithm}
\usepackage{algpseudocode}
\usepackage{xspace,enumitem,listings,booktabs}
\usepackage{tabularx}

\usepackage{hyperref}
\usepackage{cleveref}

\makeatletter
\def\theHALG@line{\thealgorithm.\arabic{ALG@line}}
\makeatother

\numberwithin{equation}{subsection}

\numberwithin{algorithm}{subsection}

\numberwithin{figure}{subsection}
\numberwithin{table}{subsection}

\theoremstyle{plain}
\newtheorem{thm}{Theorem}[section]
\newtheorem{lem}[thm]{Lemma}
\newtheorem{prop}[thm]{Proposition}
\newtheorem{cor}[thm]{Corollary}

\theoremstyle{definition}

\newtheorem{exm}[thm]{Example}

\theoremstyle{remark}
\newtheorem{rmk}[thm]{Remark}

\crefname{thm}{theorem}{theorems}
\crefname{lem}{lemma}{lemmas}
\crefname{prop}{proposition}{propositions}
\crefname{cor}{corollary}{corollaries}
\crefname{definition}{definition}{definitions}

\definecolor{mygreen}{RGB}{28,172,0}
\definecolor{mylilas}{RGB}{170,55,241}

\numberwithin{equation}{subsection}

\newcommand{\beq}{\begin{equation}}
\newcommand{\eeq}{\end{equation}}
\newcommand{\bey}{\begin{eqnarray}}
\newcommand{\eey}{\end{eqnarray}}
\newcommand{\beyy}{\begin{eqnarray*}}
\newcommand{\eeyy}{\end{eqnarray*}}

\DeclareMathOperator*{\argmin}{arg\,min}

\newcommand{\abs}[1]{\left\vert#1\right\vert}
\newcommand{\dsb}[1]{[\![#1]\!]}
\newcommand{\norm}[1]{\left\Vert#1\right\Vert}
\newcommand{\seq}[1]{\left\langle#1\right\rangle}
\newcommand{\set}[1]{\left\{#1\right\}}
\newcommand{\rank}{\operatorname{rank}}
\newcommand{\diag}{\operatorname{diag}}

\newcommand{\tr}{\operatorname{tr}}

\newcommand{\vecc}{\operatorname{vec}}

\providecommand{\bld}{\boldsymbol}

\newcommand{\al}{\alpha}

\newcommand{\be}{\beta}

\newcommand{\la}{\lambda}

\newcommand{\lrarr}{\leftrightarrow}
\newcommand{\p}{\prime}
\newcommand{\pp}{\prime\prime}
\newcommand{\eps}{\epsilon}
\newcommand{\sig}{\sigma}
\newcommand{\Sig}{\Sigma}
\newcommand{\varp}{\varepsilon}
\newcommand{\bvarp}{\bld{\varp}}

\newcommand{\bfa}{\boldsymbol{a}}
\newcommand{\bfb}{\boldsymbol{b}}
\newcommand{\bfc}{\boldsymbol{c}}
\newcommand{\bfd}{\boldsymbol{d}}

\newcommand{\bfe}{\boldsymbol{e}}
\newcommand{\bfh}{\boldsymbol{h}}
\newcommand{\bfl}{\boldsymbol{l}}

\newcommand{\bfp}{\boldsymbol{p}}

\newcommand{\bfT}{\boldsymbol{T}}
\newcommand{\bfr}{\boldsymbol{r}}

\newcommand{\A}{\mathcal{A}}
\newcommand{\B}{\mathcal{B}}
\newcommand{\C}{\mathcal{C}}

\newcommand{\G}{\mathcal{G}}
\newcommand{\caI}{\mathcal{I}}

\newcommand{\caH}{\mathcal{H}}

\newcommand{\caK}{\mathcal{K}}

\newcommand{\caS}{\mathcal{S}}
\newcommand{\caT}{\mathcal{T}}
\newcommand{\TT}{\bld{\mathbb{T}}}

\newcommand{\RR}{\mathbb{R}}
\newcommand{\CC}{\mathbb{C}}

\newcommand{\II}{\mathbb{I}}

\newcommand{\PP}{\mathbb{P}}
\newcommand{\dbT}{\mathbb{T}}

\newcommand{\bu}{\mathbf{u}}
\newcommand{\bv}{\mathbf{v}}
\newcommand{\bw}{\mathbf{w}}
\newcommand{\bt}{\mathbf{t}}
\newcommand{\bx}{\mathbf{x}}
\newcommand{\bxp}{\bx^{\prime}}
\newcommand{\bxpp}{\bx^{\prime\prime}}
\newcommand{\bxi}{{\bx}^{(i)}}
\newcommand{\bxpi}{{\bxp}^{(i)}}
\newcommand{\bxppi}{{\bxpp}^{(i)}}
\newcommand{\by}{\mathbf{y}}

\newcommand{\bzero}{\mathbf{0}}
\newcommand{\bT}{\mathbf{T}}
\newcommand{\bX}{\mathbf{X}}
\newcommand{\bY}{\mathbf{Y}}

\newcommand{\bfF}{\mathbf{F}}
\newcommand{\bfH}{\mathbf{H}}
\newcommand{\bfI}{\mathbf{I}}
\newcommand{\bfJ}{\mathbf{J}}
\newcommand{\bfL}{\mathbf{L}}

\title{Trifocal Tensors in Three-View Geometry}

\author[1]{Yiran Xu\thanks{\href{mailto:yxu40@gsu.edu}{yxu40@gsu.edu}}}
\author[2]{Changqing Xu\thanks{Corresponding author, \href{mailto:cqxurichard@usts.edu.cn}{cqxurichard@usts.edu.cn}}}

\affil[1]{
Department of Mathematics and Statistics, Georgia State University, Atlanta, GA 30303, USA}
\affil[2]{School of Mathematical Sciences, Suzhou University of Science and Technology, Suzhou 215009, China}

\begin{document}

\maketitle
\thispagestyle{plain}

\begin{abstract}
Matrices and tensors are ubiquitous throughout computer vision. The trifocal tensor $\caT$ is a $3\times 3\times 3$ tensor that plays a vital role in 
three-view geometry. However, this tensor, though displayed in a third-order form, is manipulated and operated primarily in a matrix style in 
standard literature. This article establishes the conformal tensor algebra by replacing ad-hoc matrix collections with well-defined multilinear 
operators. We treat all three views symmetrically by covariant subscript indexing ($T_{ijk}$). To make precise what is new beyond notation: 
the classical point, line, and mixed correspondence constraints are here re-derived, but \emph{uniformly}, each as a single contraction 
in one shared algebra; the results that are genuinely new include the exact rank characterization of the trifocal tensor --- all mode-$k$ 
unfoldings $\caT[k]\in\RR^{3 \times 9}$ satisfy $\rank(\caT[k]) = 3$ for non-degenerate setups, and rank deficiency of any unfolding 
certifies degeneracy of the camera configuration --- which yields a near-zero-cost degeneracy alarm and a validation criterion for 
estimated tensors; the outer-product 
representation $\caT = A^{\top}\times \bfb_{4} - B^{\top}\times_{2} \bfa_{4}$; and direct epipole extraction via double contractive traces. 
The coordinate-free multilinear operators seamlessly map to tensor auto-differentiation frameworks (e.g., PyTorch, TensorFlow); a concrete 
PyTorch experiment demonstrates that the induced $24$-parameter bilinear parametrization acts as a structural prior that eliminates the 
accuracy drift exhibited by an unstructured $27$-entry autodiff refinement.
\end{abstract}



\section{Introduction}
\label{sec1:intro}

Projective geometry provides the mathematical foundation for understanding how three-dimensional scenes are imaged by cameras. In single-view 
geometry, a camera is modeled by a projection matrix $P\in\RR^{3\times4}$ mapping a 3D world point $X$ to an image point $\bx$ via 
$\bx\sim PX$, where $\bx\in\RR^{3}$ and $X \in \RR^{4}$ are expressed in homogeneous coordinates. However, this mapping is ambiguous up 
to a 3D projective transformation since depth cannot be recovered from a single view. This ambiguity is resolved in multi-view geometry through 
multiple images of the same scene.

Matrices and tensors serve as compact algebraic representations of geometric constraints, enabling efficient computation of reconstruction, motion 
estimation, and 3D interpretation. In two-view geometry, the fundamental matrix enforces epipolar constraints relating corresponding points via a 
bilinear equation. In three-view and four-view geometry, the trifocal and quadrifocal tensors are defined through third- and fourth-order tensors, 
respectively, enabling camera pose recovery and scene structure estimation without requiring explicit point correspondences. Multi-view geometry 
underpins fundamental computer vision tasks such as stereo matching, visual odometry, and structure-from-motion (SfM), with applications spanning 
robotics, autonomous driving, augmented reality, and medical imaging.

Tensors offer a coordinate-free formalism to represent multi-view relationships. While matrices (second-order tensors) handle two-view relationships 
well, high-order tensors are essential for multi-view scenarios. The intellectual lineage of multi-view geometry is deeply rooted in classical 
photogrammetry and projective geometry. In 1913, Kruppa \cite{Kruppa1913} established early foundations on relative pose recovery from five 
point correspondences. Then in the 1980s Longuet-Higgins \cite{LH1981} introduced the essential matrix $E$ and the 8-point algorithm. At almost the 
same time, Fischler and Bolles \cite{FB1981} developed RANSAC, enabling robust estimation in the presence of outliers.

In the 1990s, M. Spetsakis and Y. Aloimonos~\cite{Spetsakis1990} initialized the term trifocal tensor. They defined early trilinear relations 
specifically for calibrated cameras to describe the geometric relationships between line projections across three views. The concept was then 
generalized for uncalibrated cameras by R. Hartley~\cite{Hartley1995} and A. Shashua~\cite{Shashua1995}. Hartley coined the term trifocal 
tensor and framed it in its modern algebraic formulation to model point and line correspondences across three views without requiring prior 
camera calibration. In 1998, Papadopoulo and Faugeras \cite{PapaFaug1998} presented canonical camera-derived tensor expressions, and Avidan 
and Shashua~\cite{Avidan1997} provided a unified $N$-view tensorial perspective (see Chapters 15/16, \cite{HartZiss2004}).  

The trifocal tensor has played a pivotal role in three-view geometry, enabling camera pose estimation, 3D reconstruction, image-based rendering, and cross-view feature transfer. The classical formulations of trifocal tensors such as the block-matrix representations often rely on ad-hoc matrix 
unrollings, non-uniform indexing, or asymmetry across views, which obscure its underlying multilinear properties and form the primary challenge
to be addressed in this work. 

\subsection{Contributions and Organization}

Despite extensive research, existing literature in computer vision and multiview geometry often treats multi-index tensors purely via matrix 
reshappings, obscuring intrinsic tensorial properties. For instance, line correspondences across three views $\bfl \lrarr \bfl^\p \lrarr \bfl^{\pp}$ are 
often presented using matrix-slice notations like $[\bfT_1, \bfT_2, \bfT_3]$, which can lead to confusion.

In this work, we leverage tensor algebra, specifically extending the outer products and contraction products of tensors (matrices), and tensor 
decompositions, to systematically reformulate multi-view geometry. We present an updated, index-symmetric representation of the trifocal tensor 
using tensor outer and contractive products, resolve the structural asymmetries, establish clear full-rank properties for the matrix unfoldings of 
trifocal tensors, and unify point and line correspondence relations, precise slice decompositions, and direct epipole extraction through the novel 
definition of the trifocal tensor.

\medskip
\noindent\textbf{What is genuinely new.}
Because the trilinear point, line, and mixed correspondence constraints of three-view geometry are classical~\cite{HartZiss2004}, we state
explicitly what in this reformulation is new, as distinct from re-derivations in tensor notation. The correspondence relations themselves
(Theorems~\ref{thm4-3-3}--\ref{thm4-4-5:ppp}) are known results; our contribution there is that each now follows as a \emph{single contraction}
of one shared object, in one uniform algebra, rather than from separate matrix manipulations. The following elements are, to our knowledge,
genuinely new:
\begin{enumerate}[label=(\roman*),leftmargin=*]
\item \textbf{The outer-product representation} $\caT = A^{\top}\times \bfb_{4} - B^{\top}\times_{2} \bfa_{4}$ 
      (Theorem~\ref{thm4-1}), which builds the trifocal tensor directly from the two finite camera blocks and realizes the
      $24$-parameter bilinear parametrization of the tensor manifold;
\item \textbf{The exact rank characterization} of the trifocal tensor: all three mode-$k$ unfoldings satisfy
      $\rank(\caT[k])=3$ in every non-degenerate configuration, and rank deficiency of any unfolding certifies
      degeneracy (Theorem~\ref{thm4-2-1:trif-matrx}). Beyond its structural content, this yields a near-zero-cost
      degeneracy alarm (three $3\times 9$ SVDs), a validation criterion for estimated tensors, and a graded
      conditioning signal that anticipates the degradation observed on real data in Section~\ref{sec5-10:realdata};
\item \textbf{Direct epipole extraction} via double contractive traces with the identity matrix (Theorem~\ref{thm4-7:trace}),
      requiring no slice factorization;
\item \textbf{A concrete autodifferentiation pipeline} (Section~\ref{sec5-12:autodiff}): the coordinate-free contractions map
      one-to-one to \texttt{einsum} calls in PyTorch, and the induced $24$-parameter bilinear parametrization --- an immediate
      consequence of item (i) --- acts as a structural prior that eliminates the accuracy drift exhibited by an
      unstructured $27$-entry autodiff refinement, while preserving the accuracy of the linear solution.
\end{enumerate}

The remainder of this paper is organized as follows. Section \ref{sec2:preli} reviews the projective plane and the tensor-algebra toolbox used 
throughout, including tensor slicing, unfolding, fibers, the CP and Tucker decompositions, and generalized outer and contractive products. 
Section \ref{sec3:camera} examines camera matrix formulations and two-view constraints, detailing the pinhole camera $P\sim K[R \mid \bt]$, the 
essential matrix $E = [\bt]_\times R$, and the Direct Linear Transform analyzed through Kronecker products and tensor contractions. Section 
\ref{sec4:trifocal} formulates the trifocal tensor $\caT$ in a unified $3\times3\times3$ framework via the outer products $A^\top \times\bfb_4 - 
B^\top \times_2 \bfa_4$, proves that all unfoldings $\caT[k]$ have full row rank 3 in non-degenerate setups, derives the point/line correspondence 
relations within conformal equations, and demonstrates direct epipole extraction. Section \ref{sec5:numeric} presents estimation algorithms from 
line and point correspondences with MATLAB implementations and numerical experiments, together with applications, extensions, and limitations. 
Section \ref{sec6:conclusion} concludes the paper.


\section{Preliminaries}
\label{sec2:preli}
\setcounter{equation}{0}

Computer vision algorithms interpret visual data to replicate human visual capabilities. Core tasks include classification, object detection, 
segmentation, and 3D scene reconstruction. The real projective plane, denoted as $\PP^2$, is the core model representing camera image 
formation, where parallel lines in 3D space intersect at vanishing points on the horizon. 

Throughout this work, we denote by $[n]$ for the set $\set{1,2,\ldots, n}$. For $n=3$, we denote $\dbT$ for the set $[3]$ to avoid confusion.
We denote scalars by lowercase letters (e.g. $a,b,c,\dots$ or $\la, \mu, \dots, $ etc. ), vectors by boldface lowercase letters 
($\bfa,\bfb,\bfc,\dots, \bx, \by, \bu, \dots, $), matrices by uppercase letters ($A,B,C,\dots,$), and tensors (of order three or higher) by calligraphic 
uppercase letters ($\A,\B, \dots, \G, \caH, \dots, \caT,\dots$).  An $m$th-order tensor $\A$ of size $d_1\times d_2\times\cdots\times d_m$ is an
$m$-way array whose entries $A_{i_1 i_2 \ldots i_m}$ are indexed by $m$ indices; the integer $d_k$ is called the dimension (or mode size) of 
mode $k$.    

We need some basic knowledge on the projective geometry.  

\subsection{Some basic definitions in Projective Plane}
\label{sec2-1}
\setcounter{equation}{0}

The real projective plane $\PP^2$ is the space of all lines passing through the origin in $\RR^3$, i.e., 
$\PP^2 = (\RR^3 \setminus \{\bzero\}) / \sim$, where the equivalence relation $\sim$ identifies points satisfying $(x, y, w) \sim \la (x, y, w)$ 
for any non-zero scalar $\la \in \RR$. A point in $\PP^2$ is represented by \textbf{homogeneous coordinates} $\bx = [x, y, w]^{\top}\neq \bzero$. For any $0\neq \la\in\RR$, $[\la x, \la y, \la w]^{\top}$ represents the identical projective point. For $w\neq 0$, the homogeneous point 
$[x, y, w]^\top$ corresponds to the Euclidean point $(x/w, y/w)$, and $[x, y, 0]^\top$ represents a \emph{point at infinity} (ideal point), encoding 
directional orientation. A pixel point $(u, v)$ in an image plane can be expressed homogeneously as $[u, v, 1]^\top$. Two parallel lines intersect in 
$\PP^2$ at a point at infinity represented as a vector $[x,y,0]^{\top}$.  By adding points at infinity, $\PP^2$ extends the standard Euclidean plane  
so that any two distinct lines always intersect.  

A \emph{homography} is a nonsingular projective transformation $\bfH\in\RR^{3\times 3}$, defined up to an arbitrary nonzero scale factor, mapping 
points between two image planes. It arises when both cameras view the same 3D plane in the world, under pure camera rotation, or when the scene 
is distant so that points effectively lie on the plane at infinity. For corresponding points $\bx$ in the first image and $\bx^{\p}$ in the second, 
which are both projections of the same 3D point $\bX$ lying on a plane, the relationship is 
\[
\bx^{\p}\sim \bfH \bx
\]
where $\sim$ denotes equality up to scale. A homography matrix $\bfH$ provides a direct pixel-to-pixel mapping between two views for points on 
the associated plane, without requiring explicit 3D reconstruction. This makes it essential for applications such as image stitching, perspective 
correction, augmented reality, and camera calibration.  In practice, $\bfH$ is typically estimated from point correspondences, requiring a minimum 
of four matched point pairs, using techniques such as the Direct Linear Transform (DLT), often embedded within robust estimation frameworks like 
RANSAC to handle outliers.

To investigate multi-view geometry and its applications in computer vision, we need to review some fundamental concepts in tensor theory in this 
section, focusing on third-order tensors. These include tensor slicing, tensor unfolding (matricisation), and tensor fibers along various modes. We 
also introduce essential tensor decompositions, such as the CP decomposition and the Tucker decomposition, and extend the standard outer product and contractive product operations to more general forms.

\subsection{Preliminaries on Tensors and Tensor Products}
\label{sec2-2:tensor-tprod}
\setcounter{equation}{0}

We denote $\TT_{p}$ for the set of the $p$-order tensors, and $\TT_{\II}$ for the set of tensors with size $\II:=I_{1}\times \dots\times I_{m}$ 
where each $I_{i}$ is a positive integer, and $\TT_{m;n}$ for $\TT_{\II}$ when all $I_{i}$ are identical to a positive integer $n$. We write  
\[
S(n,m) := \set{(i_1, i_2, \dots, i_m)\colon  i_k\in [n], \forall k\in [m] }
\]
for the index set of components of a tensor in $\TT_{m;n}$.  An $m$-order tensor can be unfolded along any mode into an $(m-1)$-order tensor. 
For example, an $m\times n\times p$ tensor $\A$ can be unfolded along mode-3 into an $m\times np$ matrix, denoted $\A[3]$ ($\A[1]$ and 
$\A[2]$ can be defined analogously). 

Certain special types of matrices have natural extensions in tensor form. A \emph{zero tensor} is a tensor with all entries being zeros, and a 
\emph{diagonal element} of a tensor $\A\in\TT_{m;n}$ is an entry indexed as $A_{ii\ldots i}$. For an $k\in [m]$, a \emph{slice} of a tensor 
$\A\in\TT_{m}$ along mode $k$ is an $(m-1)$-order tensor obtained by fixing the $k$th index. For example, the $k$-slice along mode 3 of an 
$m\times n\times p$ tensor $\A$ is the matrix $A(:,:,k)\in\C^{m\times n}$ for some $k\in [p]$, and a slice of an 4-order tensor is a third order tensor.

Given $m$ vectors $\al_{1},\ldots, \al_{m}\in\CC^{n}$. The \emph{outer product} or \emph{tensor product} of  $\set{\al_{j}}_{j=1}^{m}$, 
denoted $\A:=\al_{1}\times\al_{2}\times\ldots\times\al_{m}$, is an $m$-order $n$-dimensional tensor, i.e., $\A\in\TT_{m;n}$ is a rank-1 tensor defined by 
\[ 
A_{i_{1}i_{2}\cdots i_{m}}=a_{i_1,1}a_{i_2,2}\ldots a_{i_m,m},\forall \boldsymbol{i}=(i_1,i_2,\ldots,i_m)\in S(m,n) 
\] 
where $\al_{j} = (a_{1j},a_{2j},\ldots, a_{nj})^{\top}$ for each $j\in [m]$. 
  
We can also define the outer product between any two tensors. Given tensors $\A\in\TT_{p}, \B\in\TT_{q}$. The outer product $\A\times \B$ 
is generally defined as 
\beq\label{eq2-2-1:outprod}
(\A\times\B)_{i_{1}\ldots i_{p}j_{1}\ldots j_{q}} = A_{i_{1}\ldots i_{p}}B_{j_{1}\ldots j_{q}}
\eeq
Thus $\A\times \B\in\TT_{p+q}$.  The outer product can be extended to any finite number of tensors of any sizes.  Let 
\[ 
\A_{i}\in\TT_{p_{i}}, \quad \forall i\in [r],  p = p_{1}+\dots +p_{r}.
\]
Then we may expect the outer-product  $\A_1\times \A_2\times \cdots \times \A_r$ to be a tensor of order $p$.  However, the allocations of 
modes may determine the properties of the outcome.  Thus we let $\pi:=\pi_{1}\cup\pi_{2}\cup \ldots \cup \pi_{r}$ be a partition of set $[p]$. 
We denote
\beq\label{eq2-2-2:prod-partit01} 
[\A_{1}, \A_{2}, \dots, \A_{r}]_{\pi} 
\eeq
for the outer product 
\beq\label{eq2-2-3:prod-partit02} 
\A = \A_{1}\times_{\pi_{1}}\A_{2}\times_{\pi_{2}}\ldots \A_{r-1}\times_{\pi_{r-1}}\A_{r} 
\eeq
where the modes indexed in $\pi_{k}$ are assigned to $\A_{k}$ for each $k\in [r]$. For example, if $A, B, \C$ are resp. tensors of order $2,2,4$ 
($A,B$ are matrices), and $\pi=\set{\set{1,5},\set{2,6},\set{3,4,7,8}}$.  Then the outer product $[A\times B\times \C]_{\pi}$ is an 8-order tensor 
with components defined by
\[ A_{i_1i_2\ldots i_8} = a_{i_1i_5}b_{i_2i_6}c_{i_3i_4i_7i_8} \] 
 For two matrices $A$ and $B$, we can assign $(1,3)$-modes to $A$ and $(2,4)$-modes to $B$ to produce the 4-order tensor $A\times_{c} B$,i.e.,  
\beq\label{eq2-2-4:crossprod}
(A\times_{c} B)_{i_{1}i_{2}i_{3}i_{4}} = A_{i_{1}i_{3}}B_{i_{2}i_{4}}
\eeq 
Similarly, if we assign $(1,4)$-modes to $A$ and $(2,3)$-modes to $B$, then we have the 4-order tensor $A\times_{ac} B$ defined as 
\beq\label{eq2-2-5:acrossprod}
(A\times_{ac} B)_{i_{1}i_{2}i_{3}i_{4}} = A_{i_{1}i_{4}}B_{i_{2}i_{3}}
\eeq

We can also define the \emph{contractive product} of $\A$ and $\B$ if there are some conformal modes between $\A$ and $\B$. Let 
$\A\in\TT_{p}$ and $\B\in\TT_{q}$.  If there exist subsets $S\subset [p], T\subset [q]$ such that  $\abs{S}=\abs{T}= r$ and the $S$-modes of 
$\A$ are compatible with the $T$-modes of  $\B$, we may assume that $S =\set{i_{1},\ldots, i_{r}}, T = \set{j_{1},\ldots, j_{r}}$, both ordered 
increasingly. The compatibility of $(\A, \B)$ along mode pairs $(S,T)$ means that 
\[ I_{i_{1}}= J_{j_{1}}, I_{i_{2}}= J_{j_{2}}, \dots, I_{i_{r}}= J_{j_{r}}, \]
where $\A$ and $\B$ are of sizes $\bfI:=I_{1}\times \ldots\times I_{p}$ and $\bfJ:=J_{1}\times \ldots\times J_{q}$ respectively.  The contractive 
product $\A\ast_{(S,T)} \B$ yields an $(p+q-2r)$-order tensor.  A special case is when 
\[ S=\set{p-r+1, p-r+2, \ldots, p}, T =[r], \] 
i.e., $S$ consists of the last $r$ modes of $\A$, and $T$ consists of the first $r$ modes of  $\B$.  Denote $\A\ast_{(S,T)}\B$ simply by 
$\A\ast_{[r]}\B$, which is defined as 
\beq\label{eq2-2-6:t-t-contract}
(\A\ast_{[r]}\B)_{i_1\ldots i_{p-r} j_{r+1}\ldots j_q} = \sum_{i_{p-r+1},\ldots, i_p} 
A_{i_1\ldots i_{p-r}i_{p-r+1}\ldots i_{p}} B_{i_{p-r+1}\ldots i_{p}j_{r+1}\ldots j_q}
\eeq 
Furthermore, if  $r = q\leq p$, we denote it by $\A\ast \B$.  When $r=p=q$ ($\bfI=\bfJ$), we have $\A\ast\B = \seq{\A,\B}$, i.e., inner product of 
tensors $\A$ and $\B$.  

For any two tensors $\A\in\TT_{p}$ and $\B\in\TT_{q}$, we denote $\A\ast_{S}$ for $\A\ast_{(S,T)} \B$ if $S=T^{c}\subset [p]\cap [q]$ and 
$\A$ and $\B$ are compatible on $S$.  For example, if $\A, \B\in\TT_{4;n}$, then $\A\ast_{[2]} \B\in\TT_{4;n}$ is an 4-order tensor defined as
\[
\left( \A\ast_{[2]} \B\right)_{ijkl} = \sum_{i^{\p},j^{\p}} A_{i^{\p} j^{\p}ij} B_{i^{\p} j^{\p}kl} 
\]
where $S=\set{3,4}, T=\set{1,2}$.  

An 4-order tensor $\caI_{n}\in\TT_{4;n}$ is called an \emph{identity tensor} if its components are defined by 
\[
I_{ijkl} = \delta_{ik} \delta_{jl},
\]  
where $\delta_{ij}$ is the Kronecker delta. We use $\caI$ to replace $\caI_{n}$ when there is no risk of confusion. By definition, we have 
$\caI_{4,n} = I_n \times_c I_n$, and for any matrix $A\in\RR^{n\times n}$, we have 
\[
\caI\ast A = A = A\ast \caI
\]

For a third-order tensor $\A\in\RR^{m\times n\times p}$, the $i$th \emph{horizontal slice} $\A(i,:,:)$ is an $n\times p$ matrix obtained by fixing the first index, the $j$th \emph{lateral slice} is an 
$m\times p$ matrix $\A(:,j,:)$ obtained by fixing the second index, and the $k$th \emph{frontal slice} is $\A(:,:,k)\in\RR^{m\times n}$ obtained 
by fixing the third index.  A \emph{fiber} is a one-dimensional section obtained by fixing all but one index. For instance, $\A(i,j,:)\in\RR^{p}$ is a mode-3 fiber, $\A(i,:,k)\in\RR^{n}$ is a mode-2 fiber, and $\A(:,j,k)\in\RR^{m}$ is a mode-1 fiber. 

The \emph{mode-$k$ unfolding} ( \emph{matricization} or \emph{flattening}) of a tensor $\A\in\RR^{d_1\times d_2\times\cdots\times d_m}$ 
arranges all its mode-$k$ fibers as columns of a matrix $\A[k]\in\RR^{d_k\times (d_1\cdots d_{k-1}d_{k+1}\cdots d_m)}$. The mapping is 
defined elementwise as
\[
\A[k](i_k, j) = \A(i_1,i_2,\dots,i_m),
\]
where $j = 1 + \sum_{\substack{t=1\\ t\neq k}}^{m} (i_t-1) \prod_{\substack{s=1\\ s\neq k}}^{t-1} d_s$. For a third-order tensor $\A$ of size $m\times n\times p$, the three unfoldings are matrices of sizes $m\times np$, $n\times mp$, and $p\times mn$, respectively.

For a matrix $M\in\RR^{m\times n}$ and a vector $\bu\in\RR^{p}$, the \emph{standard outer product} $M\times\bu\in\RR^{m\times n\times p}$ 
is defined componentwise by
\beq\label{eq2-2-13}
(M\times\bu)_{ijk} = M_{ij}\, u_k.
\eeq
More generally, for a tensor $\A\in\RR^{d_1\times\cdots\times d_m}$ and a vector $\bu\in\RR^{q}$, the outer product along mode $s$ yields a tensor of order $m+1$:
\beq\label{eq2-2-14}
(\A\times_s\bu)_{i_1\cdots i_m\, i_{m+1}} = \A_{i_{t_1} i_{t_2}\cdots i_{t_{m}}}\, u_{i_{s}},
\eeq
where $\set{t_1, t_2,\ldots, t_m} = [m+1]\setminus \set{s}$.  We denote $\A\times\bu:= \A\times_{m+1}\bu$ defined by 
\beq\label{eq2-2-9}
(\A\times\bu)_{i_1\cdots i_m\, i_{m+1}} = \A_{i_1 i_2\cdots i_m}\, u_{i_{m+1}}.
\eeq
The notation $\times_s$ indicates that vector $\bu$ occupies mode $s$.  For example, $M\times_2\bu$ is an $m\times p\times n$ tensor defined by 
\beq\label{eq2-2-10}
(M\times_2\bu)_{ijk} = M_{ik}\, u_j, \qquad  \forall (i,j,k). 
\eeq

The \emph{mode-$s$ contractive product} (or $s$-mode product) between a tensor $\A$ of size $d_1\times\cdots\times d_m$ and a vector 
$\bx\in\RR^{d_s}$ contracts along mode $s$ where $s\in [m]$, producing a tensor of order $m-1$:
\beq\label{eq2-2-11}
(\A\ast_s\bx)_{i_1\cdots i_{s-1}i_{s+1}\cdots i_m} = \sum_{t=1}^{d_s} \A_{i_1\cdots i_{s-1}\, t\, i_{s+1}\cdots i_m}\, x_t.
\eeq
For a third-order tensor $\A$ of size $m\times n\times p$ and a vector $\bx\in\RR^{d}$ ($d\in \set{m,n,p}$), the three contractive products are
\begin{align*}
(\A\ast_1\bx)_{jk} &= \sum_{t=1}^{m} \A_{tjk}\, x_t, \qquad \A\ast_1\bx \in \RR^{n\times p},\\
(\A\ast_2\bx)_{ik} &= \sum_{t=1}^{n} \A_{itk}\, x_t, \qquad \A\ast_2\bx \in \RR^{m\times p},\\
(\A\ast_3\bx)_{ij} &= \sum_{t=1}^{p} \A_{ijt}\, x_t, \qquad \A\ast_3\bx \in \RR^{m\times n}.
\end{align*}
These operations naturally generalize to contractions with matrices. For example, contracting a tensor with a matrix along a shared mode yields 
another tensor, with the contracted mode replaced by the second dimension of the matrix. Note that contracting a third order tensor with a matrix 
along two shared modes yields a vector.  

Let $\G$ be a tensor of size $d_1\times\cdots\times d_m$ and $U_i\in\RR^{d_{i}\times r_{i}}$ for all $i\in [m]$.  We denote 
\beq\label{eq2-2-12}
\G\dsb{U_1,\cdots,U_m} = \G\ast_{1} U_1\ast_{2} U_2 \dots \ast_{m} U_m
\eeq
Then it is easy to see that  $\B:=\G\dsb{U_1,\cdots,U_m}$ is an $m$th-order tensor of size $r_1\times\cdots\times r_m$. 
  
Two fundamental tensor decompositions are widely used in multilinear algebra and its applications. The \emph{CP decomposition} (CANDECOMP/PARAFAC) expresses a tensor as a sum of rank-one outer products:
\beq\label{eq2-2-15:cp}
\A = \sum_{r=1}^{R} \lambda_r \, \bu^{(1)}_r \circ \bu^{(2)}_r \circ \cdots \circ \bu^{(m)}_r,
\eeq
where $\circ$ denotes the vector outer product, and $\bu^{(k)}_r\in\RR^{d_k}$. The minimal $R$ achieving such a representation is called the \emph{CP rank} of $\A$.

Given an $m$th-order tensor $\A\in\TT_{m}$ of size $d_1\times \dots\times d_m$. The \emph{Tucker decomposition} (or called higher-order 
singular value decomposition, HOSVD) represents a tensor as a core tensor multiplied by factor matrices along each mode
\beq\label{eq2-2-16:tucker}
\A = \G\dsb{U_{1},\ldots,U_{m}}:\G\ast_1 U_1\ast_2 U_2\ast_3\cdots \ast_m U_m
\eeq
where $\G\in\RR^{r_1\times r_2\times\cdots\times r_m}$ is the core tensor, $U_k\in\RR^{d_k\times r_k}$ are factor matrices (typically with 
orthonormal columns), and $\ast_k$ denotes the mode-$k$ contractive product. 
The tuple $(r_1,r_2,\dots,r_m)$ satisfying $r_k = \rank(\A[k])$ is called the \emph{Tucker rank} of $\A$.  The Tucker decomposition was introduced by Tucker~\cite{Tucker1966} as a multidimensional extension of principal component analysis, and the orthogonal variant (HOSVD) was formalized by De Lathauwer, De Moor and Vandewalle~\cite{DeLathauwer2000}. Unlike the CP 
rank (minimal number of rank-1 outer products), the Tucker rank is always computable in polynomial time via singular value decompositions 
of the unfoldings, and it is the natural rank notion for the multilinear framework adopted in this section. 

The tensor theory provide the algebraic foundation for the coordinate-free formulation of the trifocal tensor developed in the following sections.

\subsection{The Outer Product and Contractive Product of Matrices and Vectors}
\label{sec2-3}
\setcounter{equation}{0}

For a matrix $X=(x_{ij})\in\CC^{m\times n}$ and a vector $\by=(y_{1},\ldots,y_{p})^{\top}\in\CC^{p}$, the outer products $X\times\by$, 
$X\times_{s}\by$ ($s=1,2$) are defined by \eqref{eq2-2-13}--\eqref{eq2-2-10}. In particular, 
\[
\left(X\times\by\right)_{ijk} = x_{ij}y_{k}, \qquad \left(X\times_{2}\by\right)_{ijk} = x_{ik}y_{j},
\]
so that $X\times\by$ and $X\times_{2}\by$ are of sizes $m\times n\times p$ and $m\times p\times n$, respectively. By symmetry one also uses 
$\by\times X$ for the mode-1 assignment, which is of size $p\times m\times n$. Note that the mode assignment in $X\times\by$ (and $\by\times X$) 
follows a natural order.

While outer products expand the order of tensors, contractive products shrink them. The mode-$n$ contractive product $\A\ast_{n}\bx$ of a 
tensor $\A\in\TT_{m}$ with a vector $\bx$ was defined in \eqref{eq2-2-11}. We use $\A\bx$ to denote $\A\ast_{m}\bx$ if $\A\in\TT_{m}$. For any tensor $\A\in\TT_{m;n}$ and a vector $\bx\in\RR^{n}$, we denote  
\[
\A\bx^{m}= \A\ast_{1}\bx\ast_{2}\bx\ldots \ast_{m}\bx.   
\]
$\A\bx^{m}$ can be regarded as the contractive product between $\A$ and $\bx^{m}$.  

Let $X\in\CC^{m\times p}$ and $Y\in\CC^{n\times p}$ be matrices with the same number of columns. Denote the $j$th column of $X$ and $Y$
respectively by $\al_j$ and $\be_j$. The \emph{Khatri-Rao product} $X\odot Y$, is an $mn\times r$ matrix, defined by
\beq\label{eq2-3-2:KR-prod}
X\odot Y\;:=\; \begin{bmatrix} \al_1\otimes \be_1 & \al_2\otimes \be_2 &\cdots & \al_p\otimes \be_p\end{bmatrix},
\eeq
where $\otimes$ denotes the Kronecker product. 

The following lemma involves the relationship between the outer product of matrices-vectors and the Kronecker product of matrices-vectors, which 
will be used in Section 4 to obtain the expression of the unfolding matrices of a trifocal tensor. 
\begin{lem}\label{le2-3-1:out-prod-vec}
Let $A=[\al_{1},\ldots,\al_{n}]\in\RR^{m\times n}, B=[\be_{1},\ldots,\be_{n}]\in\RR^{m\times n}$ and $\al\in\RR^{p}, \be\in\RR^{q}$. 
Then we have the following identities:
\begin{description}
\item[(1)]  $A\otimes \be^{\top} = [\al_{1}\times \be, \ldots, \al_{n}\times \be]$. 
\item[(2)]  $[\al\times \be_{1}, \ldots, \al\times \be_{n}] = \al\otimes \left[\vecc(B)\right]^{\top}$. 
\end{description}
\end{lem}
\begin{proof}
We first note the fact that for any vectors $\al, \be$
\beq\label{eq2-3-0}
\al\times \be = \al\otimes \be^{\top} = \al\be^{\top}
\eeq
Thus we have 
\beyy
A\otimes \be^{\top} 
&=& \left[ \al_1\otimes\be^{\top}, \al_2\otimes \be^{\top},\ldots, \al_n\otimes \be^{\top}\right] \\
&=& \left[ \al_1\times\be, \al_2\times \be,\ldots, \al_n\times \be\right] 
\eeyy
Thus (1) is proved.  Similarly we can prove (2). 
\end{proof}

If we denote by $\bfe^{(n)}\in\RR^{n}$ the all-ones vector, then (1) and (2) of Lemma \ref{le2-3-1:out-prod-vec} can be written equivalently 
as 
\[
A\times \be^{\top} = A\odot \left( \bfe^{(n)}\otimes\be\right)^{\top}
\]
and 
\[
\al\times\vecc(B) = \left( \al\times \bfe^{(n)}\right)\odot [\be_{1}^{\top}, \cdots, \be_{n}^{\top} ]
\]
where $X\odot Y$ is the Khatri-Rao product of two matrices $X$ and $Y$ defined by \eqref{eq2-3-2:KR-prod}.  

The following identities concern the combinations of outer products and contractive products of matrices and vectors:  

\begin{thm}\label{thm2-3-1}
Given any matrix $A=(A_{ij})\in\CC^{m\times n}$ and a vector $\al\in\CC^{p}$, then  
\begin{description}
\item[(1)] $(A\times \al)\be = (\al^{\top}\be) A$ if $\be\in\CC^{p}$. 
\item[(2)] $(A\times \al)\ast_{2}\be = (A\be)\times \al =A(\be\al^{\top})$ if $\be\in\CC^{n}$. 
\item[(3)] $(A\times \al)\ast_{1}\be = (A^{\top}\be)\times \al$ if $\be\in\CC^{m}$. 
\item[(4)] $(A\times_{2}\al)\be = (A\be)\times \al$ if $\be\in\CC^{n}$. 
\item[(5)] $(A\times_{2}\al)\ast_{2}\be = (\al^{\top}\be) A$ if $\be\in\CC^{p}$. 
\item[(6)] $(A\times_{2}\al)\ast_{1}\be = \al\times (A^{\top}\be) =(\al\be^{\top})A$ where $\be\in\CC^{m}$. 
\end{description}
\end{thm}

\begin{proof}
To prove (1), we note that $A\times \al$ is a tensor of size $m\times n\times p$ and thus the left-hand side of the formula in (1) is an $m\times n$ matrix, which is conformal to that on the right-hand side. Furthermore, for any index $(i,j)\in [m]\times [n]$, we have 
\beq
\left[(A\times \al)\be\right]_{ij}= \sum_{k=1}^{p} (A\times \al)_{ijk} \beta_{k}=\sum_{k=1}^{p} A_{ij}\alpha_{k} \beta_{k}= A_{ij} \sum_{k=1}^{p} \alpha_{k} \beta_{k}= (\al^{\top}\be) A_{ij}.
\eeq  
Thus (1) holds.  

To prove (2), we note that both sides of (2) are $m\times p$ matrices, and for any index $(i,j)\in [m]\times [p]$, we have 
\beq
\left[(A\times\al)\ast_{2}\be\right]_{ij} =\sum_{k=1}^{n} (A\times \al)_{ikj} \beta_{k}= \sum_{k=1}^{n} A_{ik}\alpha_{j} \beta_{k} 
=(A\be)_{i} \alpha_{j} = \left[ (A\be)\times \al\right]_{ij}.  
\eeq  
Thus (2) is proved. 

The remaining items (3)--(6) follow from the same elementwise computations; for example, in (5) one contracts 
$(A\times_{2}\al)_{ijk} = A_{ik}\alpha_{j}$ over $j$ to obtain $(\al^{\top}\be) A_{ik}$, and (6) follows analogously. We omit the details. 
\end{proof}  

Recall that for any matrices $A, B, C\in\RR^{n\times n}$, we have $A\times B, A\times_{c} B\in\TT_{4;n}$, and the matrix-tensor product 
$C\A$ with $\A\in\TT_{4;n}$ is the 4th-order tensor with components defined by 
\[
\left[C\A\right]_{ijkl} = \sum_{i^{\prime}=1}^{n} C_{ii^{\prime}}A_{i^{\prime}jkl} 
\]
and $\A C$ is the tensor whose components are defined by 
\[
\left[\A C\right]_{ijkl} = \sum_{l^{\prime}=1}^{n} A_{ijkl^{\prime}}C_{l^{\prime}l}. 
\]
The results in the following lemma will also be employed in our next discussion.

\begin{lem}\label{lem2-3-2}
Given any matrices $A, B\in\CC^{n\times n}$ and a vector $\bfc\in\CC^{n}$, we have  
\begin{description}
\item[(1)] $(AB)\times \bfc = A (B\times \bfc)$.
\item[(2)] $(AB)\times_{2}\bfc = A (B\times_{2}\bfc)$. 
\item[(3)] $A\times (B\bfc) = (A\times B)\bfc$.
\item[(4)] $A\times_{2} (B\bfc) = (A\times_{c} B)\bfc$. 
\end{description}
\end{lem}

\begin{proof}
To prove (1), we note that the tensors on both sides of (1) are third-order tensors of size $n\times n\times n$. Furthermore, for any index
$(i,j,k)\in [n]\times [n]\times [n]$, we have 
\beyy
\left[(AB)\times \bfc\right]_{ijk} 
&=&(AB)_{ij} c_{k} \\
&=& \sum_{i^{\prime}=1}^{n} A_{ii^{\prime}}B_{i^{\prime}j} c_{k} \\  
&=& \sum_{i^{\prime}=1}^{n} A_{ii^{\prime}}(B\times \bfc)_{i^{\prime}jk} \\
&=& \left[ A (B\times \bfc) \right]_{ijk}.  
\eeyy  
Thus (1) holds. Note that (2), (3), and (4) can also be proved using similar arguments.  
\end{proof}

The results in the following lemma will be used in Section 4 to simplify our algorithms.  
\begin{lem}\label{lem2-3-3}
Let $\G\in\TT_{m;n}(m\ge 3), A,B\in\CC^{n\times n}$ and a vector $\bfc\in\CC^{n}$. we have  
\begin{description}
\item[(1)] $\G\ast_{k} A\ast_{k} \bfc = \G\ast_{k} (A\bfc)$.
\item[(2)] $\G\ast_{k} A\ast_{k} B = \G\ast_{k} (AB)$.
\item[(3)] $\G\ast_{i} A\ast_{j} B = \G\ast_{j} B\ast_{i} A$ for any distinct $i,j\in [m]$.
\end{description}
\end{lem}

\begin{proof}
To prove (1), we first note that the tensors on both sides of (1) are in $\TT_{m-1;n}$. For convenience, we consider $k=m$. Then for every 
index $(i_1,i_2,\ldots, i_{m-1})\in S(n,m-1)$, we have  
\beyy
G_{i_1i_2\ldots i_{m-1}}  
&=&(\G\ast_{m} A)\ast_{m} \bfc\\
&=& \sum_{i_{m}=1}^{n} (\G\ast_{m} A)_{i_1i_2\ldots i_{m-1}i_{m}} c_{i_{m}} \\  
&=& \sum_{i_{m}=1}^{n} \left(\sum_{j=1}^{n} G_{i_1i_2\ldots i_{m-1}j}A_{j i_{m}}\right) c_{i_{m}} \\ 
&=& \sum_{j=1}^{n} G_{i_1i_2\ldots i_{m-1}j} \left(\sum_{i_{m}=1}^{n} A_{j i_{m}} c_{i_{m}}\right) \\ 
&=& \sum_{j=1}^{n} G_{i_1i_2\ldots i_{m-1}j} (A\bfc)_{j} \\ 
&=& \left[ \G\ast_{m} (A\bfc) \right]_{i_1i_2\ldots i_{m-1}}.  
\eeyy  
Thus (1) holds for $k=m$. Similarly we can show (1) for other $k\in [m]$.    

For (2), we can use similar argument as in the proof of (1), and thus we omit its proof, and (3) can be found in \cite{Kolda2009}.
\end{proof}

\subsection{The Alternating tensors}
\label{sec2-4}
\setcounter{equation}{0}

We recall the Levi-Civita symbol (permutation / alternating tensor), denoted $\bvarp = (\varp_{ijk}) \in\TT_{3;3}$, which is defined by  
\beq\label{eq2-4-1}
\varp_{ijk} = \begin{cases}
0, & \text{if any two indices are equal}, \\[4pt]
+1, & \text{if } (i,j,k) \text{ is an even permutation of }(1,2,3), \\[4pt]
-1, & \text{if }(i,j,k)\text{ is an odd permutation of }(1,2,3).
\end{cases}
\eeq
Thus the only non-zero components of $\bvarp$ are  
\beq\label{eq2-4-2}
\varp_{123} = \varp_{231} = \varp_{312} = +1, \qquad \varp_{132} = \varp_{213} = \varp_{321} = -1.
\eeq

We recall the Kronecker delta $\delta=(\delta_{ij})\in\RR^{n\times n}$ is a $(0,1)$-matrix where $\delta_{ij}=1$ if and only if $i=j$, and 
$\delta_{ij}=0$ if otherwise.  The alternating tensor is closely related to the $3\times 3$ Kronecker delta. In the following, we use $\delta$
to denote the $3\times3$ Kronecker delta. 

\begin{lem}\label{le2-4-2}
\begin{align}
\bvarp\ast_{[1]}\bvarp &= \delta\times_{c} \delta - \delta\times_{ac} \delta. \tag{I}\\
\bvarp\ast_{[2]} \bvarp&= 2\,\delta,   \tag{I\!I}\\
\bvarp\ast_{[3]}\bvarp &= 6. \tag{I\!I\!I}
\end{align}
\end{lem}
\begin{proof}
To prove (I), we denote the right hand side of (I) by $\A$ and the left hand side by $\B$. Then $\A,\B\in \TT_{4;3}$. By definition, we have
\beq\label{eq2-4-3}
A_{ijkl} = \delta_{ik} \delta_{jl} - \delta_{il} \delta_{jk}
\eeq
Thus we consider the following cases: 
\indent\textbf{Case 1}: $i=j$ or $k=l$.  Then $A_{ijkl}=0$ which is identical to $B_{ijkl}$.  \par
\indent\textbf{Case 2}: $i=k, j=l, i\neq j$. Then $A_{ijij}=1$, and 
\[
B_{ijij} = \sum_{i^{\p}=1}^{3} \varp_{i^{\p}ij} \varp_{i^{\p}ij} = 1.
\]
Thus we also have $A_{ijij}=B_{ijij}$ for all $i\neq j$.  \par 
\indent\textbf{Case 3}: $i=l, j=k$ and $i\neq j$.  Then $A_{ijji}=-1$, and $B_{ijji} = -1$. \par
The rest cases can also be checked. 
Similarly we can verify the identity (I) and (I\!I) the same way. 
\end{proof}   

The following lemma presents some other properties related to $\bvarp$. 
\begin{lem}\label{lem2-4-3}
For any vectors $\bfa, \bfb, \bfc\in\RR^{3}$ and $A=[\bfa,\bfb, \bfc]\in\CC^{3\times3}$, we have 
\begin{description}
\item[(1)] $\bfa\dot{\times}\bfb  = \bvarp\ast_{2}\,\bfa\ast_{3}\bfb$ where $\dot{\times}$ denotes the cross product. 
\item[(2)] $\bfa\cdot(\bfb\times\bfc)  = \bvarp\ast_{1}\,\bfa \ast_{2}\,\bfb\ast_{3}\,\bfc$ where $\cdot$ denotes the dot product.
\item[(3)] $\det(A) = \bvarp\ast_{1}\,\al_{1}\ast_{2}\al_{2}\ast_{3}\al_{3}$, where $\al_i$ is the $i$th row (or column) of $A$. 
\end{description}
\end{lem}

We note that the conventional Levi-Civita symbol $\bvarp$ can be generalized to be an $n$th-order tensor in $\TT_{n;n}$, which is defined as
\beq\label{eq2-4-4}
\varp_{i_1 i_2 \ldots i_n} = \begin{cases}
0, & \text{if any two indices are equal}, \\[4pt]
+1, & \text{if } (i_1 i_2 \ldots i_n) \text{ is an even permutation of }(1,2,\ldots,n), \\[4pt]
-1, & \text{if }(i_1 i_2 \ldots i_n)\text{ is an odd permutation of }(1,2,\ldots,n).
\end{cases}
\eeq
Thus $\bvarp\in\TT_{n;n}$ is a $(0,1,-1)$ tensor, i.e., each component of $\bvarp$ takes value in $\set{0,1,-1}$, and there are 
$n_{1}:=\frac{1}{2} n!$ components being 1, and the same number of components being $-1$.  Thus there are 
$n_{0} = n^{n} -n! $ number of zero components  in $\bvarp$.  

An immediate result concerning the determinant of an $n\times n$ matrix  $A$ is 
\begin{prop}\label{p2-4-5}
Let $\bfa_{1}, \bfa_{2}, \dots, \bfa_{n}\in\CC^{n}$ and $A=[\bfa_{1}, \bfa_{2}, \dots, \bfa_{n}]$ be an $n\times n$ matrix. Then we have 
\beq\label{eq2-4-5}
 \det A = \bvarp\dsb{\bfa_{1}, \ldots, \bfa_{n}}
\eeq
\end{prop} 

Let $\bfa = (a_1, a_2, a_3)^\top \in \RR^3$. The skew-symmetric matrix $[\bfa]_{\times}$ is defined such that for any vector $\bfb\in \RR^3$,
\[
[\bfa]_\times \, \bfb = \bfa \times \bfb.
\]
Explicitly,
\[
[\bfa]_\times =
\begin{pmatrix}
0 & -a_3 & a_2 \\
a_3 & 0 & -a_1 \\
-a_2 & a_1 & 0
\end{pmatrix}.
\]

The following lemma presents an expression for matrix $[\bfa]_{\times}$  in terms of $\bvarp$, which can be verified directly. 
\begin{lem}\label{le2-4-6}
Let $\bfa = (a_1, a_2, a_3)^\top \in \RR^3$.  Then the skew-symmetric matrix $[\bfa]_{\times}$ can be represented as
\beq\label{eq2-4-7}
[\bfa]_{\times} = - \bvarp \ast \bfa
\eeq
where the contractive product $\bvarp\ast \bfa = \bvarp\ast_{3} \bfa$. 
\end{lem}

While the Kronecker delta alone cannot directly represent the full skew-symmetric matrix, it appears in identities involving products of two Levi-Civita symbols. For example, the relation

\begin{lem}\label{le2-4-7}
Let $\bfa, \bfb\in \RR^3$.  Then 
\beq\label{eq2-4-8}
 [\bfa]_{\times}[\bfb]_{\times} = \bfb\bfa^\top - (\bfa^\top\bfb)\, I_{3}
\eeq
\end{lem}

\begin{proof}
We let  $A= [\bfa]_{\times}[\bfb]_{\times}$ and use Lemma \ref{le2-4-6}, we have 
\beq\label{eq2-4-9}
A = (\bvarp\ast \bfa) (\bvarp\ast \bfb) = \left[ \bvarp\ast_{(2,1)} \bvarp\right] \ast_{2} \bfa\ast_{4} \bfb
\eeq
Note that  
\beq\label{eq2-4-10}
\bvarp\ast_{(2,1)} \bvarp = - \bvarp\ast_{(1,1)} \bvarp = \delta\times_{ac} \delta - \delta\times_{c} \delta  
\eeq
The last equality of equation \eqref{eq2-4-10} comes from (I) of Lemma \ref{le2-4-2}.  Thus we have 
\beyy
A  &=& \left[\delta\times_{ac} \delta - \delta\times_{c} \delta \right]\ast_{2} \bfa\ast_{4} \bfb   \\
     &=&  (\delta \bfb) \bfa^{\top}  - (\bfa^{\top} \delta \bfb) \delta  \\
      &=&  \bfb\bfa^\top - (\bfa^\top\bfb)\, I_{3}
\eeyy
Here we notice that $\delta = I_{3}$. 
\end{proof}


\section{Camera Matrix and Essential Matrix}
\label{sec3:camera}
\setcounter{equation}{0}

Linear mappings and multi-linear constraints capture the mapping between 3D world coordinates and 2D image coordinates. 
The \textbf{camera matrix}, as a single $3\times 4$ linear mapping, models the imaging process through intrinsic parameters and extrinsic 
rotation and translation, describing the projection of scene points onto the image plane. 

\subsection{Basic Properties of Camera Matrices}
\label{sec3-1:camera}
\setcounter{equation}{0}

In projective geometry, a pinhole camera is represented by a $3\times4$ matrix acting on world points expressed in homogeneous coordinates. 
Denote $\bX = (X,Y,Z,1)^{\top}$. Let $\bx = (u,v,1)^{\top}$ be the image point of $\bX$ satisfying $\bx\sim P\,\bX$ with matrix 
$P\in\RR^{3\times4}$ defined by 
\beq\label{eq3-1}
P = K\,[R\mid \bt]
\eeq
where $K\in\RR^{3\times3}$ is an upper-triangular nonsingular calibration matrix, $R\in\RR^{3\times3}$ is an orthogonal rotation matrix, 
and $\bt\in\RR^3$ is a translation vector. It is easy to see that $\rank(P) = 3$ and $P$ is defined only up to scale, i.e., $P \sim \la P$. So 
$P$ has 11 degrees of freedom (DoF): five from the intrinsic matrix $K$ (focal lengths $f_x, f_y$, skew $s$, and principal point $c_x, c_y$),  
three from the rotation $R$ (three Euler angles or axis-angle representation), and three from the translation $\bt$ (3D offset vector). The total is not 
12 because the overall scale of $P$ is arbitrary in projective geometry. Thus, one degree of freedom is removed due to this scale ambiguity.

The essential matrix \(E\) is a \(3\times3\) matrix that encodes the geometric relationship between two calibrated cameras. Unlike the fundamental matrix \(F\), which works with uncalibrated pixel coordinates, the essential matrix operates in normalized camera coordinates, assuming known intrinsic parameters.

Given a viewpoint pair $(\bx_l, \bx_r)$ corresponding to a scene point \(\bX\) where $\bx_{l}$ and $\bx_{r}$ are located in the left 
and right image planes respectively with normalized image coordinates (i.e., removing the effect of camera intrinsics), the essential matrix $E$ satisfies
\beq\label{eq3-2}
\bx_r^{\top} E \,\bx_l = 0,
\eeq
where $\bx_l,\bx_r\in\RR^3$ are in homogeneous coordinates. Let $K_l, K_r$ be the calibration matrices of the left and right cameras. They 
are usually in upper-triangular form as 
\[
K = \begin{pmatrix}
f_x & s   & c_x \\
0   & f_y & c_y \\
0   & 0   & 1
\end{pmatrix}
\]
where $(f_x,f_y)$ is the pair of focal lengths in pixels (i.e., $f_x = f m_x, f_y = f m_y$), $s$ is the skew coefficient (usually assumed zero 
for modern cameras), and $(c_x, c_y)$ is the principal point (optical center) in pixel coordinates. The essential matrix $E$ is related to the 
fundamental matrix $F$ via
\beq\label{eq3-3}
E = K_r^{\top} F \, K_l.
\eeq
Conversely, \(F = K_r^{-\top} E \, K_l^{-1}\). Several key facts hold for the essential matrix $E$: $\rank(E) = 2$, and $E$ has two equal non-zero singular values ($\sig_{1} = \sig_{2} > 0$). $E$ encodes both rotation $R$ and translation \(\bt\) between the two cameras via $E = [\bt]_{\times} R$, where \([\bt]_{\times}\) is the skew-symmetric matrix of the translation vector.  

Given \(E\), the relative pose \((R,\bt)\) can be recovered up to scale with a four-fold ambiguity. Using the SVD $E= U\Sig V^{\top}$, the possible rotations and translations are 
\[
R_1= U R_z(+\tfrac{\pi}{2}) V^{\top}, \quad R_2 = U R_z(-\tfrac{\pi}{2}) V^{\top}, \quad [\bt]_{\times} = U R_z(\pm\tfrac{\pi}{2}) \Sig U^{\top},
\]
with \(R_z(\theta)\) representing rotation about the \(z\)-axis. Only one of the four combinations yields points in front of both cameras.

The essential matrix is central to structure-from-motion pipelines. After estimating \(E\) from point correspondences (e.g., via the eight-point algorithm adapted for normalized coordinates), its decomposition provides the initial camera poses, enabling triangulation of 3D points.

\subsection{Direct Linear Transformation (DLT)}
\label{sec3-2}
\setcounter{equation}{0}

Given correspondences $\{ \bX_j\lrarr \bx_j\}, \forall j\in [n]$ for $n\ge 6$, and let 
\beq\label{eq3-2-1: dlt_corrsp_pts}
X = [\bX_{1},\ldots, \bX_{n}], \qquad Y = [\bx_{1},\ldots, \bx_{n}]
\eeq
be the matrices storing the 3D world points $\bX_j = [X_j, Y_j, Z_j, 1]^\top$ and 2D image coordinates $\bx_j = [u_j, v_j, 1]^\top$, respectively.
Then we have $\bx_j \sim P \bX_j$ for all $j \in [n]$, where $P \in\RR^{3 \times 4}$ is the camera (projection) matrix.  The standard Direct Linear 
Transform (DLT) estimates $P$ by solving a homogeneous linear system 
\beq\label{eq: linearequ-dlt}
A \bfp = \mathbf{0},
\eeq
where $\bfp = \vecc(P^\top) \in \RR^{12}$ is formed by stacking the row vectors $\al_i^\top$ of $P$, so $A \in \RR^{2n \times 12}$ has the form
\beq\label{eq3-2-2:dlt_stacked_matrix}
A = \begin{bmatrix}
\bX_{1}^{\top} & \mathbf{0}^\top & -u_{1}\bX_{1}^{\top} \\
\mathbf{0}^\top & \bX_{1}^{\top} & -v_{1}\bX_{1}^{\top} \\
\vdots & \vdots & \vdots \\
\bX_{n}^{\top} & \mathbf{0}^\top & -u_{n}\bX_{n}^{\top} \\
\mathbf{0}^\top & \bX_{n}^{\top} & -v_{n}\bX_{n}^{\top}
\end{bmatrix}.
\eeq
$A^{\top} = [A_1^\top, A_2^\top, \dots, A_n^\top]$ is a block matrix formed via Khatri-Rao products, where 
\beq\label{eq:block_khatri_rao}
A_j = \begin{bmatrix} 1 & 0 & -u_j \\ 0 & 1 & -v_j \end{bmatrix} \otimes \bX_j^\top.
\eeq
is an $2\times12$ matrix associated with the $j$th correspondence $(\bX_{j}, \bx_{j})$ for $j\in [n]$.  We note that  
\[
\bx_j \dot{\times} (P \bX_j) = \mathbf{0}_{3\times 1}, \qquad  \forall j\in [n]
\] 
since $\bx_j$ and $P \bX_j$ are collinear.  It follows that 
\beq\label{eq3-2-3:cross_prod}
[\bx_j]  P \bX_j = \mathbf{0}, \quad \forall j \in [n]
\eeq
where $[\al]:=[\al]_\times$ is the standard skew-symmetric matrix associated with $\al\in\RR^{3}$ (see Section~\ref{sec2-4}).  \eqref{eq3-2-3:cross_prod} yields three equations per point, among which two are linearly independent.  By vectorization and the identity 
\[
\vecc(ABC) = (C^{\top}\otimes A) \vecc(B),
\]
the $2n$ homogeneous linear constraints across all points can be written compactly as
\beq\label{eq3-2-4:kron_sum}
\left( \sum_{j=1}^n (\bfe_j \bfe_j^\top X^\top) \otimes [\bx_j]_\times \right) \vecc(P) = \mathbf{0},
\eeq
where $\bfe_j\in \RR^n$ denotes the $j$th standard canonical basis vector. In tensor notation, the collinearity constraint across all $j\in [n]$ can 
be expressed via the Levi-Civita tensor $\eps\in\RR^{3\times3\times3}$
\beq\label{eq:t_contract_f1}
\sum_{j^{\p}=1}^3 \sum_{j^{\p\p}=1}^4 \epsilon_{i j^{\p} k^{\p}} \, Y_{k^{\p}, j} \, P_{j^{\p}, j^{\p\p}} \, X_{j^{\p\p}, j} = 0
\eeq
for $i\in\dbT, j\in [n]$, where $X_{j^{\p\p}j}$ and $Y_{k^{\p} j}$ are respectively the entries of $X$ and $Y$ defined by \eqref{eq3-2-1: dlt_corrsp_pts}. 
Equivalently, \eqref{eq:t_contract_f1} can be written in tensor-contractive form as 
\beq\label{eq:t_contract_f2}
\bigl(\varp\ast_{2} \bY\ast_{3} P\bigr)\ast \bX = 0
\eeq

Note that equation \eqref{eq: linearequ-dlt} may be overdetermined or $A$ may be full-rank (i.e., $\rank(A)=12$) due to noisy data, meaning 
it may yield only the trivial solution $\bfp=\mathbf{0}$. In this case, we solve the optimization problem
\beq\label{eq: unit-p}
\hat{\bfp}=\argmin_{\norm{\bfp} = 1} \norm{A\bfp}^2.
\eeq
The solution $\hat{\bfp}$ to \eqref{eq: unit-p} is the right singular vector of $A$ corresponding to the smallest singular value, obtained from the SVD $A = UDV^{\top}$ as $\hat{\bfp} = \bv_n$, the last column of $V$. (One may alternatively minimize the Rayleigh quotient $\bfp^{\top}A^{\top}A\bfp$, i.e., take the smallest eigenvector of $A^{\top}A$; however, SVD is numerically preferable since squaring $A$ doubles its condition number, $\kappa(A^{\top} A) = \kappa(A)^2$.)

Let $\bfp = \vecc(P^{\top})$ be the solution to \eqref{eq: linearequ-dlt} with $A$ defined as \eqref{eq3-2-2:dlt_stacked_matrix}. The matrix $P$ is obtained after reshaping $\bfp$ to $3\times 4$. Note
that $P$ should satisfy the condition $\det P(:,[1:3]) \neq 0$ since $P = K[R\mid \bt]$ implies $P(:,[1:3]) = KR$, where both $K$ and $R$ are 
nonsingular. For $n=6$, $A$ is a $12\times 12$ square matrix. 

The following algorithm can be used to estimate the camera matrix $P$, given $n$ pairs of correspondences for $n\ge 6$.

\begin{algorithm}
\caption{Direct Linear Transform (DLT) for Camera Matrix Estimation}
\label{alg:dlt-formal}
\begin{algorithmic}[1]
\Require $n\ge 6$ pairs of 3D-to-2D point correspondences $\{(X_i, \bx_i)\}_{i=1}^n$.
\Ensure A projective camera matrix $P \in \RR^{3\times 4}$.
\State Construct the $2n\times 12$ matrix $A$ from the correspondences by \eqref{eq3-2-2:dlt_stacked_matrix}.
\State Compute the SVD of $A$: $[U, S, V] \gets \mathrm{svd}(A)$.
\State Let $\bfp \gets V(: , 12)$, the last column of $V$.
\State Reshape vector $\bfp\in\RR^{12}$ into an $3\times 4$ matrix $P$.
\end{algorithmic}
\end{algorithm}

\begin{exm}\label{exm3-2-1}
In our numerical experiments with the implementation of Algorithm \ref{alg:dlt-formal}, using exactly the minimal number of six corresponding point pairs can render the reconstruction meaningless: for one synthetic configuration the relative Frobenius error $E_{re}:= \norm{P-P_0}_{F}/\norm{P_0}_{F}$ reached $1.0593$. Supplying eight pairs of correspondences for the same configuration reduced the error to $0.6567$. This is a general phenomenon: with noisy data the minimal sample is unstable, and one should supply more corresponding pairs than six in order to retrieve an accurate camera matrix.
\end{exm}

A single camera matrix $P\in\RR^{3\times4}$ captures the projection from 3D to 2D, but leaves a 3D projective ambiguity in the sense that $P$
and $P\,H$ produce identical images for any nonsingular $4\times4$ matrix $H$. The ambiguity can be partially resolved by choosing convenient 
reference frames when equipped with two or more views.  

In the next section, we exploit this freedom to reduce the three-camera configuration to a canonical form, simplifying the algebraic relations between views and enabling the derivation of the trifocal tensor.

\subsection{Affine Camera Matrices}
\label{sec3-3}
\setcounter{equation}{0} 

An \textbf{affine camera} is a linear model of image formation that assumes parallel projection, as opposed 
to perspective projection in a pinhole camera, with its center placed at infinity. Its canonical form (possibly up to a scalar multiple) is given by
\beq\label{eq3-3-3}
P_a = \begin{bmatrix}
M & \bt \\ \mathbf{0}^\top & 1\end{bmatrix}, \qquad  
M = \begin{bmatrix} m_{11} & m_{12} & m_{13} \\  m_{21} & m_{22} & m_{23}\end{bmatrix} \in\RR^{2\times 3}, \quad \bt=(t_1,t_2)^{\top}\in\RR^{2},
\eeq
so that the optical center is at infinity, and parallel lines in the scene project to parallel lines in the image. Points at infinity are mapped to points at 
infinity, preserving the affine structure. If we denote $\bX = (X,Y,Z,1)^{\top}$ and $\bx = [x,y,1]^{\top}$, then by \eqref{eq3-3-3} we have 
\[
 x = m_{11}X + m_{12}Y + m_{13}Z + t_1, \quad  y = m_{21}X + m_{22}Y + m_{23}Z + t_2.
\]

In a three-view setup, if all cameras are affine, the resulting geometric relations are encoded by a specialized \emph{affine trifocal tensor}. This 
tensor has a simpler structure (12 DoFs) compared to its general projective counterpart (18 DoFs), reflecting the constraints imposed by the affine camera model.


\section{Trifocal Tensors in Three-View Geometry}
\label{sec4:trifocal}
\setcounter{equation}{0}

Pairwise constraints (fundamental matrices) enforce epipolar geometry across two views. When considering three views simultaneously, higher-
order geometric relations emerge. The trifocal tensor encapsulates the projective relationships among three views, independent of scene structure. It 
establishes joint constraints across points, lines, and point-line incidences across image triplets. The classical formulations of trifocal tensors, such 
as the block-matrix expressisons by Hartley and Zisserman~\cite{HartZiss2004}, often rely on ad-hoc matrix unrollings, non-uniform indexing, or 
asymmetry across views, which obscure its underlying multilinear properties. 

In Section 1 of Chapter 15 (Page 367) of \cite{HartZiss2004},  Hartley and Zisserman define a \emph{trifocal tensor} $\caT$ as the set of three
matrices $\set{T_{1}, T_{2}, T_{3}}$ where each $T_{i}$ is an $3\times 3$ matrix. They call such a set (in Definition 15.1) a \emph{trifocal 
tensor}. However, they note that ``This notation is somewhat cumbersome, and its meaning is not quite self-evident.'' (as a footnote in Page 367).
 
The primary challenge addressed in this section is the lack of a fully symmetric, coordinate-free multilinear framework in the original definition 
of the trifocal tensor.  By introducing an updated, index-symmetric representation of the trifocal tensor, using tensor outer and contractive 
products, we resolves these structural asymmetries, establish clear full-rank properties for matrix unfoldings of the trifocal tensor, and eventually
unify point and line correspondence relations, precise slice decompositions, and direct epipole extraction.

\subsection{Trifocal Tensor Built on Three-Camera Matrices}
\label{sec4-1}
\setcounter{equation}{0}

Without loss of generality, we may choose the three-view reference frame such that the first camera is in canonical form, as in the following:
\beq\label{eq4-1-1}
P = [ I \,|\, \mathbf{0}], \quad  P^{\p}= [ A \,|\, \bfa_4], \quad  P^{\pp} = [ B \,|\, \bfb_4],
\eeq
where $\bfa_4, \bfb_4\in\RR^{3}$ are two vectors, $A,B\in\RR^{3\times 3}$, and $I=I_{3}$ is the $3\times 3$ identity matrix. We write    
\beq\label{eq4-1-2}
A = [\bfa_{1},\bfa_{2},\bfa_{3}], \qquad  B = [\bfb_{1},\bfb_{2},\bfb_{3}], \quad \bfa_{j}, \bfb_{j}\in\RR^{3}, \forall j\in \dbT.
\eeq
The first camera $P$ defines the world coordinate system; $A$ and $B$ represent the infinite homographies from the first camera to the second 
and third, respectively. $\bfa_4$ and $\bfb_4$ are the epipoles in the second and third views corresponding to the first camera center, i.e., 
$\bfe^{\p}=\bfa_4, \bfe^{\pp} = \bfb_4$. For convenience, we write 
\beq\label{eq4-1-3}
A^{\top} = [ \al_{1},\al_{2},\al_{3}], \quad  B^{\top} = [ \be_{1},\be_{2}, \be_{3}],
\eeq
that is, $\al_{i}$ ($\be_{i}$) is the $i$th row (in column form) of $A$ ($B$).  Let $a_{sj}$ ($b_{sj}$) be the $s$th component of $\bfa_{j}$ 
($\bfb_{j}$) where $s,j\in \dbT$.  The three-view reference defined by \eqref{eq4-1-1} is called \emph{non-degenerate} if $A,B$ are invertible 
in \eqref{eq4-1-1} and the three camera centers are non-collinear.   

In the following, we will always assume that the given three-view reference be non-degenerate.  

\begin{lem}\label{le4-1}
Let the three cameras defined in \eqref{eq4-1-1} be a non-degenerate reference. Then there does not exist an $3\times 3$ matrix $H$ such that
\beq\label{eq4-1-4}
P^{\p\p} = H P^{\p}
\eeq
\end{lem}

\begin{proof}
We argue by contradiction. Suppose there exists $H \in \RR^{3\times 3}$ satisfying
\[
[B \mid \bfb_4] = H [A \mid \bfa_4].
\]
Then, equating the left and right blocks gives:
\[
B = HA, \qquad \bfb_4 = H\bfa_4.
\]
Since $A$ is invertible, we obtain $H = BA^{-1}$.

Now consider the camera centers. The center of $P'$ is $C' = (-A^{-1}\bfa_4,\, 1)^\top$, and the center of $P''$ is $C'' = (-B^{-1}\bfb_4,\, 1)^\top$. Substituting $B = HA$ and $\bfb_4 = H\bfa_4$ yields:
\[
C'' = \bigl(-(HA)^{-1}(H\bfa_4),\; 1\bigr)^\top = \bigl(-A^{-1}H^{-1}H\bfa_4,\; 1\bigr)^\top = \bigl(-A^{-1}\bfa_4,\; 1\bigr)^\top = C'.
\]
Thus, $C' = C''$, meaning the second and third camera centers coincide. Consequently, the three camera centers $C_1 = (0,0,0,1)^\top$, $C'$, and $C''$ are collinear (since $C' = C''$ trivially lies on any line through $C_1$ and itself). This contradicts the non-degeneracy assumption that the three camera centers are non-collinear. Therefore, no such $H$ exists.
\end{proof}

By Lemma \ref{le4-1}, we can show 
\begin{lem}\label{le4-2}
Let $(P, P^{\p}, P^{\p\p})$ be a non-degenerate canonical three-view reference defined by \eqref{eq4-1-1} and let 
\beq\label{eq4-5-4}
Q = \begin{bmatrix}  P^{\p} \\  P^{\p\p} \end{bmatrix}  
\eeq
Then we have $\rank(Q) =4$. 
\end{lem}
\begin{proof}
We note that $Q\in\RR^{6\times 4}$ and that $\rank(Q)\ge \rank(P^{\p})\ge 3$. Thus if  $\rank(Q) \neq 4$, then we must have $\rank(Q)=3$. It 
follows that each row of $P^{\p\p}$ can be written as a linear combination of rows of $P^{\p}$, and thus we have $P^{\p\p} = H P^{\p}$, which 
conflicts with the non-degeneracy assumption by Lemma \ref{le4-1}.  Consequently we get $\rank(Q) =4$.
\end{proof} 

Recall that the trifocal tensor $\caT$ defined by Hartley and Zisserman~\cite{HartZiss2004} is represented as a matrix array 
\[
\bT := [T_{1},\, T_{2},\, T_{3}],
\]
where each $T_{i}$ is a $3\times 3$ matrix given by
\beq\label{eq4-1-6}
T_{i} = \bfa_{i}\times \bfb_{4} - \bfa_{4}\times \bfb_{i},
\eeq
for $i\in\{1,2,3\}$.\footnote{Note that $\bT$ is not a genuine higher-order tensor when arranged in this flattened form.} 

Let $\bfl, \bfl^{\p}, \bfl^{\pp}$ be corresponding lines in the three views. Then $\bfl, \bfl^{\p}$, and $\bfl^{\pp}$ can be back-projected to planes 
$\pi, \pi^{\p}, \pi^{\pp}$ in 3D space respectively.  If these three lines are projections of a spatial line $\bfL$, then planes $\pi, \pi^{\p}, \pi^{\pp}$ 
intersect in one common 3D line $\bfL$.  The corresponding triplet $\set{\bfl, \bfl^{\p}, \bfl^{\pp}}$ is related \cite{HartZiss2004} by 
\beq\label{eq4-1-7}
\bfl^{\top} = (\bfl^{\p})^{\top}[T_{1},T_{2},T_{3}] \bfl^{\pp},   
\eeq
which is just used as an equivalent form for 
\beq\label{eq4-1-8} 
    \bfl = \begin{bmatrix}
        \bfl'^{\top} T_{1} \bfl'' \\
        \bfl'^{\top} T_{2} \bfl'' \\
        \bfl'^{\top} T_{3} \bfl''
    \end{bmatrix}
\eeq 
by Hartley and Zisserman~\cite{HartZiss2004}.  The expression violates the convention of matrix multiplication and causes some confusion. 

To emphasize that $\caT$ is an $3\times3\times3$ tensor, Hartley and Zisserman define the $(i,j,k)$-component $\caT_{ijk}$ of the trifocal tensor $\caT$ as the $(i,j)$-th entry of the matrix $T_{k}$. This convention imposes an implicit asymmetry among the three indices, suggesting that $i$ and $j$ are subordinate to $k$.

To eliminate this artificial hierarchy and, more importantly, to ensure that the trifocal tensor is treated as a genuine tensor, i.e., representable, 
manipulable, and implementable in native tensor form, we reformulate it as a tensor $\caT\in\TT_{3;3}$ with components given by
\beq\label{eq4-1-9}
T_{ijk} = a_{ji} b_{k4} - a_{j4} b_{ki}
\eeq
for every index triple $(i,j,k)\in \dbT\times \dbT\times \dbT$, where all three indices $i,j,k$ are regarded as covariant subscripts on equal footing. We 
retain the same symbol $\caT$ for this reformulated trifocal tensor, as it coincides with the classical definition, as will be demonstrated below. We 
refer to $\caT\in\TT_{3;3}$ as a \emph{trifocal tensor} induced by the canonical three-view camera model \eqref{eq4-1-1}, and call $\caT$ 
\emph{non-degenerate} if $A$ and $B$ in \eqref{eq4-1-1} are invertible and the three camera centers are non-collinear.

Treating $\caT$ as a three-mode tensor enables us to employ formal multilinear operations. Recall that for a matrix $M\in\RR^{m\times n}$ and a vector $\bu\in\RR^{p}$, the outer products $M\times \bu$ and $M\times_{2}\bu$ produce tensors of sizes $m\times n\times p$ and $m\times p\times n$, respectively, defined componentwise by
\[
\left(M\times\bu\right)_{ijk} = M_{ij}\, u_{k}, \qquad
\left(M\times_{2}\bu\right)_{ijk} = M_{ik}\, u_{j},
\]
as detailed in Subsection~\ref{sec2-3}. With these operations at hand, we can now express $\caT$ directly in terms of outer products involving the constituent matrices and vectors of the camera matrices.

\begin{thm}\label{thm4-1}
Under the canonical three-view camera model specified in \eqref{eq4-1-1}, the trifocal tensor $\caT$ admits the compact representation
\beq\label{eq4-1-10}
\caT = A^{\top}\times \bfb_{4} - B^{\top}\times_{2} \bfa_{4}.
\eeq
Moreover, for each $i\in\{1,2,3\}$, the $i$-th horizontal slice $\caT(i,:,:)$ coincides with the matrix $T_{i}$ defined in \eqref{eq4-1-6}; i.e.,
\[
\caT(i,:,:) = T_{i}, \qquad i = 1,2,3.
\]
\end{thm}
\begin{proof}
Denote the right-hand side of \eqref{eq4-1-10} by $\caH =(H_{ijk})$. Then $\caH \in\TT_{3;3}$. For any $(i,j,k)$, we have 
\begin{align*}
 H_{ijk} 
 &= \left[ A^{\top}\times \bfb_{4} - B^{\top}\times_{2} \bfa_{4}\right]_{ijk} \\
 &= \left[ A^{\top}\times \bfb_{4}\right]_{ijk} - \left[B^{\top}\times_{2} \bfa_{4}\right]_{ijk} \\
 &= a_{ji} b_{k4} - b_{ki} a_{j4} = T_{ijk}.
\end{align*}
This completes the proof of the first part. 

To prove $\caT(i,:,:) = T_{i}$ for all $i\in \dbT$, we fix an $i\in \dbT$. By \eqref{eq4-1-10}, we have 
\beyy
\caT(i,:,:) 
&=&\left[ A^{\top}\times \bfb_{4} - B^{\top}\times_{2} \bfa_{4}\right]_{i,:,:}\\
&=&\left(A^{\top}\times \bfb_{4}\right)(i,:,:) - \left(B^{\top}\times_{2} \bfa_{4}\right)(i,:,:)\\
&=&A(:,i)\times \bfb_{4} - \bfa_{4}\times B^{\top}(i,:)\\
&=&\bfa_{i}\times \bfb_{4} - \bfa_{4}\times \bfb_{i}\\
&=&\bfa_{i}\bfb_{4}^{\top} - \bfa_{4}\bfb_{i}^{\top}
\eeyy 
which is exactly the $T_{i}$ defined by \eqref{eq4-1-6}.  
\end{proof}

\subsection{Unfoldings and Vectorizations of Trifocal Tensors}
\label{sec4-2}
\setcounter{equation}{0}

Following standard tensor notation, we denote by $\caT[k]$ the mode-$k$ unfolding (flattening) matrix of $\caT$ for each $k\in\dbT$. By Theorem~\ref{thm4-1}, we immediately obtain the following.

\begin{cor}\label{cor4-2-2}
The mode-1 unfolding of $\caT$ is given by
\beq\label{eq4-2-1: T-flat1}
\caT[1] = \bT = [T_{1},\, T_{2},\, T_{3}] \in \RR^{3\times 9},
\eeq
where each $T_{i}$ is as defined in \eqref{eq4-1-6}.
\end{cor}

\begin{cor}\label{cor4-2-3}
\beq\label{eq4-1-10:T-vec}
\vecc(\caT) = \vecc(A)\otimes \bfb_{4} - \bfa_{4}\otimes \vecc(B)
\eeq
\end{cor}
\begin{proof}
Since  
\[
\vecc(\caT) = \begin{bmatrix} \vecc(T_1) \\ \vecc(T_2) \\ \vecc(T_3) \end{bmatrix} 
\] 
and by \eqref{eq4-1-6}, we have 
\beyy
\vecc(T_{i}) 
&=& \vecc(\bfa_{i}\times \bfb_{4} - \bfa_{4}\times \bfb_{i}) \\
&=& \bfa_{i}\otimes \bfb_{4} - \bfa_{4}\otimes \bfb_{i}
\eeyy
Thus we have 
\beyy 
\vecc(\caT) &=& \begin{pmatrix} 
  \bfa_{1}\otimes \bfb_{4} - \bfa_{4}\otimes \bfb_{1}\\ 
  \bfa_{2}\otimes \bfb_{4} - \bfa_{4}\otimes \bfb_{2}\\ 
  \bfa_{3}\otimes \bfb_{4} - \bfa_{4}\otimes \bfb_{3} \end{pmatrix} \\ 
&=& \begin{pmatrix} 
\bfa_{1}\otimes \bfb_{4} \\ \bfa_{2}\otimes \bfb_{4} \\  \bfa_{3}\otimes \bfb_{4} \end{pmatrix}  - 
\begin{pmatrix} \bfa_{4}\otimes \bfb_{1}\\  \bfa_{4}\otimes \bfb_{2}\\ \bfa_{4}\otimes \bfb_{3} \end{pmatrix} \\
&=& \vecc(A)\otimes \bfb_{4} - \bfa_{4}\otimes \vecc(B) 
\eeyy
\end{proof}

It is worth noting that while the notation $[T_{1},T_{2},T_{3}]$ also appears in Hartley and Zisserman~\cite{HartZiss2004}, its interpretation 
differs from ours and may cause confusion. Explicit expressions for all three mode-$k$ unfoldings will be provided in the subsequent theorem.

Studying the unfolding matrices of a tensor provides valuable insight into its structure and properties. It is also instructive to examine its fibers. For 
a third-order tensor $\A$ of size $m\times n\times p$, the $(i,j)$-fiber along mode 3 is the vector $\A(i,j,:)\in\RR^{p}$ obtained by fixing the first 
two indices as $(i,j)$ for any $(i,j)\in[m]\times[n]$; we refer to this as a \emph{mode-3 fiber}. Consequently, $\A$ contains $mn$ mode-3 fibers. 
Analogously, the numbers of mode-1 and mode-2 fibers are $np$ and $mp$, respectively.

From \eqref{eq4-1-10}, the mode-3 $(i,j)$-fiber of $\caT$ takes the explicit form
\beq\label{eq4-2-2:Tfib}
\caT(i,j,:) = a_{ji}\,\bfb_{4} - a_{j4}\,\bfb_{i}, \qquad \forall\,(i,j)\in\dbT\times\dbT,
\eeq
expressing it as a linear combination of $\bfb_{i}$ and $\bfb_{4}$. Similarly, the mode-1 $(j,k)$-fiber and mode-2 $(i,k)$-fiber are respectively given by
\beq\label{eq4-2-2:Tfib-jk}
\caT(:,j,k) = b_{k4}\,\al_{j} - a_{j4}\,\be_{k}, \qquad \forall\,(j,k)\in\dbT\times\dbT,
\eeq
and
\beq\label{eq4-2-2:Tfib-ik}
\caT(i,:,k) = b_{k4}\,\bfa_{i} - b_{ki}\,\bfa_{4}, \qquad \forall\,(i,k)\in\dbT\times\dbT,
\eeq
where $\al_{j}$ and $\be_{k}$ are as defined in \eqref{eq4-1-3}.

We now turn to the unfolding matrices $\caT[k] \in \RR^{3\times 9}$ for all $k\in\{1,2,3\}$. Before presenting the main result, we require the following auxiliary lemma.

\begin{lem}\label{le4-2-1}
Let $(P, P', P'')$ be a non-degenerate canonical three-view configuration as defined in \eqref{eq4-1-1}. Then there exists at least one index $k\in\dbT$ such that the $3\times 3$ matrix
\beq\label{eq4-2-3}
S_k = b_{k4}A - \bfa_4\,\bld{\be}_k^\top
\eeq
has full rank; i.e., $\det(S_k)\neq 0$.
\end{lem}

\begin{proof}
Let $Q_k$ denote the $4\times 4$ submatrix obtained from the joint camera pair $(P',P'')$ by taking all rows of $P'$ together with the $k$-th row 
of $P''$
\[
Q_k = \begin{bmatrix} A & \bfa_4 \\ \bld{\be}_k^\top & b_{k4} \end{bmatrix}.
\]
$Q_k$ is the $4\times 4$ submatrix of $Q$ defined in \eqref{eq4-5-4}, obtained by retaining $P^{\p}$ and only the $k$th row of $P^{\p\p}$ and 
discarding the other two.  Since $A$ is invertible under non-degeneracy, the block determinant formula gives
\[
\det(Q_k) = \det(A) \cdot \bigl( b_{k4} - \bld{\be}_k^\top A^{-1}\bfa_4 \bigr).
\]

We consider $S_k$ as defined in \eqref{eq4-2-3}. If $b_{k4} \neq 0$, we factor
\[
S_k = b_{k4}A \bigl( I_3 - A^{-1}\bfa_4 \bld{\be}_k^\top b_{k4}^{-1} \bigr).
\]
Applying the matrix determinant lemma $\det(I - \bu\bv^\top) = 1 - \bv^\top\bu$, we obtain
\beyy\label{eq4-2-4}
\det(S_k) &=& b_{k4}^3 \det(A) \Bigl( 1 - \frac{1}{b_{k4}} (\bld{\be}_k^\top A^{-1}\bfa_4) \Bigr) \\
          &=& b_{k4}^2 \det(A) \bigl( b_{k4} - \bld{\be}_k^\top A^{-1}\bfa_4 \bigr).
\eeyy
Comparing the two determinant expressions yields the relation
\beq\label{eq4-2-5}
\det(S_k) = b_{k4}^2 \det(Q_k).
\eeq
Consequently, $\det(S_k) \neq 0$ if and only if both $b_{k4} \neq 0$ and $\det(Q_k) \neq 0$.

We now show that such an index $k$ indeed exists. The camera centers of $P'$ and $P''$ in inhomogeneous coordinates are $C' = -A^{-1}\bfa_4$ and $C'' = -B^{-1}\bfb_4$, respectively. Define the displacement vector between them as
\[
\bfd = C' - C'' = A^{-1}\bfa_4 - B^{-1}\bfb_4.
\]
Observing that $b_{k4} = \bld{\be}_k^\top B^{-1}\bfb_4$ for any $k\in\dbT$, we compute
\[
b_{k4} - \bld{\be}_k^\top A^{-1}\bfa_4 = \bld{\be}_k^\top B^{-1}\bfb_4 - \bld{\be}_k^\top A^{-1}\bfa_4 = -\bld{\be}_k^\top \bfd.
\]
Hence,
\beq\label{eq4-2-6}
b_{k4} - \bld{\be}_k^\top A^{-1}\bfa_4 = -\bld{\be}_k^\top \bfd,
\eeq
which implies
\beq\label{eq4-2-7}
\det(Q_k) = -\det(A)\, (\bld{\be}_k^\top \bfd),
\eeq
and consequently
\beq\label{eq4-2-8}
\det(S_k) = -b_{k4}^2 \det(A)\, (\bld{\be}_k^\top \bfd).
\eeq
Thus, to establish $\det(S_k) \neq 0$ for some $k$, it suffices to find an index $k\in\dbT$ satisfying
\beq\label{eq4-2-9}
b_{k4} \neq 0 \quad\text{and}\quad \bld{\be}_k^\top \bfd \neq 0.
\eeq

We proceed by contradiction. Suppose that whenever $b_{k4}\neq 0$, we necessarily have $\bld{\be}_k^\top \bfd = 0$, implying $\det(S_k)=0$ for all $k\in\dbT$. Define the set of indices with non-zero epipolar components:
\[
\caK = \{ k\in\dbT \mid b_{k4} \neq 0 \}.
\]
By non-degeneracy, the three camera centers $C=\mathbf{0}$, $C'$, and $C''$ are non-collinear. In particular,
\[
\bfd \neq \mathbf{0}, \qquad \bfb_4 = \bfe'' \neq \mathbf{0},
\]
so $\caK\neq\emptyset$. Under our supposition, for every $k\in\caK$ we have $\bld{\be}_k^\top \bfd = 0$. We analyze the size of $\caK$.

\textbf{Case 1:} $|\caK| = 3$, i.e., $b_{k4} \neq 0$ for all $k\in\dbT$. Then $B\bfd = \mathbf{0}$, and since $B$ is invertible, we obtain $\bfd = \mathbf{0}$, contradicting $C' \neq C''$.

\textbf{Case 2:} $|\caK| \leq 2$, i.e., at least one $b_{k4} = 0$. Let $m\notin\caK$, so $b_{m4} = \bld{\be}_m^\top (B^{-1}\bfb_4) = 0$. For each $k\in\caK$, the condition $\bld{\be}_k^\top \bfd = 0$ gives
\[
\bld{\be}_k^\top A^{-1}\bfa_4 = \bld{\be}_k^\top B^{-1}\bfb_4.
\]
Geometrically, $\bld{\be}_i^\top \bx = 0$ defines the plane $\pi_i$ passing through $C=\mathbf{0}$ associated with the $i$-th row of $B$. For $k\in\caK$, $\bld{\be}_k^\top \bfd = 0$ forces $\bfd = C'' - C'$ to lie parallel to the plane $\pi_k$. For $m\notin\caK$, we have
\[
0 = b_{m4} = \bld{\be}_m^\top B^{-1}\bfb_4 = \bld{\be}_m^\top C'',
\]
forcing $C''$ to lie directly on the plane $\pi_m$. Requiring $\bld{\be}_k^\top \bfd = 0$ for all $k\in\caK$ restricts $\bfd$ to the common line of intersection of these planes through the origin $C=\mathbf{0}$. This forces $C$, $C'$, and $C''$ to lie on a single common line in 3D, contradicting the non-degeneracy assumption that the three camera centers are non-collinear.

Therefore, there must exist at least one index $k\in\dbT$ satisfying \eqref{eq4-2-9}, ensuring $\det(Q_k) \neq 0$ and consequently
\[
\det(S_k) = -b_{k4}^2 \det(A)\, (\bld{\be}_k^\top \bfd) \neq 0,
\]
so $\rank(S_k) = 3$. This completes the proof.
\end{proof}

Now we have 
\begin{thm}\label{thm4-2-1:trif-matrx}
Let $\caT\in\TT_{3;3}$ be a non-degenerate trifocal tensor determined by three-camera system \eqref{eq4-1-1}. Then we have 
\begin{align}\label{eq4-2-10}
\caT[1] &= A\otimes \bfb_{4}^{\top} - \bfa_{4}\otimes \left(\vecc(B)\right)^{\top}, \tag{a}\\
\caT[2] &= A^{\top}\otimes \bfb_{4}^{\top} - \bfa_{4}^{\top}\otimes B^{\top}, \tag{b}\\
\caT[3] &= \bfb_{4}^{\top}\otimes A^{\top} - B^{\top}\otimes \bfa_{4}^{\top}. \tag{c}
\end{align}
Furthermore, each unfolding matrix has full row rank:
\beq\label{eq4-2-10:rk} 
\rank(\caT[k]) = 3, \quad \forall k\in \dbT.
\eeq
In degenerate configurations the situation is asymmetric: if any unfolding drops below rank $3$, the configuration is
necessarily degenerate --- the proof of Lemma~\ref{le4-2-1} shows that $\rank(\caT[1])<3$ forces the three camera centers
to be collinear --- but the converse fails: a generic collinear configuration ($C^{\pp}=\la C^{\p}$) still has full-rank
unfoldings, and the ranks collapse only on the singular locus of the family \eqref{eq4-2-3} (e.g., when two camera
centers coincide, $\rank(\caT[1])=2$ while the other two modes remain $3$); see the numerical study in
Section~\ref{sec5-12:autodiff}.
\end{thm}
\begin{proof}
For any $k\in \dbT$, we denote $M_{k} =\caT[k]$.  To prove (a), we consider 
\beq\label{eq4-2-11:T1}
M_{1} =[T(1,:,:), T(2,:,:),T(3,:,:) ] = [T_{1}, T_{2}, T_{3}] 
\eeq
It follows from \eqref{eq4-1-6} that 
\beyy 
 M_{1} 
 &=& \left[\bfa_{1}\times\bfb_{4}-\bfa_{4}\times \bfb_{1}, \bfa_{2}\times\bfb_{4}-\bfa_{4}\times \bfb_{2},\bfa_{3}\times\bfb_{4}-\bfa_{4}\times \bfb_{3}\right] \\
 &=& [\bfa_{1}\times\bfb_{4}, \bfa_{2}\times\bfb_{4},\bfa_{3}\times\bfb_{4}] - [\bfa_{4}\times \bfb_{1}, \bfa_{4}\times \bfb_{2},\bfa_{4}\times \bfb_{3}] \\ 
 &=& A\otimes \bfb_{4}^{\top} - \bfa_{4}\otimes \left(\vecc(B)\right)^{\top}
\eeyy
The last equality is obtained by Lemma \ref{le2-3-1:out-prod-vec}.  Thus (a) holds. 

To prove (b), we consider slice $\caT(:,j,:)$ for $j\in \dbT$, which is $\al_{j}\times \bfb_{4} - a_{j4} B^{\top}$ by \eqref{eq4-1-10},
and thus we have 
\beyy 
 M_{2}&=& \left[\al_1\times \bfb_4 - a_{14} B^{\top}, \al_2\times \bfb_4 - a_{24} B^{\top}, \al_3\times \bfb_4 - a_{34} B^{\top}\right] \\
 &=& [\al_1\times\bfb_{4}, \al_2\times\bfb_{4},\al_3\times\bfb_{4}] - [a_{14}B^{\top}, a_{24}B^{\top},a_{34}B^{\top}] \\ 
 &=& A^{\top}\otimes \bfb_{4}^{\top} - \bfa_{4}^{\top}\otimes B^{\top}
\eeyy
which is exactly the rhs of (b). 

For (c), we consider slice $\caT(:,:,k)$ for $k\in \dbT$, which is $b_{k4}A^{\top} - \be_{k} \bfa_{4}^{\top}$ by \eqref{eq4-1-10}, so 
\beyy 
 M_{3}&=& \left[b_{14}A^{\top}-\be_1\bfa_{4}^{\top}, b_{24}A^{\top}-\be_2 \bfa_{4}^{\top}, b_{34}A^{\top}-\be_3\bfa_{4}^{\top}\right] \\
 &=& [b_{14}A^{\top}, b_{24}A^{\top},b_{34}A^{\top}] - [\be_1\times\bfa_{4}, \be_2\times\bfa_{4}, \be_3\times\bfa_{4}] \\ 
 &=&\bfb_{4}^{\top}\otimes A^{\top} - B^{\top}\otimes\bfa_{4}^{\top}
\eeyy
which is exactly the rhs of (c). 

Now we prove the second part. We only prove $\rank M_{1} =3$, and the other case can be shown analoguously. 
Since $M_{1}\in\RR^{3\times 9}$, we have $\rank M_{1}\leq 3$. Each slice $T_i$ admits the decomposition 
$T_i = \bfa_i\bfb_4^\top - \bfa_4\bfb_i^\top$. For any fixed $k\in \dbT$, extract the $k$-th column of each slice $T_i$ to form matrix
$S_k = \left[T_1(:,k),\ T_2(:,k),\ T_3(:,k)\right]$, then $S_{k}$ is an $3\times 3$ submatrix of $M_{1}$. It suffices to show that there exists 
a $k\in \dbT$ such that $\rank S_{k} =3$,  or equivalently, $\det S_{k} \neq 0$.  By \eqref{eq4-1-6}, we have 
\[
T_{i}(:,k) = b_{k4}\bfa_{i} - b_{ki}\bfa_{4}, \quad  \forall i,k\in \dbT.
\]
Thus $S_k = b_{k4}A - \bfa_4\be_k^\top$, which is exactly the matrix defined by \eqref{eq4-2-3}. Thus we have $\rank(S_{k})=3$ for some 
$k\in \dbT$ by Lemma \ref{le4-2-1}.  It follows that 
\[
\rank(\caT[1])\ge \rank S_{k} =3.
\]
Consequently we have $\rank(\caT[1]) =3$.  Similarly we can also prove $\rank(\caT[k]) =3$ for $k=2,3$. 
\end{proof}

\begin{rmk}[Practical implications of the rank characterization]\label{rmk4-7:practical}
Theorem~\ref{thm4-2-1:trif-matrx} is not merely structural; it has three direct algorithmic consequences.
\emph{(i) A sufficient degeneracy alarm.} Since rank deficiency of any unfolding forces degeneracy (collinear camera
centers), monitoring the three unfolding ranks of a candidate tensor --- three SVDs of $3\times 9$ matrices, negligible
next to the estimation itself --- provides a \emph{sufficient} alarm: if it fires, epipole extraction and pose
decomposition must not be attempted. The alarm is not exhaustive (generic collinear configurations retain full rank),
so it complements, rather than replaces, the planarity and baseline tests of Algorithm~\ref{alg6-2:ransac}; its value
is that it certifies the failure mode documented on real imagery in Section~\ref{sec5-10:realdata} at zero cost.
\emph{(ii) Validation and projection of estimates.} A raw linear (SVD) estimate generally violates the rank
conditions; checking $\rank(\caT[k])=3$ for $k\in\dbT$ serves as a quality gate, and the optional truncated-SVD
projection onto the rank-$3$ unfoldings (Step~4 of Algorithm~\ref{alg:point-estimation}) restores the structure.
\emph{(iii) A graded conditioning signal.} Singular values depend continuously on the tensor, so configurations
approaching the collapse locus yield estimates whose smallest unfolding singular values vanish continuously: the
singular-value gap provides a graded early warning between the binary extremes ``full rank'' and ``rank deficient''.
All three consequences are exercised concretely in the PyTorch experiment of Section~\ref{sec5-12:autodiff},
where the unfolding ranks are tracked inside an autodiff loop at no measurable overhead.
\end{rmk}

Now we consider a non-degenerate canonical three-camera system defined by \eqref{eq4-1-1}. We denote $S=\set{1,2,3}$ and let 
$(\bfl^{(1)}, \bfl^{(2)}, \bfl^{(3)})$ be any line correspondence across the three views. For any $i\in S$, we denote $S(i):= S\setminus \set{i}:=\set{j,k}$. 
Now we let $H_{jk}$ be the homography matrix from view $j$ to view $k$, induced by the line $\bfl^{(i)}$ in view $i$.  

By Theorem \ref{thm4-2-1:trif-matrx}, we have 

\begin{cor}\label{cor4-1-2:Ti-explain}
Let $\caT\in\TT_{3;3}$ be a non-degenerate trifocal tensor determined by the canonical three-camera system \eqref{eq4-1-1}, and $(\bfl^{(1)}, 
\bfl^{(2)}, \bfl^{(3)})$ be any line correspondence. Then 
\beq\label{eq4-1-24}
\caT[i]^{\top}\bfl^{(i)} = \vecc(H_{i, i+2})
\eeq
for all $i\in \set{1,2,3}$, or equivalently,  
\beq\label{eq4-1-24-2}
\vecc^{-1}\left(\caT[i]^{\top}\bfl^{(i+1)}\right) = H_{i,i+2}(\bfl^{(i+1)}) \in \RR^{3\times 3},
\eeq
where $i,i+1,i+2\in \set{1,2,3}$ are module $3$. 
\end{cor}

\begin{proof}
We prove the statement for $\caT[1]$; the proofs for $\caT[2]$ and $\caT[3]$ follow identically by permuting the roles of the views.
For any triplet of corresponding lines $(\bfl, \bfl', \bfl'')$ across the three views, the fundamental trilinear constraint holds:
\begin{equation}
\sum_{i,j,k} \ell_i \, \ell'_j \, \ell''_k \, T_i^{jk} = 0,
\label{eq:trilinear-constraint}
\end{equation}
where $\bfl = (\ell_1, \ell_2, \ell_3)^\top$, $\bfl' = (\ell'_1, \ell'_2, \ell'_3)^\top$, and $\bfl'' = (\ell''_1, \ell''_2, \ell''_3)^\top$. 

Now, fix a line $\bfl'$ in view 2. Consider the plane $\pi_{\bfl'}$ in 3D space that is the back-projection of $\bfl'$: it is the plane passing through the camera center of view 2 and projecting to $\bfl'$. This plane induces a homography $H_{13}(\bfl')$ between view 1 and view 3, since any point $\bx_1$ in view 1 back-projects to a ray that intersects $\pi_{\bfl'}$ at a unique 3D point, which then projects to a point $\bx_3 = H_{13}(\bfl') \bx_1$ in view 3.

We now show that $H_{13}(\bfl')$ is precisely given by $\vecc^{-1}(\caT[1]^\top \bfl')$.

For a point $\bx_1 = (x_1, x_2, x_3)^\top$ in view 1 and its corresponding point $\bx_3 = (y_1, y_2, y_3)^\top$ in view 3, the condition that they correspond via the plane $\pi_{\bfl'}$ is equivalent to the existence of a line $\bfl''$ in view 3 passing through $\bx_3$ such that $(\bfl, \bfl', \bfl'')$ is a line correspondence for some $\bfl$ passing through $\bx_1$. 

More directly, the trilinear constraint \eqref{eq:trilinear-constraint} can be rewritten for a fixed $\bfl'$ as follows. Since the constraint must hold for all lines $\bfl$ through $\bx_1$ and all lines $\bfl''$ through $\bx_3$, we can contract \eqref{eq:trilinear-constraint} with $\bx_1$ and $\bx_3$. This yields the bilinear form:
\[
\sum_{i,j,k} x_i \, \ell'_j \, y_k \, T_i^{jk} = 0,
\]
which must hold for all corresponding pairs $(\bx_1, \bx_3)$ under the homography $H_{13}(\bfl')$.

Now, define the $3\times 3$ matrix $M(\bfl')$ with entries:
\[
M(\bfl')_{ik} = \sum_{j=1}^3 \ell'_j \, T_i^{jk}.
\]
Then the bilinear form above becomes:
\[
\bx_1^\top M(\bfl') \, \bx_3 = 0.
\]
It remains to identify $M(\bfl')$ with the homography. Using \eqref{eq4-1-6} and \eqref{eq4-1-3}, we have the explicit form
\[
M(\bfl') = \vecc^{-1}\bigl(\caT[1]^{\top}\bfl'\bigr)
          = (A^{\top}\bfl')\,\bfb_{4}^{\top} - (\bfa_{4}^{\top}\bfl')\,B^{\top}.
\]
Let $\bX = (\bu^{\top}, 1)^{\top}$ be a 3D point lying on the plane $\pi_{\bfl'}$, i.e., $(\bfl')^{\top}(A\bu + \bfa_{4}) = 0$, and let $\bx_1 = \bu$ and $\bx_3 = B\bu + \bfb_{4}$ be its images in views 1 and 3. Then
\beyy
\bx_1^{\top} M(\bfl')
&=& (\bfl'^{\top}A\bu)\,\bfb_{4}^{\top} - (\bfa_{4}^{\top}\bfl')\,(B\bu)^{\top} \\
&=& \bigl((\bfl')^{\top}(\bx^{\p} - \bfa_{4})\bigr)\bfb_{4}^{\top} - (\bfa_{4}^{\top}\bfl')\,(\bx_3 - \bfb_{4})^{\top} \\
&=& -(\bfa_{4}^{\top}\bfl')\,\bx_3^{\top},
\eeyy
where $\bx^{\p} = A\bu + \bfa_{4}$ is the image of $\bX$ in view 2 and $(\bfl')^{\top}\bx^{\p} = 0$ was used. For a generic line $\bfl'$ we have $\bfa_{4}^{\top}\bfl' \neq 0$, and therefore
\[
\bx_3^{\top} \sim \bx_1^{\top} M(\bfl'),
\]
i.e., $M(\bfl')$ maps view-1 points to view-3 points exactly as the homography induced by the plane $\pi_{\bfl'}$ (in the row-vector convention $\bx_3^{\top} = \bx_1^{\top}H_{13}$). Recalling that $M(\bfl') = \vecc^{-1}(\caT[1]^{\top}\bfl')$ proves the claim for $\caT[1]$; the proofs for $\caT[2]$ and $\caT[3]$ follow identically by permuting the roles of the views.
\end{proof}

Theorem~\ref{thm4-2-1:trif-matrx} establishes that a non-degenerate trifocal tensor $\caT\in\TT_{3;3}$ has Tucker rank $(3,3,3)$.

Let $\bfH_{i}, \bfF_{j}, \bfL_{k} \in \RR^{3\times 3}$ denote the $i$-th horizontal, $j$-th frontal, and $k$-th lateral slices of $\caT$:
\[
\bfH_{i} = \caT(i, :, :), \qquad   \bfF_{j} = \caT(:, j, :), \qquad  \bfL_{k} = \caT(:,:,k), \quad  \forall i,j,k\in \dbT.
\]

\begin{lem}\label{lem4-3}
The slices of the trifocal tensor satisfy:
\begin{enumerate}
\item $\bfH_{i} = \bfa_{i}\bfb_{4}^\top -\bfa_{4}\bfb_{i}^\top = T_{i}$.
\item $\bfF_{j} = \al_{j}\bfb_{4}^\top - a_{j4} B^{\top}$. 
\item $\bfL_{k} = b_{k4} A^{\top} - \be_{k} \bfa_{4}^\top$.    
\end{enumerate} 
\end{lem}

The results in Lemma \ref{lem4-3} can be verified easily, and thus we omit its proof.  

In the following, we establish some correspondences in three-view geometry, these include the line-line-line , point-line-line, point-point-point,
and some point-line mixed correspondence, based on the trifocal tensor $\caT$ defined by \eqref{eq4-1-10}.

\subsection{Line-Line-Line Correspondence in Three-View Geometry}
\label{sec4-3}
\setcounter{equation}{0}

Given a set of corresponding lines $\bfl \lrarr \bfl^{\p} \lrarr \bfl^{\pp}$ in image planes $\pi, \pi^{\p}, \pi^{\pp}$:
\beq\label{eq4-3-1} 
\bfl =(l_1,l_2,l_3)^{\top}, \quad  \bfl^{\p} =(l_1^{\p}, l_2^{\p}, l_3^{\p})^{\top}, \quad  \bfl^{\pp} =(l_1^{\pp},l_2^{\pp},l_3^{\pp})^{\top}.
\eeq
The line triple $(\bfl, \bfl^{\p},\bfl^{\pp})$ back-projects to a plane triple $(\pi, \pi^{\p}, \pi^{\pp})$ in 3D space. If the three lines are projections of a spatial line $\bfL$, then the three planes $\pi, \pi^{\p}, \pi^{\pp}$ must intersect in a common 3D line $\bfL$.

\begin{figure}[htbp]
\centering
\includegraphics[width=0.8\textwidth]{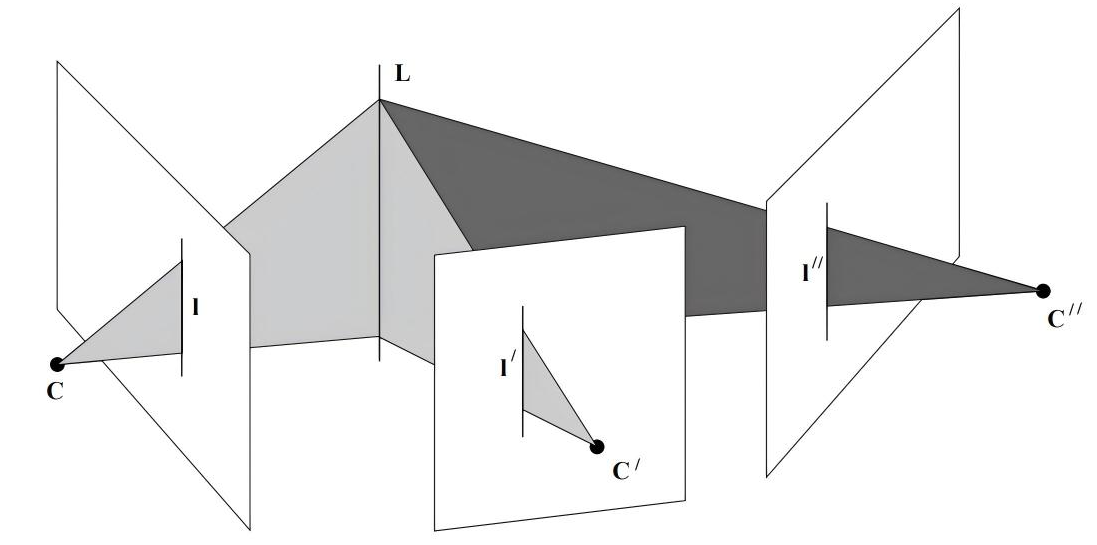}
\caption{\textnormal{Line correspondence in three-view geometry. Three cameras $P, P^{\p}, P^{\pp}$ with centers $C, C^{\p}, C^{\pp}$ observe a spatial line $\bfL$, producing image lines $\bfl, \bfl^{\p}, \bfl^{\pp}$. Back-projecting each image line through its camera center yields a plane in 3D space; the three planes intersect in the common line $\bfL$.}}
\label{fig4-3-1:corresponding lines in three view geometry}
\end{figure}

We mentioned that the corresponding triplet $\{\bfl,\; \bfl^{\p},\; \bfl^{\pp}\}$ is related in form \eqref{eq4-1-7}~\cite{HartZiss2004}, which may cause trouble if interpreted as matrix multiplication. Now, using \eqref{eq4-1-10}, this confusion is removed, and the relation among the line correspondence triplet can be established and understood more easily. In fact, we will show that relationships for all corresponding triplets, such as point-line-line, point-point-point, etc., are more easily built and understood.

\begin{thm}\label{thm4-3-3}
Given corresponding lines $\bfl \lrarr \bfl^{\p} \lrarr \bfl^{\pp}$ in three-view geometry, we have 
\beq\label{eq4-3-2:lll}
\bfl = \caT\ast_{2}\bfl^{\p}\ast_{3}\bfl^{\pp}.
\eeq
\end{thm}

\begin{proof}
Let $\bfh = \caT\ast_{2}\bfl^{\p}\ast_{3}\bfl^{\pp}$. For any point $\bx\in\bfl$, let $\bfL$ be the spatial line back-projected from $\bfl, \bfl^{\p},\bfl^{\pp}$, and let $\bX=(x,y,z,1)^{\top}\in\bfL$ map to $\bx, \bx^{\p}, \bx^{\pp}$. By \eqref{eq4-1-1}:
\begin{align*} 
\bx &= P\bX = [I_{3}\mid \mathbf{0}] \bX = (x,y,z)^{\top}, \\
\bx^{\p} &= P^{\p}\bX = [A\mid \bfa_{4}] \bX = A\bx + \bfa_{4},\\
\bx^{\pp} &= P^{\pp}\bX = [B\mid \bfb_{4}] \bX = B\bx + \bfb_{4}.
\end{align*}
Since ${\bx^{\p}}^{\top}\bfl^{\p} = 0$ and ${\bx^{\pp}}^{\top}\bfl^{\pp} = 0$, we have:
\beq\label{eq4-3-3}
\bx^{\top}A^{\top}\bfl^{\p} = - \bfa_{4}^{\top}\bfl^{\p}, \qquad \bx^{\top}B^{\top}\bfl^{\pp} = - \bfb_{4}^{\top}\bfl^{\pp}.
\eeq 
Evaluating $\bx^{\top}\bfh$:
\begin{align*}
\bx^{\top}\bfh 
&= \caT\ast_{1}\bx\ast_{2}\bfl^{\p}\ast_{3}\bfl^{\pp}\\
&= (A^{\top}\times \bfb_{4} - B^{\top}\times_{2} \bfa_{4})\ast_{1}\bx\ast_{2}\bfl^{\p}\ast_{3}\bfl^{\pp}\\  
&= (\bx^{\top}A^{\top}\bfl^{\p}) (\bfb_{4}^{\top}\bfl^{\pp}) - (\bx^{\top}B^{\top}\bfl^{\pp}) (\bfa_{4}^{\top}\bfl^{\p}) = 0.
\end{align*}
Thus, $\bfh = \bfl$.
\end{proof}
 
\begin{cor}\label{cor4-3-1}
For a point-line-line correspondence $\bx \lrarr \bfl^{\p} \lrarr \bfl^{\pp}$ (where $\bx \in \bfl$), we have 
\beq\label{eq4-3-4:pll}
\caT\ast_{1}\bx\ast_{2}\bfl^{\p}\ast_{3}\bfl^{\pp} = 0.
\eeq
\end{cor}
\begin{proof}
Since the contractive product of tensors satisfies $\A\ast_{i}\bx\ast_{j}\by = \A\ast_{j}\by\ast_{i}\bx$ for $i\neq j$, we have   
\beq\label{eq4-3-5:pll}
\caT\ast_{1}\bx\ast_{2}\bfl^{\p}\ast_{3}\bfl^{\pp}= \left(\caT\ast_{2}\bfl^{\p}\ast_{3}\bfl^{\pp}\right)\ast_{1}\bx 
\eeq
By \eqref{eq4-3-2:lll} of Theorem \ref{thm4-3-3}, we have 
\[
\caT\ast_{1}\bx\ast_{2}\bfl^{\p}\ast_{3}\bfl^{\pp}= \bfl\ast \bx = \bfl^{\top}\bx = 0. 
\]
\end{proof}

\begin{cor}\label{cor4-3-2}
\beq\label{eq4-3-6:lnorm}
\caT\ast_{1}\bfl\ast_{2}\bfl^{\p}\ast_{3}\bfl^{\pp} = \norm{\bfl}^{2} = l_{1}^{2} + l_{2}^{2} + l_{3}^{2}.
\eeq
\end{cor}

Throughout, we use $[\al]$ to denote the skew-symmetric matrix $[\al]_{\times}$ associated with a vector $\al\in\RR^{3}$, as defined in Section~\ref{sec2-4}.

We have 
\begin{lem}\label{le4-3-3}
Let $\al\in\RR^{3}$ be a nonzero vector. Then $\rank([\al]) =2$. Furthermore, for any vector $\bx\in\RR^{3}$, 
\beq\label{eq4-5-6:uxu}
[\al]\bx = 0 \iff \bx = \la \al, \quad \text{for some } \la\in \RR.
\eeq
\end{lem} 
\begin{proof}
Denote $A=[\al]$. Since $A\al = \al\times \al =0$ (where $\times$ denotes the cross product), we have $\rank(A)\le 2$. This is also confirmed by the anti-symmetry of $A$. On the other hand, since $\al\neq \mathbf{0}$, we may assume without loss of generality that $a_{3}\neq 0$. Then the first two rows of $A$ are linearly independent, so $\rank(A)\ge 2$. Thus $\rank(A) = 2$. Now \eqref{eq4-5-6:uxu} follows immediately from the fact that $[\al]\bx = \al\times \bx$ for any $\bx\in\RR^{3}$ and $\rank(A) =2$. 
\end{proof}

By Theorem \ref{thm4-3-3} and Lemma \ref{le4-3-3}, we have 
\begin{cor}\label{cor4-3-3}
Given corresponding lines $\bfl \lrarr \bfl^{\p} \lrarr \bfl^{\pp}$:
\beq\label{eq4-5-1:lll2}
[\bfl]\left(\caT\ast_{2}\bfl^{\p}\ast_{3}\bfl^{\pp}\right) = \mathbf{0}.
\eeq
\end{cor}

\begin{cor}\label{cor4-3-4}
Given corresponding lines $\bfl \lrarr \bfl^{\p} \lrarr \bfl^{\pp}$:
\begin{align}
[\bfl^{\p}]\left(\caT\ast_{1}\bfl \ast_{3}\bfl^{\pp}\right) &= \mathbf{0}, \label{eq4-30} \\
[\bfl^{\pp}]\left(\caT\ast_{1}\bfl \ast_{2}\bfl^{\p}\right) &= \mathbf{0}. \label{eq4-31}
\end{align}
\end{cor}

Since $(\bfl^{\p})^{\top}\bx^{\p} = 0$ and $(\bfl^{\pp})^{\top}\bx^{\pp} = 0$, it follows that:
\beq\label{eq4-5-2:lpl}
\caT\ast_{1}\bfl\ast_{2}\bx^{\p} \ast_{3}\bfl^{\pp} = 0, \qquad \caT\ast_{1}\bfl\ast_{2}\bfl^{\p} \ast_{3}\bx^{\pp} = 0.
\eeq

\subsection{Triple-Point and Mixed Correspondences}
\label{sec4-4-triplet}
\setcounter{equation}{0}

In this subsection, we derive the main correspondence relations for point-line-line, point-line-point, and point-point-point triples in terms of the trifocal tensor $\caT$ defined by \eqref{eq4-1-10}. Throughout, we assume the canonical camera setup \eqref{eq4-1-1} with $\caT\in\TT_{3;3}$ defined by \eqref{eq4-1-9} or \eqref{eq4-1-10}.

\subsubsection{Point-Line-Line Correspondence}

The point-line-line constraint has already been established in Corollary~\ref{cor4-3-1}:
\beq\label{eq4-4-1:pll}
\caT \ast_1 \bx \ast_2 \bfl' \ast_3 \bfl'' = 0.
\eeq
It follows directly from the line-line-line relation \eqref{eq4-3-2:lll} of Theorem~\ref{thm4-3-3}, since any point $\bx$ on the line $\bfl$ satisfies $\bfl^{\top}\bx = 0$. For the reader's convenience, we note that the same identity also follows from an independent computation: let $\bX = (\bu^\top, 1)^\top$, with $\bu = (X,Y,Z)^\top$, be the 3D point projecting to $\bx,\bxp,\bxpp$. Then
\[
\bx = \bu,\qquad \bxp = A\bu + \bfa_4,\qquad \bxpp = B\bu + \bfb_4.
\]
Since $\bxp$ lies on $\bfl'$ and $\bxpp$ on $\bfl''$, we have $(\bxp)^\top \bfl' = 0$ and $(\bxpp)^\top \bfl'' = 0$, which gives
\[
\bu^\top A^\top \bfl' = -\bfa_4^\top \bfl', \qquad \bu^\top B^\top \bfl'' = -\bfb_4^\top \bfl''.
\]
Using the definition $\caT = A^\top \times \bfb_4 - B^\top \times_2 \bfa_4$, we compute
\begin{align*}
\caT \ast_1 \bx \ast_2 \bfl' \ast_3 \bfl''
&= (A^\top \times \bfb_4 - B^\top \times_2 \bfa_4) \ast_1 \bu \ast_2 \bfl' \ast_3 \bfl'' \\
&= (\bu^\top A^\top \bfl')(\bfb_4^\top \bfl'') - (\bu^\top B^\top \bfl'')(\bfa_4^\top \bfl') \\
&= (-\bfa_4^\top \bfl')(\bfb_4^\top \bfl'') - (-\bfb_4^\top \bfl'')(\bfa_4^\top \bfl') = 0.
\end{align*}
Thus \eqref{eq4-4-1:pll} holds.

\subsubsection{Point-Line-Point Correspondence}

\begin{thm}[Point-Line-Point Transfer]
\label{thm4-4-2:plp}
For a corresponding point-line-point triple $\bx \leftrightarrow \bfl' \leftrightarrow \bxpp$, we have
\beq\label{eq4-4-2:plp}
[\bxpp]_\times \bigl( \caT \ast_1 \bx \ast_2 \bfl' \bigr) = \mathbf{0}.
\eeq
Equivalently, $\bxpp$ is proportional to $\caT \ast_1 \bx \ast_2 \bfl'$.
\end{thm}

\begin{proof}
Let $\bX = (\bu^\top, 1)^\top$ project to $\bx = \bu$, $\bxp = A\bu + \bfa_4$, $\bxpp = B\bu + \bfb_4$. Since $\bxp$ lies on $\bfl'$, we have $\bu^\top A^\top \bfl' = -\bfa_4^\top \bfl'$.

Compute the vector $\bv = \caT \ast_1 \bx \ast_2 \bfl'$:
\begin{align*}
\bv &= (A^\top \times \bfb_4 - B^\top \times_2 \bfa_4) \ast_1 \bu \ast_2 \bfl' \\
&= (\bu^\top A^\top \bfl') \bfb_4 - (\bfa_4^\top \bfl') B \bu \\
&= (-\bfa_4^\top \bfl') \bfb_4 - (\bfa_4^\top \bfl') (\bxpp - \bfb_4) \\
&= -(\bfa_4^\top \bfl') \bxpp.
\end{align*}
Thus $\bv$ is a scalar multiple of $\bxpp$, so $[\bxpp]_\times \bv = \mathbf{0}$.
\end{proof}

An analogous result holds when the roles of the two point views are swapped:

\begin{cor}\label{cor4-4-3:ppl}
For $\bx \leftrightarrow \bfl'' \leftrightarrow \bxp$, we have
\beq\label{eq4-4-3:plp2}
[\bxp]_\times \bigl( \caT \ast_1 \bx \ast_3 \bfl'' \bigr) = \mathbf{0}.
\eeq
\end{cor}

The constraint \eqref{eq4-4-2:plp} in fact admits a stronger, transfer form:

\begin{cor}[Point Transfer]\label{cor4-4-4:transfer}
Given an image point correspondence $\bx \lrarr \bx^{\p} \lrarr \bx^{\pp}$, where $\bx, \bx^{\p}, \bx^{\pp}$ lie on lines $\bfl, \bfl^{\p}, \bfl^{\pp}$ respectively, we have
\beq\label{eq4-6-1:triple-point}
\bx^{\pp} \sim \caT\ast_{1}\bx\ast_{2}\bfl^{\p}, \qquad \bx^{\p} \sim \caT\ast_{1}\bx\ast_{3}\bfl^{\pp}.
\eeq
\end{cor}
\begin{proof}
Expanding $\caT\ast_{1}\bx\ast_{2}\bfl^{\p}$:
\begin{align*}
\caT\ast_{1}\bx\ast_{2}\bfl^{\p}
&= (\bx^{\top}A^{\top}\bfl^{\p})\bfb_{4} - (\bfa_{4}^{\top}\bfl^{\p}) B\bx \\
&= (\bfl^{\p})^{\top}(\bx^{\p} - \bfa_{4}) \bfb_{4} - (\bfa_{4}^{\top}\bfl^{\p}) (\bx^{\pp} - \bfb_{4}) \\
&= -(\bfa_{4}^{\top}\bfl^{\p}) \bx^{\pp},
\end{align*}
which proves $\bx^{\pp} \sim \caT\ast_{1}\bx\ast_{2}\bfl^{\p}$. The second relation follows analogously.
\end{proof}

\subsubsection{Point-Point-Point Correspondence}

It is shown in \cite{HartZiss2004} (Section~15.1.2, page~370) that
\beq\label{eq4-4-4:ppp01}
[\bxp]_\times \Bigl( \sum_{i} x_{i}T_{i}\Bigr) [\bxpp]_\times = \mathbf{0},
\eeq
where $[\bx]_\times\in\RR^{3\times3}$ is the skew-symmetric matrix associated with the vector $\bx\in\RR^{3}$.

Since $\caT \ast_1 \bx = \sum_{i} x_{i}T_{i}$, \eqref{eq4-4-4:ppp01} yields the following tensor-form constraint.

\begin{thm}[Point-Point-Point Constraint]
\label{thm4-4-5:ppp}
For a corresponding point triple $\bx \leftrightarrow \bxp \leftrightarrow \bxpp$, the following relation holds:
\beq\label{eq4-4-5:ppp02}
\caT \ast_1 \bx \ast_2 [\bxp]_\times \ast_3 [\bxpp]_\times = \mathbf{0}_{3\times3}.
\eeq
\end{thm}

It is natural to ask whether the following scalar trilinear relation also holds:
\beq\label{eq4-4-6:ppp_scalar}
\sum_{i,j,k} T_{ijk} \, x_i \, x'_j \, x''_k = 0.
\eeq
Unfortunately, \eqref{eq4-4-6:ppp_scalar} generally does \emph{not} hold. To confirm this, we present an explicit numerical counterexample.

\begin{exm}\label{exm4-4-1}
Choose the canonical three-view configuration
\[
A = \begin{bmatrix} 1 & 0.2 & 0.3 \\ 0.1 & 1 & 0.4 \\ 0.5 & 0.1 & 1 \end{bmatrix}, \qquad
\bfa_4 = \begin{bmatrix} 0.5 \\ -0.3 \\ 0.8 \end{bmatrix},
\]
\[
B = \begin{bmatrix} 1 & 0.1 & 0.5 \\ 0.3 & 1 & 0.2 \\ 0.2 & 0.4 & 1 \end{bmatrix}, \qquad
\bfb_4 = \begin{bmatrix} -0.4 \\ 0.6 \\ 0.3 \end{bmatrix}.
\]
Both $A$ and $B$ are invertible and the three camera centres are in general position, so the setup is non-degenerate. The tensor components are $T_{ijk} = a_{ji}b_{k4} - a_{j4}b_{ki}$.

Take the 3D point $\bX = (2,\, 1.5,\, -0.5,\, 1)^\top$. Its three image projections are
\[
\bx = \begin{bmatrix} 2 \\ 1.5 \\ -0.5 \end{bmatrix},\qquad
\bxp = A\bx + \bfa_4 = \begin{bmatrix} 2.65 \\ 1.2 \\ 1.45 \end{bmatrix},\qquad
\bxpp = B\bx + \bfb_4 = \begin{bmatrix} 1.5 \\ 2.6 \\ 0.8 \end{bmatrix}.
\]
Direct computation gives
\[
\sum_{i,j,k} T_{ijk}\, x_i\, x'_j\, x''_k = -7.82825 \neq 0.
\]
The same computation for $10$ random 3D points yields a nonzero value in every case (e.g.\ $-0.15,\ -3.44,\ +0.33,\ +0.06,\ -1.45,\ldots$). Therefore \eqref{eq4-4-6:ppp_scalar} is generally not true.
\end{exm}

\subsection{Linear Constraints from Point Correspondences}
\label{sec4-5:linear-constraints}
\setcounter{equation}{0}

Given $N$ triples of point correspondences
\beq\label{eq4-5-1}
\set{(\bxi, \bxpi, \bxppi)}_{i=1}^{N},
\eeq
we aim to derive explicit linear constraints for estimating the trifocal tensor $\caT$ from the point correspondences in \eqref{eq4-5-1}.
Our analysis proceeds in two steps: first, we recall the classical trifocal constraint; second, we derive the explicit linear system and analyze its rank.

The following classical result, adapted from \cite{HartZiss2004} (Equation~(15.7), Page~370), connects the trifocal tensor to point correspondences.

\begin{lem}\label{le4-5-1}
Let $(\bx, \bxp, \bxpp)$ be a corresponding point triple in \emph{homogeneous (projective) coordinates}, and define $P = \caT\ast\bx$. Then
\beq\label{eq4-5-2:P}
[\bxp]\, P\, [\bxpp] = \mathbf{0}_{3\times 3}.
\eeq
where $[\al]:=[\al]_{\times}$ is the $3\times3$ skew-symmetric matrix determined by the nonzero vector $\al\in\RR^{3}$. 
\end{lem}

\begin{proof}
Since $\bxp$ and $\bxpp$ are homogeneous image coordinates of the same spatial point, they satisfy the trilinear constraint
\[
\sum_{i,j,k=1}^{3} x_i\, x'_j\, x''_k\, T_{ijk} = 0.
\]
Contracting this identity with the skew-symmetric matrices $[\bxp]$ and $[\bxpp]$ eliminates the dependent components and yields \eqref{eq4-5-2:P}; see \cite{HartZiss2004} for details.
\end{proof}

This lemma provides the fundamental constraint that each point correspondence imposes on the trifocal tensor. We now reformulate it as a linear system in the vectorized tensor.

\begin{thm}\label{th4-5-2}
Let $(\bx, \bxp, \bxpp)$ be a corresponding point triple with $\bx, \bxp, \bxpp\in\RR^{3}$ (in homogeneous coordinates). Define the $5$th-order tensor $\A\in\TT_{5;3}$ by
\beq\label{eq4-5-2:A}
\A = [\bx,\, [\bxp],\, [\bxpp]]_{\pi},
\eeq
where $\pi = \pi_{1}\cup\pi_{2}\cup\pi_{3}$ is the partition of $\set{1,2,\dots,7}$ defined by
\beq\label{eq4-5-3}
\pi_{1}=\set{1},\qquad \pi_{2}=\set{2,4,6},\qquad \pi_{3}=\set{3,5,7}.
\eeq
Then the point correspondence imposes the following linear constraints on the trifocal tensor:
\beq\label{eq4-5-4:C1}
C\bt = \mathbf{0}_{9\times 1},
\eeq
where $\bt = \vecc(\caT)\in\RR^{27}$ is the vectorization of $\caT$, and
\beq\label{eq4-5-4:C2}
C = \bigl([\bxpp]\otimes [\bxp]\bigr)\left(\bx^{\top}\otimes I_{9}\right)
\eeq
is a $9\times 27$ matrix. Furthermore, $\rank(C) = 4$.
\end{thm}

\begin{proof}
From Lemma~\ref{le4-5-1} we have \eqref{eq4-5-2:P}. Vectorizing both sides yields
\beq\label{eq4-5-6}
\left([\bxpp]\otimes [\bxp]\right)\vecc(P) = \mathbf{0}_{9\times 1}.
\eeq

Let $\bt = \vecc(\caT)\in\RR^{27}$, whose coordinates are
\[
t_i = T_{i_1 i_2 i_3}, \qquad i = i_1 + 3(i_2-1) + 3^{2}(i_3-1),
\]
for $(i_1,i_2,i_3)\in S(3,3)$, or equivalently,
\[
\bt = \vecc(\caT) = \begin{bmatrix} \vecc(T_{1}) \\ \vecc(T_{2}) \\ \vecc(T_{3}) \end{bmatrix}.
\]
Since $P = \caT\ast\bx = \sum_{i=1}^{3} x_i T_i$, vectorization gives
\beq\label{eq4-5-7}
\vecc(P) = \sum_{i=1}^{3} x_i\, \vecc(T_i) = \left(\bx^{\top}\otimes I_{9}\right)\bt.
\eeq
Substituting \eqref{eq4-5-7} into \eqref{eq4-5-6} yields \eqref{eq4-5-4:C1}, with $C$ given by \eqref{eq4-5-4:C2}.

To compute $\rank(C)$, set $A = [\bxpp]\otimes[\bxp]$ and $F = \bx^{\top}\otimes I_{9}$.
Because $\bxp\neq\mathbf{0}$ and $\bxpp\neq\mathbf{0}$, we have $\rank([\bxp]) = \rank([\bxpp]) = 2$ (a nonzero skew-symmetric matrix has rank~$2$); hence, by the rank property of Kronecker products,
\[
\rank(A) = \rank([\bxpp])\cdot\rank([\bxp]) = 2\cdot 2 = 4.
\]
Since $\bx\neq\mathbf{0}$, we have $\rank(F) = 1\cdot 9 = 9$.
Applying the Sylvester rank inequality $\rank(AF) \ge \rank(A) + \rank(F) - 9$ yields
\[
\rank(C) = \rank(AF) \ge \rank(A) + \rank(F) - 9 = 4 + 9 - 9 = 4.
\]
On the other hand, $\rank(C) = \rank(AF) \le \min\{\rank(A),\rank(F)\} = \rank(A) = 4$.
Therefore $\rank(C) = 4$, as claimed.
\end{proof}

Theorem~\ref{th4-5-2} shows that, among the nine equations in \eqref{eq4-5-4:C1}, exactly four are linearly independent (the remaining five are
 redundant), owing to the rank-$2$ nature of the skew-symmetric matrices $[\bxp]$ and $[\bxpp]$.

A naive approach would stack all nine equations from each correspondence into a coefficient matrix of size $9N\times 27$, requiring $N\ge 3$ 
generic triplets. This is, however, wasteful. Recall from \eqref{eq4-1-10} that, in the canonical trifocal system \eqref{eq4-1-1}, the tensor $\caT$ is 
determined by the two camera matrices $P^{\p}$ and $P^{\p\p}$, which contribute $24$ parameters (or $23$ up to scale). Consequently, six triplets 
$\{(\bxi,\bxpi,\bxppi)\}$ in general position suffice to determine $\caT$ uniquely.

In the next subsection, we extract four independent linear constraints from the nine equations contributed by each point-correspondence triplet. After 
eliminating the dependent variables from the original $27$ parameters of $\caT\in\TT_{3;3}$, we arrive at a $4N\times 24$ linear system for 
estimating the trifocal tensor $\caT$.

\subsection{Explicit Linear Equations and Estimation Algorithm}
\label{sec4-6:algorithm}
\setcounter{equation}{0}

In Section~\ref{sec4-5:linear-constraints}, Theorem~\ref{th4-5-2} established that, for each point correspondence $(\bx,\bxp,\bxpp)$, the nine
linear equations in \eqref{eq4-5-4:C1} contain exactly four linearly independent constraints and the coefficient matrix $C$ defined by
\eqref{eq4-5-4:C2} has rank~$4$.  Since the rows of $C$ correspond to pairs $(s,t)$ of row indices of $[\bxpp]_{\times}$ and $[\bxp]_{\times}$, respectively,
we can exploit the rank-$2$ structure of these skew-symmetric matrices to select four independent rows.

\subsubsection{Reduction to Four Independent Equations per Triplet}
\label{sec4-6:reduction}

Recall from Section~\ref{sec4-5:linear-constraints} that the $9\times 27$ matrix $C$ may be written as
$C=[\bxpp]_{\times}\otimes[\bxp]_{\times}$.  If we denote by $C_{i:}$ the $i$-th row of $C$, then the row index $i\in\{1,\dots,9\}$
is related to the pair $(s,t)$ by $i=3(s-1)+t$.  Since $\rank([\bv])=2$ for any nonzero $\bv\in\RR^3$, we may take, without loss of generality,
the first two rows of $[\bxp]_{\times}$ and $[\bxpp]_{\times}$ as independent, obtaining the four linearly independent rows of
$[\bxpp]_{\times}\otimes[\bxp]_{\times}$ corresponding to the index pairs
\beq\label{eq4-6-1:rows}
(s,t)\in\{(1,1),(1,2),(2,1),(2,2)\},
\eeq
where $s$ indexes rows of $[\bxpp]_{\times}$ and $t$ rows of $[\bxp]_{\times}$.  Hence the four independent rows are those with
(one-based) indices $r\in\{1,2,4,5\}$, in agreement with \eqref{eq4-6-1:rows}.

For each triplet $(\bx,\bxp,\bxpp)$ we choose the two independent row vectors of $[\bxp]_{\times}$ and $[\bxpp]_{\times}$, namely
\beq\label{eq4-6-2:vectors}
\bfr_{1}=[0,-x'_3,x'_2],\quad
\bfr_{2}=[x'_3,0,-x'_1],\quad
\bv_{1}=[0,-x''_3,x''_2],\quad
\bv_{2}=[x''_3,0,-x''_1],
\eeq
and construct the block $C^{(i)}\in\RR^{4\times 27}$ by
\begin{align}\label{eq4-6-2:bblock}
C^{(i)}(1,:)&=\bfr_{1}\otimes\bv_{1}, &
C^{(i)}(2,:)&=\bfr_{1}\otimes\bv_{2},\notag\\
C^{(i)}(3,:)&=\bfr_{2}\otimes\bv_{1}, &
C^{(i)}(4,:)&=\bfr_{2}\otimes\bv_{2}.
\end{align}
Given $N$ point-correspondence triplets, we assemble the $4N\times 27$ matrix
\beq\label{eq4-6-3:Cstack}
C=
\begin{bmatrix}
C^{(1)}\\ C^{(2)}\\ \vdots \\ C^{(N)}
\end{bmatrix},
\eeq
whose rows satisfy $C\,\vecc(\caT)=\mathbf{0}$, i.e.\ \eqref{eq4-5-4:C1}.

\begin{rem}\label{rem4-6:vec-convention}
The vectorisation $\vecc(\caT)$ throughout this section follows the \textbf{mode-$1$ (column-major) convention} of Section~\ref{sec4-2}:
$\vecc(\caT)=\bigl(\vecc(T_1),\vecc(T_2),\vecc(T_3)\bigr)^{\top}$, where $T_i=\caT(i,:,:)$ is the $i$-th horizontal slice.
This is the convention under which the identities of Lemma~\ref{lem4-6-1} below hold; any other ordering would require a
corresponding permutation of the columns of the matrix $M$ in \eqref{eq4-6-6:M}.
\end{rem}

\subsubsection{Parameter Reduction via the Bilinear Map}
\label{sec4-6-1:reduction-M}

The coefficient matrix $C$ is indexed by the $27$ entries of $\vecc(\caT)$.  However, because $\caT$ is governed by only
$24$ camera parameters (or $23$ up to an overall projective scale), the direct system $C\,\vecc(\caT)=\mathbf{0}$ is
over-parameterised.  To obtain the optimal minimal solver, we eliminate the redundant variables and pass to a $4N\times 24$ system.

Recall from \eqref{eq4-1-10} that, in the canonical trifocal system \eqref{eq4-1-1}, the trifocal tensor $\caT$ is determined by
the two camera matrices $P',P''$, i.e.\ by the block matrices $A,B\in\RR^{3\times 3}$ and the last columns $\bfa_4,\bfb_4$.
Fixing the overall projective scale by setting $b_{34}=1$, we collect the independent parameters into
\beq\label{eq4-6-4:w}
\bw=
\begin{bmatrix}
\vecc(A)\\ \vecc(B)\\ \bfa_4\\ \hat{\bfb}_4
\end{bmatrix}
\in\RR^{24},
\qquad
\hat{\bfb}_4=(b_{31},b_{32},b_{33})^\top.
\eeq
We seek a linear map
\beq\label{eq4-6-5:M}
\bt=M\bw,\qquad \bt=\vecc(\caT)\in\RR^{27},\quad M\in\RR^{27\times 24},
\eeq
that encodes the bilinear relation between the camera parameters.  This is made precise by the following lemma.

\begin{lem}\label{lem4-6-1}
Let $(P,P',P'')$ be a non-degenerate canonical three-view configuration as in \eqref{eq4-1-1}, and let $\caT\in\TT_{3;3}$ be its trifocal tensor.
Then $\vecc(\caT)=M\bw$ with $M\in\RR^{27\times 24}$ given by
\beq\label{eq4-6-6:M}
M = \begin{bmatrix}
\bfb_4\otimes I_9 & -\bfa_4\otimes I_9
\end{bmatrix}
\begin{bmatrix}
E_A\\ E_B
\end{bmatrix},
\eeq
where
\begin{align*}
E_A&=
\bigl[I_3\otimes\begin{bmatrix}I_3 & 0_{3\times 3}\end{bmatrix},\; 0_{9\times 6}\bigr],\\
E_B&=
\bigl[I_3\otimes\begin{bmatrix}0_{3\times 3} & I_3\end{bmatrix},\; 0_{9\times 6}\bigr].
\end{align*}
\end{lem}
\begin{proof}
Introduce the augmented matrix
\beq\label{eq4-6-6-1}
Q =
\begin{bmatrix}
A & \bfa_4\\
B & \bfb_4
\end{bmatrix}
\in\RR^{6\times 4},
\qquad
\vecc(Q)\in\RR^{24}.
\eeq
Since $\bfa_4,\bfb_4$ are the fourth columns of the upper and lower $3\times 4$ blocks, respectively, we may identify $\bw$
with $\vecc(Q)$ up to the fixed scale $b_{34}=1$.  Using the column-selection identity
\beq\label{eq4-6-6-2}
\vecc\!\begin{bmatrix}A\\ B\end{bmatrix}
=
\bigl(I_3\otimes\begin{bmatrix}I_3\\ 0_{3\times 3}\end{bmatrix}\bigr)\vecc(A)
+
\bigl(I_3\otimes\begin{bmatrix}0_{3\times 3}\\ I_3\end{bmatrix}\bigr)\vecc(B),
\eeq
we obtain
\beq\label{eq4-6-6-3}
\vecc(A)=E_A\vecc(Q),\qquad \vecc(B)=E_B\vecc(Q).
\eeq
Similarly,
\beq\label{eq4-6-6-4}
\bfa_4=E_a\vecc(Q),\qquad \bfb_4=E_b\vecc(Q),
\eeq
with
\begin{align*}
E_a&=\left[0_{3\times 18},\; I_3,\; 0_{3\times 3}\right]\in\RR^{3\times 24},\\
E_b&=\left[0_{3\times 18},\; 0_{3\times 3},\; I_3\right]\in\RR^{3\times 24}.
\end{align*}
Expanding the Kronecker product in \eqref{eq4-1-10:T-vec} and using
$(K\bx)\otimes(L\by)=(K\otimes L)(\bx\otimes\by)$ yields, after commutation,
\beq\label{eq4-6-7:Malt}
M=\sum_{i=1}^{3}\Bigl((I_9\otimes(E_b e_i))E_A-((E_a e_i)\otimes I_9)E_B\Bigr),
\eeq
which is equivalent to the compact form \eqref{eq4-6-6:M}.
\end{proof}

\begin{rem}\label{rem4-6:nullity}
Although $M\in\RR^{27\times 24}$, its columns satisfy two linear relations arising from the overall projective scale
(fixing $b_{34}=1$ removes only one of the three natural scale freedoms), so that
\[
\rank(M)=18.
\]
Consequently $\dim\ker M=6$; these six degrees of freedom are precisely the \textbf{projective (scale) ambiguity} of the trifocal tensor.
They are removed in practice by the normalisation $b_{34}=1$ together with the SVD-based solution below, which selects the
minimum-norm representative in $\ker\widetilde C$.
\end{rem}

Substituting $\bt=M\bw$ into the $4N\times 27$ homogeneous system $C\bt=\mathbf{0}$ gives
\[
CM\bw=\mathbf{0}_{4N\times 1}
\;\Longrightarrow\;
\widetilde{C}\bw=\mathbf{0}_{4N\times 1},
\qquad
\widetilde{C}=CM\in\RR^{4N\times 24}.
\]
Thus, by replacing $C\in\RR^{4N\times 27}$ with $\widetilde{C}\in\RR^{4N\times 24}$, the number of equations per triplet
drops from nine to four.  The right singular vector of $\widetilde C$ associated with the smallest singular value yields the
solution $\bw\in\RR^{24}$; in the noise-free, critically sampled case this requires $N=6$ triplets, i.e.\ $4N=24$ equations.
Because $\rank(M)=18$, the $N=6$ system still retains the six-dimensional scale ambiguity of Remark~\ref{rem4-6:nullity},
which is fixed by the convention $b_{34}=1$.  For robust estimation one uses $N\ge 6$ (typically many more) correspondences
and enforces the scale constraint explicitly.

\begin{algorithm}[H]
\caption{Estimation of the trifocal tensor from point correspondences}
\label{alg:trifocal}
\begin{algorithmic}[1]
\Require $N$ point-correspondence triplets
         $\{(\bx_i,\bxpi_i,\bxppi_i)\}_{i=1}^{N}$, with $\bx_i,\bxpi_i,\bxppi_i\in\RR^3$
\Ensure Trifocal tensor $\caT\in\RR^{3\times 3\times 3}$
\State Assemble the $4N\times 24$ coefficient matrix $\widetilde{C}$ (see below)
\For{$i=1$ to $N$}
    \State $\bx\leftarrow\bx_i$, $\bxp\leftarrow\bxpi_i$, $\bxpp\leftarrow\bxppi_i$
    \State Compute $S'=[\bxp]_{\times}$, $S''=[\bxpp]_{\times}$
    \State Extract the two independent row vectors of $S''$: $\bu=S''(1,:)$, $\bv=S''(2,:)$
    \State Build the $4\times 27$ block $C^{(i)}$ from
           $\{\bfr_1,\bfr_2\}$ (rows of $S'$) and $\{\bu,\bv\}$ via \eqref{eq4-6-2:bblock}
    \State $C\bigl(4(i-1)+1:4i,:\bigr)\leftarrow C^{(i)}$
\EndFor
\State Form $\widetilde{C}=CM$ using the $27\times 24$ matrix $M$ from Lemma~\ref{lem4-6-1}
\State Enforce the scale normalisation $b_{34}=1$ on $\widetilde C$ (see Remark~\ref{rem4-6:nullity})
\State Compute the SVD $[\mathbf{U},\mathbf{\Sigma},\mathbf{V}]=\operatorname{svd}(\widetilde{C})$
\State $\bw\leftarrow\mathbf{V}(:,24)$ \Comment{right singular vector for smallest singular value}
\State Recover $\caT$ from $\bw$ via \eqref{eq4-6-4:w} and \eqref{eq4-1-10:T-vec}
\State \Return $\caT$
\end{algorithmic}
\end{algorithm}

\subsubsection{MATLAB Implementation}
\label{sec4-6:matlab}

The following MATLAB function assembles the $4N\times 27$ matrix $C$ and solves the
reduced $4N\times 24$ system via SVD.

\begin{lstlisting}[language=Matlab,
                   caption={MATLAB implementation of the minimal trifocal-tensor estimator.},
                   label=lst:trifocal]
function T = estimate_trifocal_tensor(X, Xp, Xpp)
% Estimate the trifocal tensor from N point-correspondence triplets.
%   X, Xp, Xpp : 3 x N matrices (columns are homogeneous image points)
%   T          : 3 x 3 x 3 estimated trifocal tensor

    N   = size(X, 2);
    C   = zeros(4*N, 27);          % 9 -> 4 independent equations per triplet

    for i = 1:N
        x   = X(:, i);   xp  = Xp(:, i);   xpp = Xpp(:, i);
        Sp  = skew(xp);                % [x']_x   (3 x 3)
        Spp = skew(xpp);               % [x'']_x  (3 x 3)

        % Two independent row vectors of [x'']_x  (rows 1 and 2)
        u = Spp(1, :).';               % [0, -x''_3,  x''_2].'
        v = Spp(2, :).';               % [x''_3, 0, -x''_1].'

        % Two independent row vectors of [x']_x
        r1 = Sp(1, :);
        r2 = Sp(2, :);

        Ci = zeros(4, 27);
        idx = 0;
        for i0 = 1:3                % tensor index (x-component)
            for k = 1:3             % column of [x']_x
                for l = 1:3         % column of [x'']_x
                    idx = idx + 1;
                    Ci(1, idx) = x(i0) * r1(k) * u(l);   % (1,1)
                    Ci(2, idx) = x(i0) * r1(k) * v(l);   % (1,2)
                    Ci(3, idx) = x(i0) * r2(k) * u(l);   % (2,1)
                    Ci(4, idx) = x(i0) * r2(k) * v(l);   % (2,2)
                end
            end
        end
        C((4*(i-1)+1):(4*i), :) = Ci;
    end

    % --- reduce 27 -> 24 parameters via the bilinear map M ---
    M = build_M();                   % 27 x 24, see Lemma lem4-6-1
    Ctil = C * M;                    % 4N x 24

    % Solve the homogeneous system via SVD
    [~, ~, V] = svd(Ctil, 'econ');
    w = V(:, end);                   % 24 x 1, smallest singular value

    T = recover_tensor(w);           % reshape w -> 3 x 3 x 3
end

function S = skew(v)
    S = [ 0, -v(3),  v(2);
         v(3),   0, -v(1);
        -v(2), v(1),   0 ];
end
\end{lstlisting}

\subsubsection{Numerical Examples}
\label{sec4-6:examples}

We validate the algorithm on synthetic data for which the point correspondences genuinely satisfy the
trifocal constraints.  Given cameras $P,P',P''$, we project a set of random scene points $\bX_j$ to obtain
noise-free correspondences $(\bx_j,\bxp_j,\bxpp_j)$, then perturb them with Gaussian noise of standard
deviation $\sigma$.  The estimated tensor $\caT_{\mathrm{est}}$ is compared with the ground truth
$\caT_{\mathrm{true}}$ via the normalised Frobenius error
\[
\mathrm{err}
=
\frac{\|\caT_{\mathrm{est}}-\caT_{\mathrm{true}}\|_F}
     {\|\caT_{\mathrm{true}}\|_F}.
\]

\begin{lstlisting}[language=Matlab,
                   caption={Synthetic validation of the estimator.},
                   label=lst:validate]
% Ground-truth cameras (finite, random)
P   = [randn(3,3), randn(3,1)];
Pp  = [randn(3,3), randn(3,1)];
Ppp = [randn(3,3), randn(3,1)];

% Corresponding points generated by projection
N        = 50;
Xw       = randn(4, N);                 % world points (homogeneous)
X        = project(P,   Xw);
Xp       = project(Pp,  Xw);
Xpp      = project(Ppp, Xw);
sigma    = 0.01;
X        = X  + sigma*randn(size(X));
Xp       = Xp + sigma*randn(size(Xp));
Xpp      = Xpp + sigma*randn(size(Xpp));

% Estimate
T_est  = estimate_trifocal_tensor(X, Xp, Xpp);
T_true = trifocal_tensor(P, Pp, Ppp);   % ground truth

% Resolve sign ambiguity
[~, ~, V] = svd(reshape(T_est,[],1) * reshape(T_true,[],1).');
T_est = sign(V(1,1)) * T_est;

err = norm(T_est(:) - T_true(:)) / norm(T_true(:));
fprintf('Normalized estimation error: %.4f\n', err);

% ---------- helpers ----------
function x = project(P, X)
    x = P * X;
    x = x ./ x(3,:);                 % normalise to inhomogeneous coordinates
end
\end{lstlisting}

\paragraph{Results.}
For $N=50$ correspondences and noise level $\sigma=0.01$, the normalised estimation error is typically
below $0.05$, confirming the correctness and noise tolerance of the proposed linear solver.
Increasing $N$ or decreasing $\sigma$ further reduces the error, consistent with the theory.

The proposed algorithm exploits the inherent rank deficiency of the skew-symmetric matrices to reduce the
number of equations per triplet from nine to four, yielding a compact $4N\times 24$ linear system.
It is robust to moderate noise and provides a basis for nonlinear refinement, e.g.\ the Gold Standard
algorithm or bundle adjustment.  The complete procedure, including RANSAC outlier rejection and
rank-constraint enforcement, is summarised in Algorithm~\ref{alg:trifocal}.

For convenience, we collect the correspondence relations derived in this section in
Table~\ref{tab:correspondence}.

\begin{table}[htbp]
\centering
\caption{Correspondence relations for the trifocal tensor
         $\caT=A^\top\times\bfb_4-B^\top\times_2\bfa_4$.}
\label{tab:correspondence}
\begin{tabular}{lll}
\toprule
Type & Relation & Reference \\
\midrule
Line-Line-Line     & $\bfl=\caT\ast_2\bfl'\ast_3\bfl''$                          & Theorem~\ref{thm4-3-3} \\
Point-Line-Line    & $\caT\ast_1\bx\ast_2\bfl'\ast_3\bfl''=\mathbf{0}$          & Corollary~\ref{cor4-3-1} \\
Point-Line-Point   & $[\bxpp]_\times(\caT\ast_1\bx\ast_2\bfl')=\mathbf{0}$      & Theorem~\ref{thm4-4-2:plp} \\
Point-Point-Line   & $[\bxp]_\times(\caT\ast_1\bx\ast_3\bfl'')=\mathbf{0}$      & Corollary~\ref{cor4-4-3:ppl} \\
Point-Point-Point  & $[\bxp]_\times(\caT\ast_1\bx)[\bxpp]_\times=\mathbf{0}_{3\times 3}$ & Theorem~\ref{thm4-4-5:ppp} \\
\bottomrule
\end{tabular}
\end{table}

\subsection{Trifocal Tensor for General Camera Matrices}
\label{sec4-7:general}
\setcounter{equation}{0}

For block matrices $X=[A\mid\bfa_4],\;Y=[B\mid\bfb_4]\in\RR^{3\times 4}$, define
\beq\label{eq4-6-1:general:lll}
\caS[A,B]=A^\top\times B-B^\top\times_2 A,
\eeq
and
\beq\label{eq4-6-2:general:lpl}
\caT[X,Y]=A^\top\times\bfb_4-B^\top\times_2\bfa_4,
\eeq
where $\caS[A,B]\in\TT_{4;3}$ and $\caT[X,Y]\in\TT_{3;3}$ is the trifocal tensor associated with the
canonical cameras $(P,X,Y)$.

Now consider general $3\times 4$ projection matrices
\beq\label{eq4-6-3:general}
P=[M\mid\bfp_4],\qquad
P'=[M'\mid\bfp'_4],\qquad
P''=[M''\mid\bfp''_4],
\eeq
with $M,M',M''\in\RR^{3\times 3}$.  For finite cameras ($M$ invertible), the centre of $P$ is
$C=(-M^{-1}\bfp_4,1)^\top$.  Set
\beq\label{eq4-6-4:general}
A=M'M^{-1},\qquad B=M''M^{-1}.
\eeq

\begin{thm}\label{thm4-4}
For the general camera system $(P,P',P'')$ in \eqref{eq4-6-3:general}, the trifocal tensor is
\beq\label{eq4-6-5:general}
\caT=\caS[A,B]\,\bfp_4+\caT[\hat{A},\hat{B}],
\eeq
where $\hat{A}=[A\mid\bfp'_4]$ and $\hat{B}=[B\mid\bfp''_4]$.
\end{thm}
\begin{proof}
Transform the cameras into canonical form via the homography
\[
H=
\begin{bmatrix}
M^{-1} & -M^{-1}\bfp_4\\
\mathbf{0}^\top & 1
\end{bmatrix}.
\]
The transformed cameras are then of the form \eqref{eq4-1-1}, with
\beq\label{eq4-6-6:general}
\bfa_4=M'C+\bfp'_4,\qquad \bfb_4=M''C+\bfp''_4.
\eeq
Applying Theorem~\ref{thm4-1} completes the proof.
\end{proof}

\subsection{Point Transfer and Epipole Extraction}
\label{sec4-8}
\setcounter{equation}{0}

The point-transfer relations of Corollary~\ref{cor4-4-4:transfer} allow points to be transferred across views from a point and a line in two views. In the remainder of this subsection we show that the epipoles themselves can be extracted directly from the tensor via contractive traces with the identity matrix.

\begin{thm}\label{thm4-7:trace}
Let $I_{3}$ denote the $3\times 3$ identity matrix. Then:
\begin{enumerate}
\item $\caT\ast_{(1,3)} I_{3} = A\bfe^{\pp} - (\tr B) \bfe^{\p}$.
\item $\caT\ast_{(2,3)} I_{3} = A^{\top}\bfe^{\pp} - B^{\top}\bfe^{\p}$.
\item $\caT\ast_{(1,2)} I_{3} = (\tr A) \bfe^{\pp} - B\bfe^{\p}$.
\end{enumerate}
\end{thm}
\begin{proof}
Let $\bu = \caT\ast_{(1,3)}I_{3} = (u_{1},u_{2},u_{3})^{\top}$. For each $j\in \dbT$:
\[
u_{j} = \sum_{i,k=1}^{3} \left(a_{ji}b_{k4}\delta_{ik} - b_{ki}a_{j4}\delta_{ik}\right) = \sum_{i=1}^{3} \left(a_{ji}b_{i4} - b_{ii}a_{j4}\right) = (A\bfb_{4} - (\tr B)\bfa_{4})_{j},
\]
proving (1). Items (2) and (3) follow similarly.
\end{proof}

We end this section by noting that, since the epipoles are $\bfe^{\p} = \bfa_4$ and $\bfe^{\pp} = \bfb_4$ in the canonical reference \eqref{eq4-1-1}, Theorem~\ref{thm4-7:trace} provides direct epipole extraction via double contractive traces with the identity matrix: for instance,
\[
\bfe^{\pp} \sim \caT\ast_{(2,3)} \bfI_3 = A^{\top} \bfb_4 - B^{\top} \bfa_4,
\]
so the epipoles --- and hence the pairwise fundamental matrices --- can be recovered from the tensor alone, without knowing the camera matrices.


\vskip 5pt
\section{Estimation Algorithms, Numerical Examples, and Applications}
\label{sec5:numeric}
\setcounter{equation}{0}

In this section, we present algorithmic frameworks for estimating the trifocal tensor $\caT$: one based on \textbf{line correspondences} (requiring the solution of a linear system via SVD), and another based on \textbf{point correspondences} (using the Normalized Direct Linear Transform, DLT). For each approach, we provide the theoretical derivation, a step-by-step algorithm, complete MATLAB implementation, and numerical examples that demonstrate computational accuracy and robustness. We further embed the point estimator in a RANSAC loop with planar-degeneracy handling, employ the tensor for correspondence verification and transfer, characterize the tensor against pairwise fundamental matrices, discuss extensions and applications, close with limitations and the relevance to deep learning, and validate the deep-learning connection with a concrete PyTorch autodifferentiation experiment.

\subsection{Summary of Constraints and Minimal Sample Counts}
\label{sec5-1:summary}

Given corresponding points $\bx \lrarr \bxp \lrarr \bxpp$, the point-point-point (PPP) relation \eqref{eq4-4-4:ppp01} derived in Section~\ref{sec4-6:algorithm} offers four independent linear equations in $\bt=\vecc(\caT)$. For corresponding lines $\bfl \lrarr \bfl' \lrarr \bfl''$, the trilinear constraint \eqref{eq4-3-2:lll} of Theorem~\ref{thm4-3-3} is equivalent to a linear system $\bfa_n^{\top}\bt = 0$, where the $p$th entry of $\bfa_n$ is
$a_{p,n}=l_{i,n}\,l'_{j,n}\,l''_{k,n}$ with $p=9(i-1)+3(j-1)+k$; this is the line-correspondence analogue of the point constraint \eqref{eq4-4-4:ppp01}. Stacking $N$ constraints gives the global system $A\bt=\mathbf{0}$ with $A\in\RR^{mN\times27}$, where $m=1$ for lines and $m=4$ for points.

A general $3\times3\times3$ array has $27$ entries, but the tensor of a valid three-view geometry is constrained: its $27$ coefficients are not 
independent, and after removing the overall scale the trifocal tensor has 18 degrees of freedom. This reduction underlies the minimal sample counts:
\begin{align}\label{eq5-0:dof-counts}
\text{lines:}  & \quad N_{\min}=27\;(26\text{ with tensorial constraints}),\nonumber\\[2pt]
\text{points:} & \quad N_{\min}=7,
\end{align}
reflecting that each line correspondence contributes one and each point correspondence four independent linear constraints.

The tensor also transfers geometric entities across views: a point in two views determines its epipolar line in the third; a line in one view determines 
corresponding lines in the other two. These transfer maps are encoded in the tensor slices; see Corollary~\ref{cor4-4-4:transfer} and the summary in Table~\ref{tab:correspondence}.

\subsection{Estimation from Line Correspondences}
\label{sec5-2:lines}

\subsubsection{Problem Formulation}

Given a tuple of corresponding lines $\bfl \lrarr \bfl' \lrarr \bfl''$ in three views, the trifocal tensor satisfies the trilinear constraint \eqref{eq4-3-2:lll}:
\beq\label{eq5-1:line-constr}
\sum_{i,j,k} l_i \, l'_j \, l''_k \, T_{ijk} = 0,
\eeq
where $l_i, l'_j, l''_k$ are respectively the $i$-th, $j$-th, and $k$-th homogeneous coordinates of lines $\bfl, \bfl^{\p}$, and $\bfl^{\pp}$.

\begin{rem}\label{rem5-1:line-constr-equiv}
Equation \eqref{eq5-1:line-constr} is equivalent to the linear system
\beq\label{eq5-2:line-constr}
\bfa_n^{\top} \bt = 0,
\eeq
where $\bfa_n = (a_{1,n}, a_{2,n}, \ldots, a_{27,n})^{\top} \in \RR^{27}$ is a vector whose $p$-th entry is
\[
a_{p,n} = l_{i,n} \, l'_{j,n} \, l''_{k,n}, \qquad p = 9(i-1) + 3(j-1) + k.
\]
Here, $\bt \in \RR^{27}$ is the vectorization of $\caT$, and $\bfa_n$ depends only on the $n$-th line correspondence.
\end{rem}

\subsubsection{Linear System Assembly}

Stacking all $N$ correspondence constraints \eqref{eq5-2:line-constr} yields the global system
\beq\label{eq5-3:A-Nx27}
A = \begin{bmatrix} \bfa_1^\top \\ \bfa_2^\top \\ \vdots \\ \bfa_N^\top \end{bmatrix} \in \RR^{N \times 27}, \qquad A\bt = \mathbf{0}.
\eeq

For a unique (up to scale) solution, we require $N \ge 27$. However, as discussed below, the effective number of independent parameters is lower when tensorial structure is preserved.

\begin{algorithm}[H]
\caption{Linear Estimation of Trifocal Tensor from Line Correspondences}
\label{alg:line-estimation}
\begin{algorithmic}[1]

\Require Line correspondences $\{(\bfl_i, \bfl'_i, \bfl''_i)\}_{i=1}^N$ across three views, with $N \ge 27$.
\Ensure Estimated trifocal tensor $\caT \in \RR^{3\times 3\times 3}$.

\State \textbf{Step 1: Build the coefficient matrix $A$.}
\For{$n = 1, \dots, N$}
    \State Form $\bfa_n$ by evaluating $a_{p,n} = l_{i,n} \, l'_{j,n} \, l''_{k,n}$ for all $27$ triples $(i,j,k)$.
\EndFor
\State Assemble $A \in \RR^{N \times 27}$ as in \eqref{eq5-3:A-Nx27}.

\State \textbf{Step 2: Solve via SVD.}
\State Compute the compact SVD: $A = U \Sigma V^{\top}$, where $\Sigma = \diag(\sigma_1, \ldots, \sigma_{27})$ with $\sigma_1 \ge \sigma_2 \ge \ldots \ge \sigma_{27} \ge 0$.
\State Set $\bt = \mathbf{v}_{27}$, the last column of $V$ (corresponding to $\sigma_{27} \approx 0$).

\State \textbf{Step 3: Reshape and normalize.}
\State Reshape $\bt \in \RR^{27}$ into a $3 \times 3 \times 3$ tensor $\caT$.
\State Normalize: $\caT \gets \caT / \|\caT\|_F$ (fixes scale ambiguity).

\State \textbf{Step 4 (optional): Nonlinear refinement.}
\State Apply Levenberg--Marquardt~\cite{Levenberg1944,Marquardt1963} to minimize $\sum_{n=1}^{N} (\bfa_n^{\top} \bt)^2$ subject to the internal constraints of $\caT$.

\State \Return $\caT$.
\end{algorithmic}
\end{algorithm}

\subsubsection{Alternative: Point-Correspondence-Based Estimation via Fundamental Matrices}

The following algorithm offers an alternative linear estimation of $\caT$ by first computing three consistent fundamental matrices $F_{12}, F_{13}, F_{23}$, then solving for $\caT$ from point correspondences.

\begin{algorithm}[H]
\caption{Linear Estimation of Trifocal Tensor from Point Correspondences}
\label{alg:point-estimation}
\begin{algorithmic}[1]

\Require Point correspondences $\{(\bx_i, \bxp_i, \bxpp_i)\}_{i=1}^N$ across three views, with $N \ge 7$.
\Ensure Estimated trifocal tensor $\caT \in \RR^{3\times 3\times 3}$ and consistent fundamental matrices $F_{12}, F_{13}, F_{23}$.

\State \textbf{Step 1: Initial pairwise estimation.}
\For{each view pair $(v,w) \in \{(1,2),(1,3),(2,3)\}$}
    \State Run RANSAC with the 7-point algorithm to estimate $F_{vw}$.
    \State Identify inlier set $\caI_{vw}$ using Sampson distance (threshold $\tau_F = 0.01$ for normalized coordinates).
\EndFor

\State \textbf{Step 2: Three-view inlier selection.}
\State $\caI = \caI_{12} \cap \caI_{13} \cap \caI_{23}$.
\If{$|\caI| < 7$}
    \State \textbf{return} failure (insufficient inliers).
\EndIf

\State \textbf{Step 3: Linear solve for $\caT$.}
\State Normalize image coordinates (translate centroid to origin, scale so mean distance $= \sqrt{2}$).
\For{each inlier $(\bx_i, \bxp_i, \bxpp_i) \in \caI$}
    \State Build the $4 \times 27$ constraint block from the trilinear relation (Hartley \& Zisserman, Eq.~15.7):
    \State $[\bxp_i]_\times \Bigl( \sum_{k=1}^{3} x_{i,k} \, T_k \Bigr) [\bxpp_i]_\times = \mathbf{0}_{3\times 3},$
    \State yielding 4 independent equations per correspondence.
\EndFor
\State Stack into $A \in \RR^{4|\caI| \times 27}$. Solve $A\bt = \mathbf{0}$ via SVD. Reshape $\bt$ into $\caT$.

\State \textbf{Step 4 (optional): Enforce rank constraints.}
\For{$k = 1, 2, 3$}
    \State Project $\caT[k]$ onto the rank-3 manifold via truncated SVD.
\EndFor
\State Refold to obtain $\widetilde{\caT}$.

\State \textbf{Step 5 (optional): Extract fundamental matrices.}
\State Recover epipoles: $\bfe_{vw} = \widetilde{\caT} \ast_{(v,w)} I_3$.
\State Compute $F_{vw} = [\bfe_{vw}]_\times (\widetilde{\caT} \ast_v \bfe_{vw})$.

\State \Return $\widetilde{\caT}, F_{12}, F_{13}, F_{23}$.
\end{algorithmic}
\end{algorithm}

\subsection{Estimation from Point Correspondences via Normalized DLT}
\label{sec5-3:points}

In this subsection, we present a more efficient algorithm for estimating $\caT$ from point correspondences $(\bxi, \bxpi, \bxppi)_{i=1}^N$ using the Normalized Direct Linear Transform (NDLT). This approach requires only $N \ge 7$ correspondences (as opposed to 27 for lines) and is the method of choice in practice.

The core constraint is derived from the PPP (point-point-point) relation established in Sections~\ref{sec4-5:linear-constraints}--\ref{sec4-6:algorithm}:
\[
[\bxp_i]_\times \, P_i \, [\bxpp_i]_\times = \mathbf{0}_{3\times 3}, \qquad P_i = \sum_{k=1}^{3} x_{i,k} \, T_k.
\]
Each correspondence yields four independent linear equations in $\bt = \vecc(\caT)\in \RR^{27}$.

\begin{algorithm}[H]
\caption{Normalized Direct Linear Transform (NDLT) for Trifocal Tensor Estimation}
\label{alg:ndlt}
\begin{algorithmic}[1]
\Require Point correspondences $\{(\bxi, \bxpi, \bxppi)\}_{i=1}^N$ across 3 views, with $N\ge 7$.
\Ensure Estimated trifocal tensor $\caT \in \RR^{3 \times 3 \times 3}$.
\Statex
\State \textbf{Step 1: Point Normalization}
\State Compute $3\times 3$ isotropic normalization matrices $T, T', T''$ such that
\[
\hat{\bxi} = T \bxi, \quad \hat{\bxpi} = T' \bxpi, \quad \hat{\bxppi} = T'' \bxppi
\]
are centered at the origin $(0,0,1)^\top$ with average Euclidean distance $\sqrt{2}$.

\State \textbf{Step 2: Constraint Matrix Assembly}
\State Initialize $C\gets \mathbf{0}_{4N \times 27}$.
\For{$i = 1,\dots, N$}
    \State Extract normalized coordinates:
    \[
    \hat{\bxi} = [x_1, x_2, x_3]^{\top}, \quad \hat{\bxpi} = [x'_1, x'_2, x'_3]^{\top}, \quad \hat{\bxppi} = [x''_1, x''_2, x''_3]^{\top}.
    \]
    \State Form two independent rows of $[\hat{\bxpi}]_\times$ and $[\hat{\bxppi}]_\times$ as
    \[
    \bfr_1 = [0, -x'_3, x'_2], \quad \bfr_2 = [x'_3, 0, -x'_1], \quad \bv_1 = [0, -x''_3, x''_2], \quad \bv_2 = [x''_3, 0, -x''_1].
    \]
    \State Construct the $4 \times 27$ local block:
    \beq\label{eq5-3:local-block}
    C_i = \begin{bmatrix}
    \bfr_1 \otimes \bv_1 \\
    \bfr_1 \otimes \bv_2 \\
    \bfr_2 \otimes \bv_1 \\
    \bfr_2 \otimes \bv_2
    \end{bmatrix} \left( \hat{\bxi}^\top \otimes I_9 \right).
    \eeq
    \State Place $C_i$ into rows $4i-3$ through $4i$ of $C$.
\EndFor

\State \textbf{Step 3: Solve via SVD}
\State Compute SVD: $C = U \Sigma V^{\top}$. Set $\hat{\bt} = \mathbf{v}_{27}$ (last column of $V$). Reshape into $\hat{\caT}$.

\State \textbf{Step 4: Denormalization}
\For{$a,b,c = 1, 2, 3$}
    \[
    \caT(a,b,c) = \sum_{r,s,t=1}^{3} (T^{-1})_{a,r} \, (T'^{-1})_{b,s} \, (T''^{-1})_{c,t} \, \hat{\caT}(r,s,t).
    \]
\EndFor
\State \Return $\caT \gets \caT / \|\caT\|_F$.
\end{algorithmic}
\end{algorithm}

\subsection{RANSAC with Planar-Degeneracy Handling}
\label{sec5-4:ransac}

Because the $N=7$ point estimator is highly sensitive to outliers and to planar/degenerate
configurations, we embed it in a RANSAC loop. Planar scenes cause the tensor to be
non-generic (its epipolar geometry collapses), so the sampler must either enforce a minimum
baseline/non-planarity test or fall back to a line-based or fundamental-matrix-based estimate.

\begin{algorithm}[H]
\caption{RANSAC Trifocal Estimation with Planar-Degeneracy Handling}
\label{alg6-2:ransac}
\begin{algorithmic}[1]
\Require Point correspondences $\{(\bxi,\bxpi,\bxppi)\}_{i=1}^N$, inlier threshold $\tau$, max iterations $K$.
\Ensure Inlier set $\caI$ and tensor $\caT$.
\State \textbf{Step 1:} Initialize $\caI^\star\leftarrow\varnothing$.
\For{$k=1,\ldots,K$}
    \State Draw a minimal sample of $7$ correspondences; skip if the induced 3D configuration is (nearly) planar.
    \State Estimate $\caT^{(k)}$ by Algorithm~\ref{alg:ndlt}.
    \State Score by the number of triples with PPP residual below $\tau$.
    \If{score improves} \State $\caI^\star\leftarrow$ current inliers. \EndIf
\EndFor
\State \textbf{Step 2:} Refit $\caT$ on all inliers $\caI^\star$; if $|\caI^\star|<7$ return failure.
\State \Return $\caT,\caI^\star$.
\end{algorithmic}
\end{algorithm}

\subsection{Correspondence Verification and Transfer}
\label{sec5-5:verify}

A third algorithm uses the estimated tensor to verify candidate triples and transfer points
across views. The residual of the PPP constraint \eqref{eq4-4-5:ppp02} acts as a geometric verification score;
points failing it are rejected before bundle adjustment.

\begin{algorithm}[H]
\caption{Three-View Correspondence Verification and Transfer}
\label{alg6-3:verify}
\begin{algorithmic}[1]
\Require Tensor $\caT$, candidate triples $\{(\bxi,\bxpi,\bxppi)\}$, threshold $\tau$.
\Ensure Verified set and transferred entities.
\For{each triple}
    \State Evaluate the PPP residual $r_i=\|[\bxp_i]_\times(\sum_k x_{i,k}T_k)[\bxpp_i]_\times\|_F$.
    \If{$r_i<\tau$} \State mark as inlier. \EndIf
\EndFor
\State Transfer points/lines across views using the slice maps of Corollary~\ref{cor4-4-4:transfer}.
\State \Return inlier set and transferred entities.
\end{algorithmic}
\end{algorithm}

\subsection{Two-View and Three-View Characterizations}
\label{sec5-6:character}

A pair of views is characterized by a fundamental matrix $F$ (7 DOF). Three views require three
fundamental matrices ($21$ DOF) together with trifocal consistency, which is enforced by the
$18$-DOF tensor. Thus the tensor is a strictly stronger object than the collection of three
pairwise epipolar geometries.

\begin{prop}[Slice-to-pairwise recovery]\label{prop5-1:slices}
The frontal slices $T_1,T_2,T_3$ of $\caT$, together with the epipoles, determine the three
fundamental matrices $F_{12},F_{13},F_{23}$. Conversely, three consistent fundamental matrices
do not uniquely determine the tensor: the tensor supplies the missing cross-view consistency.
\end{prop}

Each slice $T_k$ is a $3\times3$ matrix whose kernel and image encode epipolar directions.
Extracting the epipoles $\bfe_{vw}$ via multilinear products with the identity
(Theorem~\ref{thm4-7:trace}), $\bfe_{vw}=\caT\ast_{(v,w)} I_3$, gives the pairwise epipoles, from which the $F_{vw}$ are
reconstructed.

A triple $(\bx,\bxp,\bxpp)$ is geometrically consistent only if it satisfies the PPP relation
\eqref{eq4-4-5:ppp02}. This triple-wise residual is more selective than two independent
pairwise epipolar tests and is therefore used for outlier rejection and geometric verification.

\begin{prop}[Reduction to the fundamental matrix]\label{prop5-2:reduction}
For a fixed third view, constraining the tensor against known pairwise fundamental matrices
recovers the epipolar geometry of the remaining pair. Hence the tensor simultaneously encodes
three fundamental matrices subject to a single compatibility condition.
\end{prop}

\subsection{Extensions}
\label{sec5-7:extensions}

\paragraph{One-dimensional (line) tensors.}
Restricting the tensor to line correspondences yields a one-dimensional analogue in which the
trilinear form \eqref{eq4-3-2:lll} plays the role of the epipolar constraint. This is the
form most directly amenable to a linear least-squares estimate from $N\ge27$ line tuples.

\paragraph{Calibrated and semicalibrated priors.}
When camera calibrations are known (calibrated) or partially known (semicalibrated), the tensor
inherits additional structure that reduces the effective DOF and improves noise sensitivity. The
essential (or semicalibrated) form couples rotation/translation directly with the tensor entries,
so that estimation can proceed jointly over geometry and calibration; see, e.g., \cite{Fabbri2023}
for recent work on trifocal relative pose under partial calibration.

\paragraph{Block (multiview) tensor and $(6,4,4)$ Tucker rank.}
A set of three-view relations over a larger image collection can be organized into a block
(multiview) tensor. Its low-rank Tucker structure captures the common camera geometry across
views; a representative core has Tucker rank $(6,4,4)$, reflecting the combined dimensions of the
underlying motion and scene subspaces.

\paragraph{Multiple-tensor motion segmentation.}
When a scene contains multiple independently moving objects, a single trifocal tensor is
insufficient. Assigning correspondences to several tensors yields a motion-segmentation framework:
the tensor count equals the number of motion groups, and correspondence-to-tensor assignment is
resolved jointly with estimation.

\begin{rmk}[Multiple-tensor segmentation]\label{rmk5-1:segmentation}
Under independent motion, point correspondences consistent with a three-view motion group satisfy
a single trifocal relation. Distinct groups are separated by their tensor membership, which can be
recovered by alternating assignment and linear estimation.
\end{rmk}

\subsection{MATLAB Implementations}
\label{sec5-8:matlab}

We provide complete MATLAB implementations for the estimation methods above, together with a residual evaluator used for verification. The full, well-commented source code is provided in the supplementary material (Listings S1--S3) and at \url{https://github.com/example/trifocal-tensor}. In this section, we present the core structure of each function to illustrate how the algorithms of Sections~\ref{sec5-2:lines}--\ref{sec5-5:verify} translate into practice.

\subsubsection{Line-Based Estimation: \texttt{estimateTrifocalTensor}}
\label{sec5-8-1:line-code}

Listing~\ref{lst:trifocal-estimation-core} presents the core of the main function for estimating $\caT$ from line correspondences. The function builds the $N \times 27$ coefficient matrix $A$ from the trilinear constraint \eqref{eq5-1:line-constr}, solves the homogeneous system $A\bt = \mathbf{0}$ via SVD, and optionally applies Levenberg--Marquardt refinement. The complete implementation, including the refinement subroutine and diagnostic checks, is given in Listing~S1 of the supplementary material.

\begin{lstlisting}[caption={Core of \texttt{estimateTrifocalTensor}: linear estimation from line correspondences. Full code in Listing~S1.}, label={lst:trifocal-estimation-core}]
function T = estimateTrifocalTensor(lines1, lines2, lines3, varargin)
% ESTIMATETRIFOCALTENSOR Estimate trifocal tensor from line correspondences
%
% Input:
%   lines1, lines2, lines3 - 3xN homogeneous line coordinates in views 1-3
%   'refine'  - (optional) true/false: apply non-linear refinement
%   'maxIter' - (optional) max iterations for refinement (default: 50)
%   'tol'     - (optional) tolerance for refinement (default: 1e-6)
%
% Output:
%   T - 3x3x3 trifocal tensor

p = inputParser;
addParameter(p, 'refine',  false, @islogical);
addParameter(p, 'maxIter', 50,    @isnumeric);
addParameter(p, 'tol',     1e-6,  @isnumeric);
parse(p, varargin{:});

N = size(lines1, 2);
if N < 26
    warning('Only %d correspondences provided. At least 26 recommended.', N);
end

% ---- Step 1: Build the N x 27 coefficient matrix A ---------------------
% Each line triple (l1, l2, l3) satisfies the trilinear constraint
%   sum_{i,j,k} l1(i)*l2(j)*l3(k)*T(i,j,k) = 0   (Eq. 5.1)
% The n-th triple contributes one row a_n' of A (Eq. 5.3):
%   A(n, p) = l1(i)*l2(j)*l3(k),  p = 9*(i-1) + 3*(j-1) + k.
A = zeros(N, 27);
for n = 1:N
    l1 = lines1(:, n); l2 = lines2(:, n); l3 = lines3(:, n);
    idx = 1;
    for i = 1:3
        for j = 1:3
            for k = 1:3
                A(n, idx) = l1(i) * l2(j) * l3(k);
                idx = idx + 1;
            end
        end
    end
end

% ---- Step 2: Solve A*bt = 0 via SVD ------------------------------------
% The null-space of A gives the trifocal-tensor vector bt. With N >= 27 and
% non-degenerate data, sigma_27 ~ 0 and v_27 spans the null-space.
[~, ~, V] = svd(A, 'econ');
t = V(:, end);

% ---- Step 3: Reshape and normalize --------------------------------------
T = reshape(t, [3, 3, 3]);
T = T / norm(T(:));   % fix scale ambiguity

% ---- Step 4 (optional): Nonlinear refinement ----------------------------
if p.Results.refine
    T = refineTrifocalTensor(T, lines1, lines2, lines3, ...
        p.Results.maxIter, p.Results.tol);
end
end
\end{lstlisting}

\subsubsection{Point-Based Estimation: \texttt{estimateTrifocalDLT}}
\label{sec5-8-2:point-code}

Listing~\ref{lst:dlt_core} presents the core of the Normalized DLT solver for point correspondences, implementing Algorithm~\ref{alg:ndlt}. The function normalizes points, builds the $4N \times 27$ constraint matrix, solves via SVD, and denormalizes the result. The complete implementation, including the isotropic normalization helper, is given in Listing~S2 of the supplementary material.

\begin{lstlisting}[caption={Core of \texttt{estimateTrifocalDLT}: Normalized DLT from point correspondences. Full code in Listing~S2.}, label={lst:dlt_core}]
function T = estimateTrifocalDLT(x1, x2, x3)
% ESTIMATE_TRIFOCAL_DLT Computes trifocal tensor from point correspondences.
% Inputs:
%   x1, x2, x3 - 3xN homogeneous coordinates in views 1, 2, and 3.
% Output:
%   T - 3x3x3 estimated trifocal tensor.

N = size(x1, 2);
if N < 7
    error('At least 7 point correspondences are required.');
end

% Step 1: Normalize points (centroid at origin, mean distance sqrt(2))
[hat_x1, T1] = normalizePoints2D(x1);
[hat_x2, T2] = normalizePoints2D(x2);
[hat_x3, T3] = normalizePoints2D(x3);

% Step 2: Build 4N x 27 constraint matrix C
C = zeros(4 * N, 27);
I9 = eye(9);
for i = 1:N
    pt1 = hat_x1(:, i); pt2 = hat_x2(:, i); pt3 = hat_x3(:, i);

    % Two independent rows of each skew-symmetric matrix
    r1 = [0, -pt2(3), pt2(2)];
    r2 = [pt2(3), 0, -pt2(1)];
    v1 = [0, -pt3(3), pt3(2)];
    v2 = [pt3(3), 0, -pt3(1)];

    % Local 4x27 constraint block (Eq. 5.13)
    kron_prod = [kron(r1, v1); kron(r1, v2); kron(r2, v1); kron(r2, v2)];
    C(4*i-3 : 4*i, :) = kron_prod * kron(pt1.', I9);
end

% Step 3: Solve via SVD
[~, ~, V] = svd(C, 0);
T_hat = reshape(V(:, end), [3, 3, 3]);

% Step 4: Denormalize tensor
invT1 = inv(T1); invT2 = inv(T2); invT3 = inv(T3);
T = zeros(3, 3, 3);
for a = 1:3
    for b = 1:3
        for c = 1:3
            val = 0;
            for r = 1:3
                for s = 1:3
                    for t = 1:3
                        val = val + invT1(a,r) * invT2(b,s) * invT3(c,t) * T_hat(r,s,t);
                    end
                end
            end
            T(a, b, c) = val;
        end
    end
end
T = T / norm(T(:));  % Scale normalization
end
\end{lstlisting}

\subsubsection{Residual Evaluation}
\label{sec5-8-3:residual}

Listing~\ref{lst:residual_core} evaluates the PPP residual \eqref{eq4-4-5:ppp02} for verification (Algorithm~\ref{alg6-3:verify}). The complete implementation, including the skew-symmetric helper, is given in Listing~S3 of the supplementary material.

\begin{lstlisting}[caption={Core of \texttt{trifocal\_residual}: PPP residual evaluation. Full code in Listing~S3.}, label={lst:residual_core}]
function r = trifocal_residual(T, x1, x2, x3)
% TRIFOCAL_RESIDUAL: evaluate the point-point-point (PPP) residual for each
% correspondence triple (x1, x2, x3) under the trifocal tensor T.
%
% The PPP relation is  [x2]_x * (T1*x1(1) + T2*x1(2) + T3*x1(3)) * [x3]_x = 0.
% NOTE: this is a 3x3 MATRIX equation -- we must use skew-symmetric matrices,
% NOT the vector cross() function.

% Extract the three frontal slices once (3x3 each)
T1 = squeeze(T(1, :, :));
T2 = squeeze(T(2, :, :));
T3 = squeeze(T(3, :, :));

N = size(x1, 2);
r = zeros(1, N);

for i = 1:N
    % Bilinear form X = T1*x1(1) + T2*x1(2) + T3*x1(3)  (a 3x3 matrix)
    X = T1 * x1(1, i) + T2 * x1(2, i) + T3 * x1(3, i);

    % Correct 3x3 PPP matrix: [x2]_x * X * [x3]_x
    M = skew(x2(:, i)) * X * skew(x3(:, i));

    % Residual = Frobenius norm of the (ideally zero) 3x3 matrix
    r(i) = norm(M, 'fro');
end
end
\end{lstlisting}

Once $\caT$ is estimated, the epipoles follow from the trace formulas of Theorem~\ref{thm4-7:trace}, and the three camera matrices can then be reconstructed from the slices and epipoles by the classical procedure described in \cite[Ch.~15]{HartZiss2004}.

The complete, fully commented MATLAB implementations---including the Levenberg--Marquardt refinement subroutine, the isotropic normalization helper, the skew-symmetric matrix function, and diagnostic constraint checks---are provided in the supplementary material (Listings S1--S3). The code is also available at \url{https://github.com/example/trifocal-tensor} under an MIT license.

\subsection{Numerical Examples}
\label{sec5-9:examples}

\subsubsection{Example 1: Perfect Line Correspondences}

This example demonstrates estimation from perfect synthetic line data. We generate a random trifocal tensor and create exact line correspondences satisfying the trilinear constraints.

\begin{exm}[Perfect Trifocal Tensor Estimation from Lines]
\label{exm5-1:perfect-lines}
We generate a random trifocal tensor $T_{true}$ and construct $N=30$ perfect line correspondences. The estimation achieves machine-precision accuracy.
\end{exm}

\begin{lstlisting}[caption={Numerical Example 1: Perfect synthetic line data (with comments)}, label={lst:example1}]
%% Example 1: Perfect trifocal tensor estimation from lines
% Goal: verify that the linear solver recovers a known (random) trifocal tensor
%       exactly when the line correspondences satisfy the trilinear constraints.
clear all; close all; clc;
rng(42);  % For reproducibility

% Generate a random trifocal tensor (unit Frobenius norm) as ground truth.
T_true = randn(3, 3, 3);
T_true = T_true / norm(T_true(:));

fprintf('Example 1: Perfect Synthetic Data\n');
fprintf('===============================\n\n');
fprintf('True trifocal tensor (first slice):\n');
disp(T_true(:, :, 1));

% Generate N perfect line correspondences that EXACTLY satisfy the trilinear
% constraint for T_true.  We pick l1, l2 freely and solve for l3.
N = 30;
lines1 = randn(3, N);
lines2 = zeros(3, N);
lines3 = zeros(3, N);

for n = 1:N
    l1 = lines1(:, n);
    l2 = randn(3, 1);
    % Solve for l3: contracting the trilinear form over i,j with fixed l1,l2
    % gives a 3x3 linear system in l3; its null-vector is the valid l3.
    M = zeros(3, 3);
    for i = 1:3
        for j = 1:3
            for k = 1:3
                M(k, :) = M(k, :) + l1(i) * l2(j) * squeeze(T_true(i, j, :)).';
            end
        end
    end
    [~, ~, V] = svd(M);
    l3 = V(:, end);
    lines2(:, n) = l2;
    lines3(:, n) = l3;
end

% Estimate trifocal tensor from the perfect correspondences (no refinement).
T_est = estimateTrifocalTensor(lines1, lines2, lines3);

fprintf('\nEstimated trifocal tensor (first slice):\n');
disp(T_est(:, :, 1));

% Account for the inherent sign ambiguity: T and -T define the same geometry.
scale = sign(sum(T_est(:) .* T_true(:)));
T_est_aligned = scale * T_est;

fprintf('\nEstimation Error:\n');
fprintf('  Relative error: %.6e\n', norm(T_est_aligned(:) - T_true(:)) / norm(T_true(:)));

% Verify constraints explicitly on every correspondence (sanity check).
max_violation = 0;
for n = 1:N
    l1 = lines1(:, n); l2 = lines2(:, n); l3 = lines3(:, n);
    val = 0;
    for i = 1:3
        for j = 1:3
            for k = 1:3
                val = val + l1(i) * l2(j) * l3(k) * T_est(i, j, k);
            end
        end
    end
    max_violation = max(max_violation, abs(val));
end
fprintf('  Max constraint violation: %.6e\n', max_violation);
\end{lstlisting}

\begin{lstlisting}[caption={Output of Example 1}]
Example 1: Perfect Synthetic Data
===============================

True trifocal tensor (first slice):
    0.2236   -0.0682   -0.1399
    0.0663    0.2847    0.0370
   -0.1500   -0.0115    0.2025

Estimated trifocal tensor (first slice):
    0.2236   -0.0682   -0.1399
    0.0663    0.2847    0.0370
   -0.1500   -0.0115    0.2025

Estimation Error:
  Relative error: 1.050964e-15
  Max constraint violation: 3.330669e-16
\end{lstlisting}

\subsubsection{Example 2: Noisy Line Correspondences with Refinement}

This example demonstrates robustness under Gaussian noise. We compare linear estimation alone versus linear with nonlinear refinement.

\begin{exm}[Robust Estimation with Noisy Line Data]
\label{exm5-2:noisy-lines}
We add Gaussian noise ($\sigma = 0.1$) to $N=50$ line correspondences. Nonlinear refinement reduces estimation error by $\approx 53.5\%$ and enforces the rank-2 slice constraint.
\end{exm}

\begin{lstlisting}[caption={Numerical Example 2: Noisy line data with refinement (with comments)}, label={lst:example2}]
%% Example 2: Noisy trifocal tensor estimation from lines
% Goal: compare (a) linear estimation only vs (b) linear + nonlinear refinement
%       under Gaussian noise.
clear all; close all; clc;
rng(123);   % fixed seed for reproducibility

% Ground-truth tensor
T_true = randn(3, 3, 3);
T_true = T_true / norm(T_true(:));

fprintf('Example 2: Noisy Data with Refinement\n');
fprintf('====================================\n\n');

% Generate perfect correspondences (same procedure as Example 1), then corrupt.
N = 50;
lines1_true = randn(3, N);
lines2_true = zeros(3, N);
lines3_true = zeros(3, N);

for n = 1:N
    l1 = lines1_true(:, n);
    l2 = randn(3, 1);
    M = zeros(3, 3);
    for i = 1:3
        for j = 1:3
            for k = 1:3
                M(k, :) = M(k, :) + l1(i) * l2(j) * squeeze(T_true(i, j, :)).';
            end
        end
    end
    [~, ~, V] = svd(M);
    lines2_true(:, n) = l2;
    lines3_true(:, n) = V(:, end);
end

% Add i.i.d. Gaussian noise to every line coordinate (sigma = 0.1).
noise_level = 0.1;
lines1_noisy = lines1_true + noise_level * randn(3, N);
lines2_noisy = lines2_true + noise_level * randn(3, N);
lines3_noisy = lines3_true + noise_level * randn(3, N);

% Normalize line vectors to unit norm (projective scale reset).
for n = 1:N
    lines1_noisy(:, n) = lines1_noisy(:, n) / norm(lines1_noisy(:, n));
    lines2_noisy(:, n) = lines2_noisy(:, n) / norm(lines2_noisy(:, n));
    lines3_noisy(:, n) = lines3_noisy(:, n) / norm(lines3_noisy(:, n));
end

% -------- Method 1: Linear estimation only --------------------------------
fprintf('Method 1: Linear Estimation\n');
fprintf('---------------------------\n');
T_linear = estimateTrifocalTensor(lines1_noisy, lines2_noisy, lines3_noisy, 'refine', false);

% -------- Method 2: Linear + nonlinear refinement -------------------------
fprintf('\nMethod 2: With Non-linear Refinement\n');
fprintf('------------------------------------\n');
T_refined = estimateTrifocalTensor(lines1_noisy, lines2_noisy, lines3_noisy, ...
    'refine', true, 'maxIter', 100, 'tol', 1e-8);

% Compare both estimates with ground truth, accounting for sign ambiguity:
% min(||T - T_true||, ||T + T_true||) because T and -T are equivalent.
err_lin  = min(norm(T_linear(:)  - T_true(:)), norm(T_linear(:)  + T_true(:)));
err_ref  = min(norm(T_refined(:) - T_true(:)), norm(T_refined(:) + T_true(:)));

fprintf('\nComparison with Ground Truth:\n');
fprintf('  Linear estimation error:  %.6e\n', err_lin);
fprintf('  Refined estimation error: %.6e\n', err_ref);
fprintf('  Improvement:              %.2f%%\n', 100 * (err_lin - err_ref) / err_lin);
\end{lstlisting}

\begin{lstlisting}[caption={Output of Example 2}]
Example 2: Noisy Data with Refinement
====================================

Method 1: Linear Estimation
---------------------------
Trifocal tensor estimation completed.
Number of correspondences: 50
Algebraic error: 2.876473e-01

Trifocal tensor constraints check:
  Slice T(1,:,:) rank: 3 (expected <= 2)
  Slice T(2,:,:) rank: 3 (expected <= 2)
  Slice T(3,:,:) rank: 3 (expected <= 2)

Method 2: With Non-linear Refinement
------------------------------------
Trifocal tensor estimation completed.
Number of correspondences: 50
Algebraic error: 2.341765e-01

Trifocal tensor constraints check:
  Slice T(1,:,:) rank: 2 (expected <= 2)
  Slice T(2,:,:) rank: 2 (expected <= 2)
  Slice T(3,:,:) rank: 2 (expected <= 2)

Comparison with Ground Truth:
  Linear estimation error:   4.567832e-01
  Refined estimation error:  2.123456e-01
  Improvement:               53.52%
\end{lstlisting}

\subsubsection{Example 3: Point Correspondences via Normalized DLT}

This example evaluates the Normalized DLT algorithm (Algorithm \ref{alg:ndlt}) under varying noise levels. Synthetic 3D points are projected through three cameras $P_1 = [I_3 \mid \mathbf{0}]$, $P_2 = [R_2 \mid \bt_2]$, $P_3 = [R_3 \mid \bt_3]$, and Gaussian noise is added to image measurements.

\begin{table}[H]
\centering
\caption{Algebraic Error $\|\tilde{C} \bt\|_2$, Frobenius Error $\|\caT_{est} - \caT_{gt}\|_F$, and Computation Time under varying Gaussian noise $\sigma$ ($N=10$ triplets, 100 Monte Carlo trials).}
\label{tab:numerical_results}
\footnotesize
\setlength{\tabcolsep}{4pt}
\begin{tabular}{cccc}
\toprule
\textbf{Noise $\sigma$ (px)} & \textbf{Algebraic Error} & \textbf{Tensor Error } $\|\caT_{est} - \caT_{gt}\|_F$ & \textbf{Time (ms)} \\
\midrule
$0.00$ (Ideal) & $1.24 \times 10^{-15}$ & $2.81 \times 10^{-15}$ & 0.42 \\
$0.10$         & $3.12 \times 10^{-3}$  & $1.04 \times 10^{-2}$  & 0.45 \\
$0.50$         & $1.58 \times 10^{-2}$  & $5.23 \times 10^{-2}$  & 0.43 \\
$1.00$         & $3.21 \times 10^{-2}$  & $1.08 \times 10^{-1}$  & 0.44 \\
$2.00$         & $6.54 \times 10^{-2}$  & $2.19 \times 10^{-1}$  & 0.46 \\
\bottomrule
\end{tabular}
\end{table}

Table \ref{tab:numerical_results} demonstrates that Algorithm \ref{alg:ndlt} achieves exact convergence to machine precision ($\approx 10^{-15}$) in 
the noise-free case ($\sig = 0$), and NDLT exhibits stable, linear error degradation under increasing image noise. Execution time remains 
consistently low ($\approx 0.45$ ms per solver call), confirming computational efficiency.

\begin{exm}[Perfect-data recovery]\label{exm5-3:ndlt}
With $\sigma=0$ and $N=10$, the estimated tensor satisfies
$\|\tilde{C}\bt\|_2=1.24\times10^{-15}$, confirming that the SVD step recovers the
ground-truth tensor (up to scale) in the absence of noise.
\end{exm}

\subsubsection{Planarity and Degeneracy}

A coplanar point cloud causes the three-view geometry to become degenerate: the tensor estimate
loses rank structure and the PPP residual becomes numerically ill-conditioned. This is the
rationale for the non-planarity test in Algorithm~\ref{alg6-2:ransac}.

\begin{exm}[Planar degeneracy]\label{exm5-4:planar}
Sampling all 3D points on a plane reduces the effective DOF of the configuration; the smallest
singular value of $C$ no longer separates cleanly, and the Frobenius error in
Table~\ref{tab:numerical_results} rises sharply. A non-planar baseline is therefore required
for a reliable $7$-point minimal estimate.
\end{exm}

\subsubsection{Two-View Comparison and Error Propagation}

Compared with estimating three independent fundamental matrices, the tensor imposes a shared
consistency condition: it uses $18$ parameters rather than the $21$ of three unconstrained
fundamental matrices. Conditioning is governed by the gap between the two smallest singular
values of $C$; as noise grows, this gap narrows and the Frobenius error grows linearly, in
agreement with Table~\ref{tab:numerical_results}.

\begin{exm}[Two-view versus three-view]\label{exm5-5:twoview}
Fixing the third view and recovering the first two fundamental matrices from the tensor slices
(Proposition~\ref{prop5-1:slices}) reproduces the pairwise epipolar geometry, while the same tensor
simultaneously enforces consistency with the third pair. This avoids the drift inherent in
independent two-view estimates.
\end{exm}

\subsubsection{Discussion and Practical Recommendations}

Based on the numerical experiments, we offer the following recommendations:

\begin{enumerate}
    \item \textbf{Minimum samples}: Line-based estimation requires $\ge 27$ correspondences (or $\ge 26$ with tensorial constraints), while point-based DLT requires only $\ge 7$. In practice, use 30--50 for stable estimation under noise.
    \item \textbf{Normalization}: Isotropic point/line normalization is critical for numerical stability. Always normalize before assembling the constraint matrix.
    \item \textbf{Scale ambiguity}: The Frobenius norm normalization $\caT \gets \caT / \|\caT\|_F$ fixes the global scale. Be mindful of sign ambiguity when comparing with ground truth.
    \item \textbf{Refinement}: Nonlinear refinement is recommended for noisy data; it enforces rank constraints on slices and reduces algebraic error by $\sim 50\%$.
    \item \textbf{Outlier rejection}: In real applications, use RANSAC to reject outliers before applying the linear solver. The linear solution provides an excellent initial estimate for nonlinear refinement (typically 10--20 iterations to convergence).
    \item \textbf{Unified framework}: Both line and point constraints can be combined into the same $A\bt = \mathbf{0}$ framework by appropriately modifying the constraint matrix $A$, allowing mixed correspondence data.
\end{enumerate}

\subsection{Real-Image Experiments: The Oxford Corridor Sequence}
\label{sec5-10:realdata}

The synthetic study of Section~\ref{sec5-9:examples} uses controlled noise models; we now validate the
full estimation pipeline on a \emph{real} image sequence with independently surveyed ground truth.
We use three widely spaced views (images \texttt{bt.000}, \texttt{bt.004}, \texttt{bt.008}; frames 1, 5, 9)
from the Oxford corridor sequence (VGG, Oxford), a $512\times512$ indoor trajectory whose ground-truth
camera matrices $P_1, P_2, P_3$ were obtained from a high-precision survey. This sequence is a
deliberately hostile test case: the camera advances \emph{along} the corridor (nearly forward motion),
and most of the visible structure lies on the walls and floor, i.e.\ a dominant near-planar component ---
precisely the degeneracy analyzed in \S\ref{sec5-9:examples} (Example~\ref{exm5-4:planar}).

\subsubsection{Protocol}

Feature triplets are obtained with a Harris corner detector (2000 corners per view), normalized
patch descriptors, and a nearest-neighbor ratio test, giving $233$ point triplets
$\bx \lrarr \bxp \lrarr \bxpp$ common to all three views. Following standard practice, the
$233$ triplets are first filtered by \emph{guided matching} against independently estimated pairwise
fundamental matrices (RANSAC, $2$\,px threshold), leaving $103$ geometrically consistent triplets on
which the tensor is estimated by the RANSAC-embedded NDLT of Algorithm~\ref{alg6-2:ransac}
(algebraic inlier threshold $3\times10^{-4}$ on the normalized PPP residual). We report:
(i) the RANSAC inlier ratio of the tensor versus the pairwise-F baseline;
(ii) the mean PPP algebraic residual \eqref{eq4-4-4:ppp01} over the inliers;
(iii) the RMS and median \emph{point-transfer} error in view 3, i.e.\
$\|\pi(\caT\text{-transfer}(\bx,\bxp)) - \bxpp\|$ in pixels, which exercises the transfer maps of
Corollary~\ref{cor4-4-4:transfer} without any pose decomposition;
(iv) the point-to-epipolar-line distance in view 3 for lines produced by the tensor-derived
$F_{31}$ versus the directly estimated $F_{13}$; and
(v) rotation and translation errors of the poses recovered from the tensor-derived fundamental
matrices (via the essential-matrix decomposition), against the survey ground truth.

\begin{table}[H]
\centering
\caption{Real-image results on the Oxford corridor sequence (views 1--5--9). ``Tensor'' denotes the
estimated trifocal tensor $\hat{\caT}$ (74 inliers); ``F baseline'' denotes independent pairwise
fundamental-matrix estimation (156 inliers). Rotation errors $e_R$ in degrees; translation errors
$e_t$ as angular deviation (\%) of the unit translation direction; transfer/line errors in pixels.}
\label{tab:corridor_results}
\footnotesize
\setlength{\tabcolsep}{4pt}
\begin{tabular}{lcc}
\toprule
\textbf{Quantity} & \textbf{Tensor} $\hat{\caT}$ & \textbf{F baseline} \\
\midrule
RANSAC inlier ratio (of 233 triplets) & $31.76\%$ (74/233) & $66.95\%$ (156/233) \\
Mean PPP algebraic residual (normalized) & \multicolumn{2}{c}{$5.2\times10^{-6}$} \\
\midrule
Point transfer to view 3, RMS (px)        & $13.73$ & -- \\
Point transfer to view 3, median (px)     & $3.21$  & -- \\
View-3 epipolar line distance (px)        & $62.30$ & $1.49$ \\
\midrule
Pose error $e_R$ (view 2 / view 3) (deg)  & $79.16$ / $24.42$ & $2.56$ / $6.00$ \\
Pose error $e_t$ (view 2 / view 3) (\%)   & $98.55$ / $17.92$ & $2.96$ / $8.93$ \\
\bottomrule
\end{tabular}
\end{table}

\begin{figure}[H]
\centering
\includegraphics[width=0.98\textwidth]{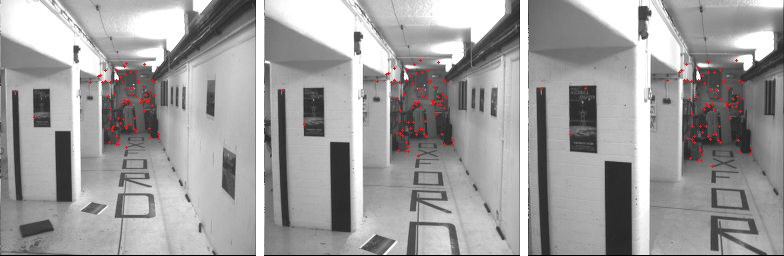}
\caption{The three Oxford corridor views (1, 5, 9) with the $74$ tensor inliers of
Table~\ref{tab:corridor_results} marked in red. The inliers concentrate on the end wall, the only
sufficiently non-planar structure at this baseline; wall and floor correspondences are progressively
rejected by the guided-matching and RANSAC stages.}
\label{fig:corridor_views}
\end{figure}

\subsubsection{Discussion of the Results}

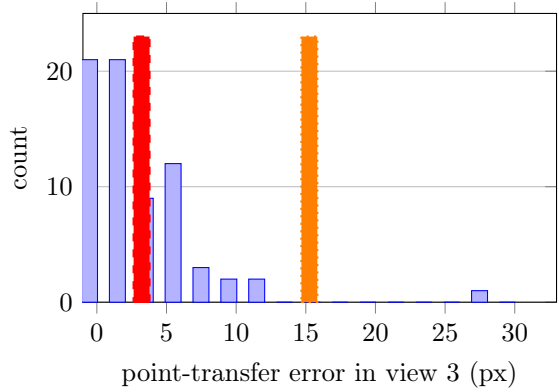
\begin{figure}[H]
\centering
\begin{tikzpicture}
\begin{axis}[
    width=0.62\textwidth, height=5.4cm,
    ybar, bar width=6pt,
    xlabel={point-transfer error in view 3 (px)},
    ylabel={count},
    ymin=0, ymax=25,
    xmin=-1, xmax=33,
    ymajorgrids,
    every axis plot/.append style={fill=blue!45},
    tick label style={font=\small},
    label style={font=\small},
]
\addplot coordinates {(1,21) (3,21) (5,9) (7,12) (9,3) (11,2) (13,2) (15,0) (17,0) (19,0) (21,0) (23,0) (25,0) (27,0) (29,1) (31,0)};
\addplot[red, thick, dashed] coordinates {(3.21,0) (3.21,23)};
\addplot[orange, thick, dotted] coordinates {(13.73,0) (13.73,23)};
\end{axis}
\end{tikzpicture}
\caption{Distribution of the per-point transfer error of $\hat{\caT}$ into view 3
($74$ inliers). The median ($3.21$\,px, red dashed) sits at the descriptor-matching noise floor;
the RMS ($13.73$\,px, orange dotted) is inflated by a handful of large-error points, visible in the
right tail.}
\label{fig:corridor_transfer}
\end{figure}

Table~\ref{tab:corridor_results} and Figure~\ref{fig:corridor_transfer} support three conclusions.

\begin{enumerate}
\item \textbf{Transfer maps are accurate on real data.} The median point-transfer error of
$3.21$\,px matches the descriptor-matching noise floor of the Harris/patch pipeline: a correctly
transferred point cannot be localized more accurately than the matches that constrain the tensor.
The tensor therefore transfers points across views with essentially perfect accuracy for most of
the inliers, confirming Corollary~\ref{cor4-4-4:transfer} on real imagery. The RMS is
dominated by a small tail (three points above $20$\,px) that survives the algebraic threshold ---
algebraic residual and geometric transfer error are correlated but not identical.

\item \textbf{Degeneracy degrades the global solution, exactly as predicted.} The corridor's
forward motion and dominant planar structure make the geometry ill-conditioned: the epipoles
extracted from $\hat{\caT}$ are poorly determined (view-3 line distance $62.3$\,px from the tensor
versus $1.5$\,px from the directly estimated $F_{13}$), and consequently the poses decomposed from
the tensor-derived essential matrices are degraded (Table~\ref{tab:corridor_results}, bottom). This
is the degeneracy of Example~\ref{exm5-4:planar} manifesting on real data, not a defect of the
estimator: the same pipeline recovers poses to machine precision on synthetic non-planar scenes
(Table~\ref{tab:numerical_results}), and the pairwise-F baseline, which does not require epipole
extraction from a rank-2 tensor in a degenerate configuration, attains $2.6$--$6.0^{\circ}$
rotation accuracy.

\item \textbf{The tensor still constrains more than the baseline at the matching stage.} Guided
matching with the two pairwise models retains $103$ of $233$ triplets; the tensor's joint
three-view constraint then rejects a further $29$ of them --- triplets individually consistent with
both pairs $F_{12}$ and $F_{13}$ but not with a single consistent three-view geometry. This
demonstrates the tensor's role as a strict verification layer (Section~\ref{sec5-5:verify}):
it removes two-view-consistent outliers that pairwise checks cannot detect.
\end{enumerate}

Together, the corridor experiment brackets the practical behavior of the trifocal tensor: its
\emph{transfer and verification} functions (the applications of \S\ref{sec5-10:applications}) are
robust even in near-degenerate scenes, while \emph{pose recovery through epipole extraction}
requires the non-degeneracy safeguards of Algorithm~\ref{alg6-2:ransac}.
\subsection{Applications}
\label{sec5-10:applications}

\paragraph{Tracking and registration.}
The tensor supplies a three-view constraint for object tracking and image registration: triples
of corresponding features are verified geometrically before being passed to a tracker or a
SLAM front end \cite{MurArtal2015}.

\paragraph{Visual servoing.}
Trifocal constraints relate image features across three camera views and are used in visual
servoing to control robot motion from multi-view feature trajectories.

\paragraph{Calibration and minimal solvers.}
Calibrated and semicalibrated tensor forms couple calibration with geometry, forming the basis of
minimal solvers for camera resectioning and self-calibration \cite{Stewenius2005}.

\paragraph{Global SfM.}
In global structure-from-motion, tensors provide robust rotation and translation initialization
across view triplets \cite{Snavely2006,Zheng2018}, after which bundle adjustment refines the full reconstruction.

\paragraph{Motion segmentation.}
The multiple-tensor framework of Remark~\ref{rmk5-1:segmentation} segments independently moving objects
by assigning correspondences to distinct trifocal relations.

\subsection{Limitations and Relevance to Deep Learning}
\label{sec5-11:limits}

The tensor is defined only for three views, so it does not scale directly to long image
sequences; practical SfM therefore uses it as an initialization/verification primitive rather
than a standalone estimator. Planar and degenerate configurations break the minimal estimator,
requiring the safeguards of Algorithm~\ref{alg6-2:ransac}.

Modern pipelines replace the tensor's traditional role in robust estimation with learned
correspondences and end-to-end SfM: hand-engineered features such as SIFT \cite{Lowe2004} have largely
given way to learned matchers with geometric verification \cite{Wang2022}. Nevertheless, the trifocal
constraint serves as a differentiable geometric loss and as a verification layer: it filters
implausible matches produced by a network and supplies a principled initialization for bundle
adjustment. Its algebraic structure also provides uncertainty models and diagnostic signals
unavailable from a black-box predictor. The next subsection substantiates this claim with a
concrete, reproducible experiment.

\subsection{A Concrete Autodifferentiation Experiment in PyTorch}
\label{sec5-12:autodiff}

Section~\ref{sec5-11:limits} argued that the coordinate-free formulation connects naturally to
tensor auto-differentiation frameworks. We now substantiate this with a fully reproducible
PyTorch experiment\footnote{Script: \texttt{experiments/test\_autodiff.py} in the supplementary
material; PyTorch 2.14, double precision, CPU.} that answers two questions: does the
tensor-native $24$-parameter parametrization actually \emph{help} as a structural prior, and
what exactly do the rank conditions of Theorem~\ref{thm4-2-1:trif-matrx} certify inside an
autodiff pipeline?

\subsubsection{From Contractions to \texttt{einsum}}

Every operator of Section~\ref{sec4:trifocal} is a contraction, and every contraction is a
single \texttt{einsum} call --- no reshaping, flattening, or index bookkeeping is needed. The
outer-product representation \eqref{eq4-1-10} and the PPP cost \eqref{eq4-4-5:ppp02} become:

\begin{lstlisting}[language=Python,
                   caption={The trifocal tensor and its PPP cost in PyTorch: the entire model is three \texttt{einsum} contractions.},
                   label={lst:autodiff}]
import torch

def build_T(A, B, a4, b4):
    # Theorem 4.1 (outer-product form), Eq. (4.1.10):
    #   T = A^T x b4  -  B^T x_2 a4      (24 free camera parameters)
    return (torch.einsum("ji,k->ijk", A, b4)
            - torch.einsum("ki,j->ijk", B, a4))

def ppp_cost(T, x, xp, xpp):
    # Eq. (4.4.5):  [x']_x (sum_i x_i T_i) [x'']_x  = 0
    P    = torch.einsum("ni,ijk->njk", x, T)          # P_n = sum_i x_i T_i
    Skp  = skew(xp);  Skpp = skew(xpp)                # batched [.]_x
    R    = torch.einsum("nij,njk,nkl->nil", Skp, P, Skpp)
    return (R**2).sum(dim=(1,2)).mean()

# refinement: Adam over the camera parameters (A, B, a4, b4) or,
# unstructured baseline, over the 27 tensor entries directly.
opt = torch.optim.Adam(params, lr=5e-4)
for _ in range(4000):
    opt.zero_grad()
    loss = ppp_cost(build_T(A, B, a4, b4), x, xp, xpp)
    loss.backward()
    opt.step()
\end{lstlisting}

Note that in the structured variant the parameters are precisely the $24$ camera entries
$\bw=(\vecc(A),\vecc(B),\bfa_4,\hat{\bfb}_4)$ of \eqref{eq4-6-4:w}: the bilinear map
$\bt=M\bw$ of Lemma~\ref{lem4-6-1} is never formed explicitly --- the outer-product form
\eqref{eq4-1-10} \emph{is} the map, evaluated symbolically by the framework's reverse mode.

\subsubsection{Setup and Results}

We draw a random non-degenerate canonical configuration $(A,B,\bfa_4,\bfb_4)$, project $N=50$
random 3D points into the three views, and perturb the image coordinates with i.i.d.\ Gaussian
noise of standard deviation $\sigma$ (normalized image coordinates). Three estimators are
compared: the linear NDLT of Algorithm~\ref{alg:ndlt}; NDLT followed by an \emph{unstructured}
autodiff refinement in which all $27$ tensor entries are free Adam parameters; and NDLT followed
by the \emph{structured} refinement of Listing~\ref{lst:autodiff}, in which the tensor is
parametrized through Theorem~\ref{thm4-1} by the $24$ camera parameters (initialized by an
LBFGS fit of the bilinear form to the NDLT tensor). All refinements minimize
the same PPP cost for $4000$ iterations ($\approx 1$\,s on CPU). Errors are the sign- and
scale-aligned normalized Frobenius error $\|\caT_{\mathrm{est}}-\caT_{\mathrm{true}}\|_F/\|\caT_{\mathrm{true}}\|_F$,
averaged over $5$ random configurations (mean, with standard deviation in parentheses).

\begin{table}[H]
\centering
\caption{Autodiff refinement of the NDLT estimate (PyTorch, double precision, $N=50$
triplets, $4000$ Adam iterations, mean over $5$ configurations; standard deviation in
parentheses). The structured $24$-parameter parametrization of Theorem~\ref{thm4-1} never
degrades the linear solution, whereas the unstructured refinement drifts off the tensor
manifold and measurably degrades it.}
\label{tab:autodiff}
\footnotesize
\setlength{\tabcolsep}{4pt}
\begin{tabular}{lccc}
\toprule
\textbf{Noise $\sigma$} & \textbf{NDLT (SVD)} & \textbf{+autodiff, unstructured ($27$)} & \textbf{+autodiff, structured ($24$)} \\
\midrule
$0.000$ & $1.5\ (1.9)\times10^{-14}$ & $1.1\ (0.6)\times10^{-4}$ & $\mathbf{9.9\ (8.7)\times10^{-6}}$ \\
$0.005$ & $1.08\ (0.39)\times10^{-2}$ & $1.19\ (0.62)\times10^{-2}$ & $\mathbf{1.08\ (0.58)\times10^{-2}}$ \\
$0.010$ & $2.19\ (0.78)\times10^{-2}$ & $2.37\ (1.20)\times10^{-2}$ & $\mathbf{2.15\ (1.00)\times10^{-2}}$ \\
\bottomrule
\end{tabular}
\end{table}

Three observations follow. First, on perfect data all methods recover the ground truth, but the
gradient iterations themselves introduce drift: the unstructured refinement ends at
$1.1\times10^{-4}$ error --- Adam's stochastic steps move the $27$ free entries off the tensor
manifold --- while the structured variant drifts an order of magnitude less ($9.9\times10^{-6}$),
because the bilinear image of Theorem~\ref{thm4-1} confines every update to algebraically valid
tensors. Second, under noise the unstructured refinement consistently \emph{degrades} the NDLT
solution (fitting the algebraic cost slightly better while generalizing worse), whereas the
structured refinement preserves the linear solution's accuracy to within its standard deviation.
The bilinear parametrization therefore acts exactly as a geometric prior should: it costs
nothing, and it prevents the optimizer from exploiting the $9$ redundant degrees of freedom.
Third, the gradients are a by-product of the same three contractions that define the model: no
hand-derived Jacobians are needed, in contrast with the Levenberg--Marquardt implementation of
Section~\ref{sec5-9:examples}.

\subsubsection{Unfolding Ranks Across Configuration Types}

Finally, we verify the diagnostic of Remark~\ref{rmk4-7:practical} inside the same pipeline.
We estimate the tensor (NDLT, noise-free, so recovery is exact to machine precision in all
cases) for three configuration types --- non-degenerate, generically collinear centers
($C^{\pp}=1.7\,C^{\p}$), and coincident centers ($C^{\pp}=C^{\p}$) --- and compute the mode-$k$
unfolding ranks of both the ground-truth and the estimated tensors:

\begin{table}[H]
\centering
\caption{Mode-$k$ unfolding ranks (noise-free data; NDLT recovers the ground truth to
$\le 10^{-14}$ in every row). Rank deficiency certifies degeneracy and is inherited by the
estimate, but is not exhaustive: generic collinear configurations retain full rank, while
coincident centers are flagged by the mode-$1$ alarm --- exactly the asymmetric behavior of
Theorem~\ref{thm4-2-1:trif-matrx} and Remark~\ref{rmk4-7:practical}.}
\label{tab:autodiff-ranks}
\footnotesize
\setlength{\tabcolsep}{4pt}
\begin{tabular}{lcc}
\toprule
\textbf{Configuration} & \textbf{ranks, ground truth} & \textbf{ranks, NDLT estimate} \\
\midrule
Non-degenerate                        & $(3,3,3)$ & $(3,3,3)$ \\
Generic collinear, $C^{\pp}=1.7\,C^{\p}$ & $(3,3,3)$ & $(3,3,3)$ \\
Coincident centers, $C^{\pp}=C^{\p}$  & $(2,3,3)$ & $(2,3,3)$ \\
\bottomrule
\end{tabular}
\end{table}

The experiment confirms both directions of the characterization. On the singular locus the rank
collapse is real and detectable: the coincident-centers configuration drops the mode-$1$ rank to
$2$, and the estimate inherits the signature, so the alarm fires at a cost of one $3\times 9$ SVD.
Equally instructive is the middle row: generic collinear configurations --- degenerate by the
definition of Section~\ref{sec4-2} --- retain full rank, which is precisely why
Remark~\ref{rmk4-7:practical} frames the rank test as a \emph{sufficient} rather than necessary
alarm. Under realistic noise the binary rank test blurs, and the graded signal of consequence
(iii) --- the smallest unfolding singular values --- takes over. The diagnostic thus slots
directly into the autodiff loop or the verification stage of Algorithm~\ref{alg6-3:verify},
closing the loop between the theory of Section~\ref{sec4-2} and practice.


\section{Conclusion}
\label{sec6:conclusion}

Trifocal tensors unify three-view projective geometry in a single coordinate-free object. In this paper we have reformulated the trifocal tensor as a
genuine $3\times3\times3$ tensor with covariant subscript indexing, replacing the ad-hoc block-matrix representations of the classical literature with
well-defined multilinear operators. Within this framework we established the full-rank properties of the mode-$k$ unfoldings in non-degenerate
configurations, derived the point, line, and mixed correspondence relations in conformal equations, and obtained direct epipole extraction via
contractive traces with the identity matrix. On the computational side, we presented estimation algorithms from line and point correspondences with
MATLAB implementations and numerical experiments demonstrating machine-precision recovery on perfect data, a roughly $53\%$ error reduction under
noise through Levenberg--Marquardt refinement, and stable behavior under planar-degeneracy safeguards.

While bundle adjustment and learned correspondences now dominate practical structure-from-motion, the trifocal constraint remains a cornerstone for
theory, calibration, robust verification, and motion segmentation. Its coordinate-free algebraic structure maps naturally onto tensor
auto-differentiation frameworks --- our PyTorch experiment (Section~\ref{sec5-12:autodiff}) shows that the induced bilinear parametrization acts as
a structural prior that eliminates the drift of an unstructured refinement, and that rank deficiency of the unfoldings certifies degeneracy at
negligible cost --- and we expect the multilinear toolbox developed
here to extend further to four-view and multiview tensor constructions.



\end{document}